\documentclass[a4paper,12pt]{article}

\usepackage[margin=3cm,footskip=1cm]{geometry}

\usepackage{amsmath,amssymb,amsthm,bm,mathtools,xspace,booktabs}
\usepackage{graphicx,color}

\numberwithin{equation}{section}

\theoremstyle{plain} 
\newtheorem{theorem}{Theorem}[section]
\newtheorem{lemma}[theorem]{Lemma}
\newtheorem{proposition}[theorem]{Proposition}
\newtheorem{definition}[theorem]{Definition}
\newtheorem{corollary}[theorem]{Corollary}
\newtheorem{assumption}[theorem]{Assumption}

\theoremstyle{definition} 

\newtheorem{remark}[theorem]{Remark}

\usepackage{authblk}

\let\oldabstract\abstract
\makeatletter
\renewcommand\abstract{%
  \providecommand\keywords{\par\medskip\noindent\textit{Keywords:}\xspace}
  \oldabstract\noindent\ignorespaces}
\makeatother

\usepackage{siunitx}

\usepackage[all]{onlyamsmath}

\usepackage{natbib}
\usepackage{algorithmic, algorithm}
\usepackage{subfig}
\makeatletter
\DeclareRobustCommand*\subref{\@ifstar\sf@@subref\sf@subref}
\makeatother

\newcommand{\bx}{\bm x}
\newcommand{\by}{\bm y}

\newcommand{\bv}{\bm v}
\newcommand{\be}{\bm e}
\newcommand{\bw}{\bm w}
\newcommand{\dee}{\text{\,d}}
\newcommand{\detproc}{\textrm{DPP}}
\newcommand{\lambdat}{\tilde{\lambda}}
\newcommand{\real}{\mathbb{R}}
\newcommand{\complex}{\mathbb{C}}
\newcommand{\conj}[1]{\overline{#1}}

\newcommand{\tr}{\textrm{tr}}

\newcommand{\spanspace}[1]{\textrm{span}_{\mathbb C}\{#1\}}
\newcommand{\var}{\ensuremath{\textrm{Var}}}
\newcommand{\mean}{\textrm{E}}
\newcommand{\prob}{\textrm{P}}

\newcommand{\fapp}{f^{\mathrm{app}}}
\newcommand{\Dapp}{D_{\mathrm{app}}}
\newcommand{\Capp}{C_{\mathrm{app}}}

\newcommand{\Cappstat}{C_{{\mathrm{app}},0}}
\newcommand{\Ct}{\tilde{C}}
\newcommand{\Ctapp}{\Ct_{{\mathrm{app}}}}
\newcommand{\Ctappstat}{\Ct_{{\mathrm{app}},0}}

\newcommand{\wm}{Whittle-Mat{\'e}rn}

\newcommand{\Cdom}{\ensuremath{C^{\text{dom}}}}
\newcommand{\Cdomstat}{\ensuremath{C_0^{\text{dom}}}}
\newcommand{\Xdom}{\ensuremath{X^{\text{dom}}}}
\newcommand{\Xapp}[1]{\ensuremath{X_{#1}^{\text{app}}}}
\newcommand{\Xper}[1]{\ensuremath{X_{#1}^{\text{per}}}}

\newcommand{\Yapp}[1]{\ensuremath{Y_{#1}^{\text{app}}}}
\newcommand{\rhomax}{\ensuremath{\rho_{\text{max}}}}
\newcommand{\unitbox}{\ensuremath{[-1/2,1/2]^d}}
\newcommand{\halfbox}{\ensuremath{[-1/4,1/4]^d}}

\newcommand{\conv}{\ensuremath{\star}}
\newcommand{\hc}{\ensuremath{h}}
\newcommand{\hcest}{\ensuremath{\hat{h}}}

\newcommand{\Z}{\mathbb{Z}}
\newcommand{\N}{\mathbb{N}}

\begin{document}

\title{Determinantal point process models\\ and
  statistical inference:\\ Extended version}

\author[1]{Fr\'{e}d\'{e}ric Lavancier}
\author[2]{Jesper M{\o}ller}
\author[2]{Ege Rubak\footnote{An alphabetical ordering has been used
    since all authors have made significant contributions to the paper.}}
\affil[1]{ Laboratoire de Math\'ematiques Jean Leray\\
University of Nantes, France\\ Frederic.Lavancier@univ-nantes.fr}
\affil[2]{Department of Mathematical Sciences\\
Aalborg University\\ jm@math.aau.dk, rubak@math.aau.dk}

\date{}

\maketitle

\begin{abstract}
  Statistical models and methods for determinantal point processes (DPPs) seem
  largely unexplored. We demonstrate that DPPs provide useful
  models for the description of spatial point pattern datasets where
  nearby points repel each other.
   Such data are usually modelled by Gibbs
  point processes, where the likelihood and moment
  expressions are intractable and simulations
  are time consuming. We exploit the appealing probabilistic properties of
  DPPs to develop parametric models, where the likelihood and
   moment expressions can be easily
  evaluated and realizations can be quickly simulated.
  We discuss how statistical inference is conducted using the likelihood
  or moment properties of DPP models, and we
  provide freely available software for simulation
  and statistical inference.

\end{abstract}

\section{Introduction}\label{sec:intro}

Spatial point process models where nearby points in
the process repel each other are often used for describing point pattern
datasets exhibiting regularity (in contrast to aggregated or clustered
point pattern
datasets). This paper studies statistical models and inference
procedures for 
determinantal point processes (DPPs) which constitute a particular
tractable class of repulsive
spatial point processes.

\subsection{Background and aim of the paper}\label{sec:aim}

DPPs are largely unexplored in statistics,
though they possess a number of very attractive properties and have
been studied in mathematical physics, combinatorics, and random matrix
theory even before the general notion was introduced in
\cite{Macchi:75}. They have been used to model fermions in quantum
mechanics, in classical Ginibre and circular unitary ensembles from
random matrix theory, for examples arising from non-intersecting
random walks and random spanning trees, and much more, see Section~2
in \cite{Soshnikov:00} and Section~4.3 in \cite{Hough:etal:09}.  They
can be defined on a locally compact space, where the two most
important cases are the
  $d$-dimensional Euclidean space $\real^d$ and a discrete state
space. Recently, DPPs have been used in machine learning
\citep{Kulesza:Taskar:12},
where the state space is finite (basically
a directory for statistical learning), and in wireless communication to model the locations of network nodes \citep{Leonardi13,Miyoshi:Shirai13}. In recent years,
DPPs have also been much studied in probability theory, see
\cite{Hough:etal:09} and the references therein.

In the present
 paper, we address several
statistical problems for DPPs defined on $\real^d$ (or a sub-region of
 $\real^d$).
Our main aims are:
\begin{enumerate}
\item[(i)] to provide a short and accessible survey for
statisticians on the definition,
existing conditions, moment properties, density expressions, and
simulation procedures for DPPs;
\item[(ii)] to clarify when stationary DPPs exist and to
develop parametric model classes for stationary DPPs (which later are
extended to inhomogeneous DPPs);
\item[(iii)]
to understand to
which extent DPPs can model repulsiveness (or regularity or
inhibition) and to demonstrate that DPPs provide useful flexible
  models for the description of repulsive spatial point processes; 
\item[(iv)] to construct useful approximations of certain
spectral-decompositions
appearing when dealing with likelihoods and simulations of DPPs;
\item[(v)] to discuss how statistical inference is conducted using the likelihood
  or moment properties of stationary as well as inhomogeneous DPP models;
\item[(vi)] to apply our methodology on real spatial point
  pattern datasets showing different degrees of repulsiveness;
\item[(vii)] to provide freely available software for simulation
  and statistical inference.
\end{enumerate}
While \cite{Hough:etal:09} provides an excellent and comprehensive survey of the
interest in probability theory on DPPs, our survey (item (i) above) is a less
technical exposition of DPPs which
provides the needed background material for our new contributions
(items (ii)-(vii) above).

\subsection{Repulsiveness and point pattern datasets}\label{sec:1.2}

Formal definitions of  repulsiveness in a spatial point process
will be given later in this paper, either in terms of second order
properties (the so-called pair correlation function and the
$K$-function, see Section~\ref{sec:defdet} and Appendix~\ref{sec:quantify})
or in terms of the Papangelou conditional intensity
(Appendix~\ref{sec:proofdensitythm}).
 The six point pattern datasets in
Figure~\ref{fig:data}
 will be fitted using parsimonious parametric
models of DPPs. These datasets have been
selected to illustrate the range of repulsiveness which can be
modelled by DPPs. They will be detailed and analysed in
Sections~\ref{sec:hom}-\ref{sec:inhom}, where Figure~\ref{fig:data}\subref{subfig:data:spanish}-\subref{subfig:data:termites}
will be modelled by stationary DPPs and
Figure~\ref{fig:data}\subref{subfig:data:Japanese}-\subref{subfig:data:mucosa}
by inhomogeneous DPPs. 

\begin{figure}[!htbp]%
\centering
\subfloat[][Locations of 69 Spanish towns in a 40 mile
by 40 mile region.]{
 \includegraphics[angle=0,width=.4\textwidth]{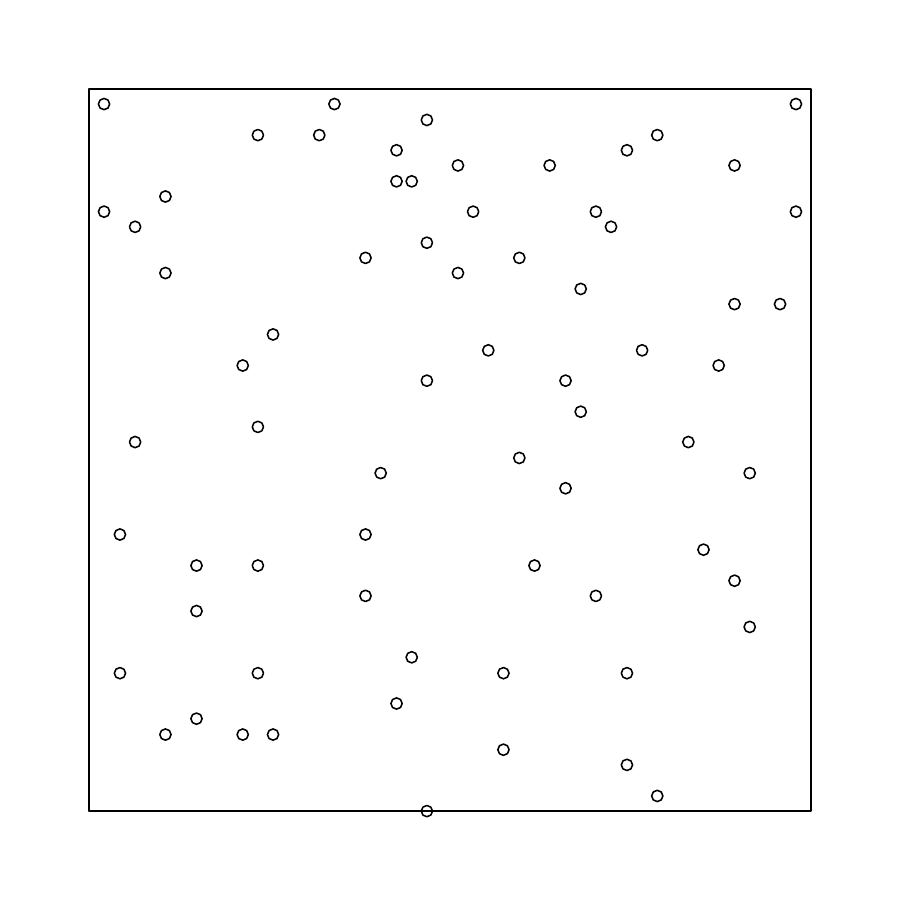}
 \label{subfig:data:spanish}
 }
 \quad
\subfloat[][Locations of 303 cells of two types in a
\SI{0.25}{\milli\metre} by \SI{0.25}{\milli\metre} region of a
histological section of the kidney of a hamster.]{
 \includegraphics[angle=0,width=.4\textwidth]{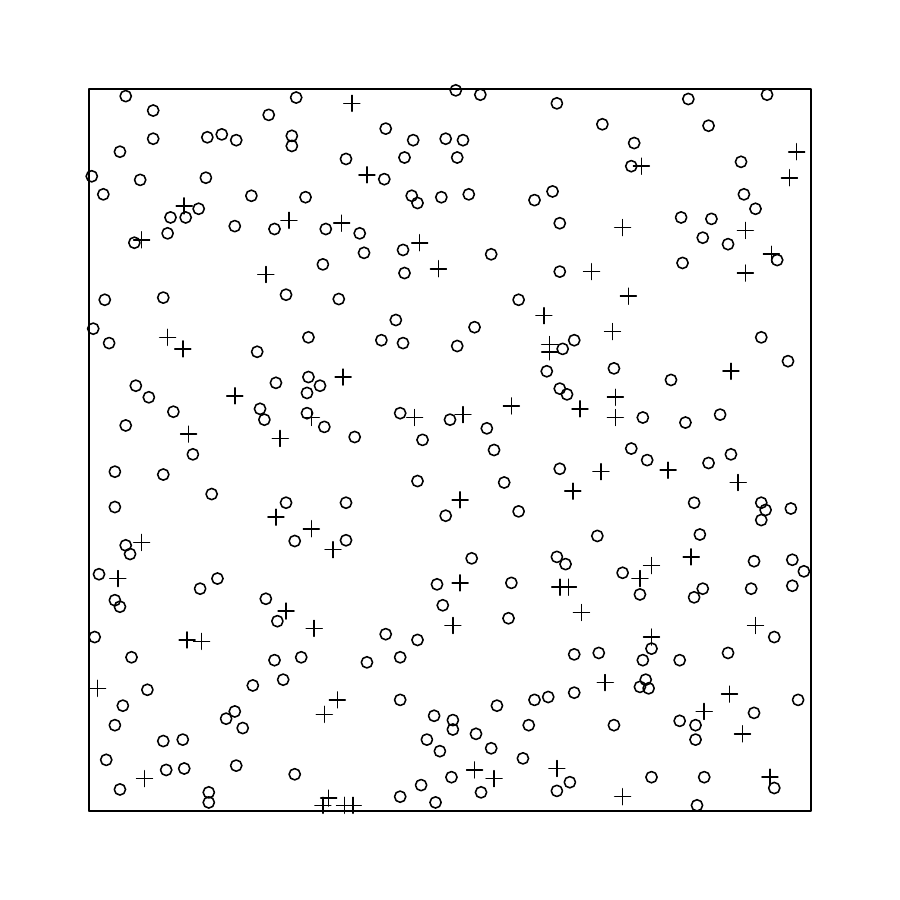}
 \label{subfig:data:hamster}
 }
 \quad
 \subfloat[][Locations of 244 trees of the species oak and beech in a
 \SI{80}{\metre} by \SI{80}{\metre} region.]{
 \includegraphics[angle=0,width=.4\textwidth]{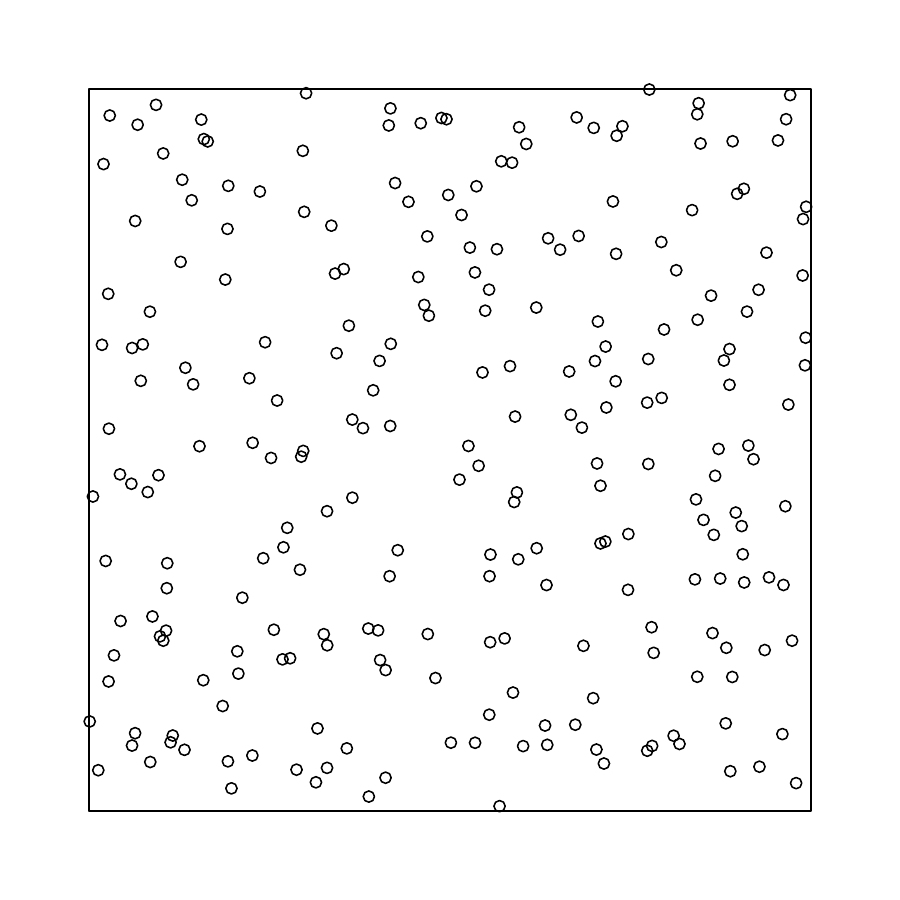}
 \label{subfig:data:oakbeech}
 }
 \quad
 \subfloat[][Locations of 48 termite mounds in a \SI{250}{\metre} by
 \SI{150}{\metre} region.]{
 \includegraphics[angle=0,width=.4\textwidth]{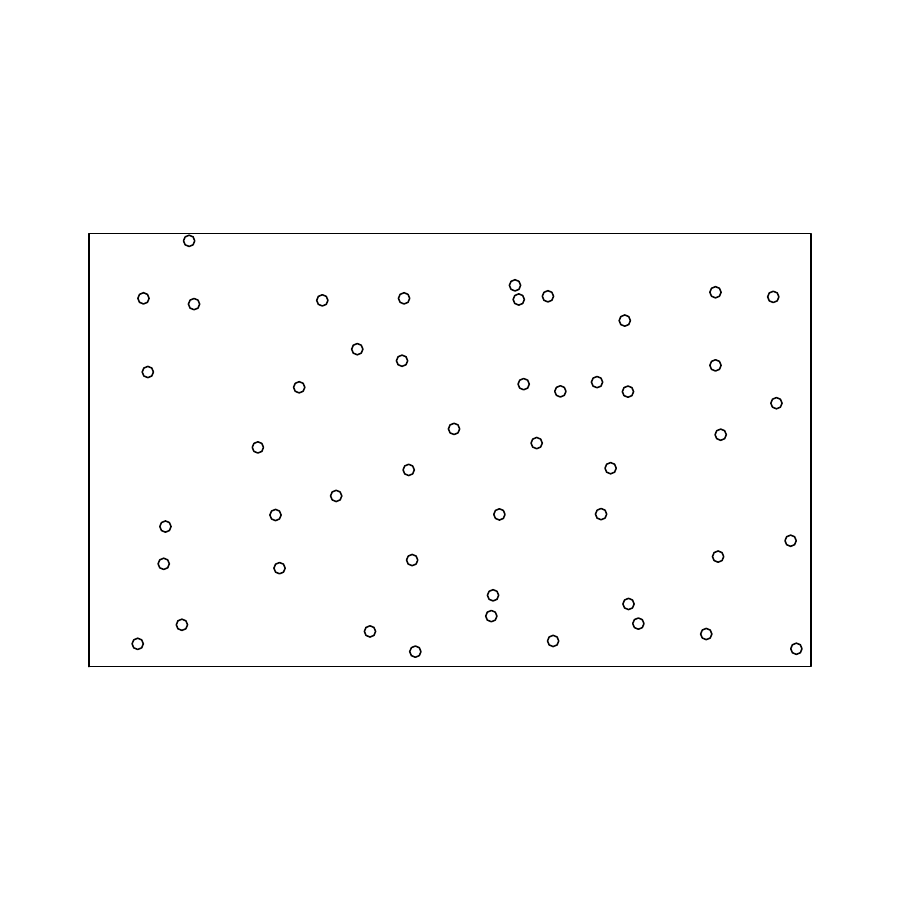}
 \label{subfig:data:termites}
 }
 \quad
 \subfloat[][Locations of 204 seedlings and saplings of Japanese black
    pines in a \SI{10}{\metre} by \SI{10}{\metre} region.]{
 \includegraphics[angle=0,width=.4\textwidth]{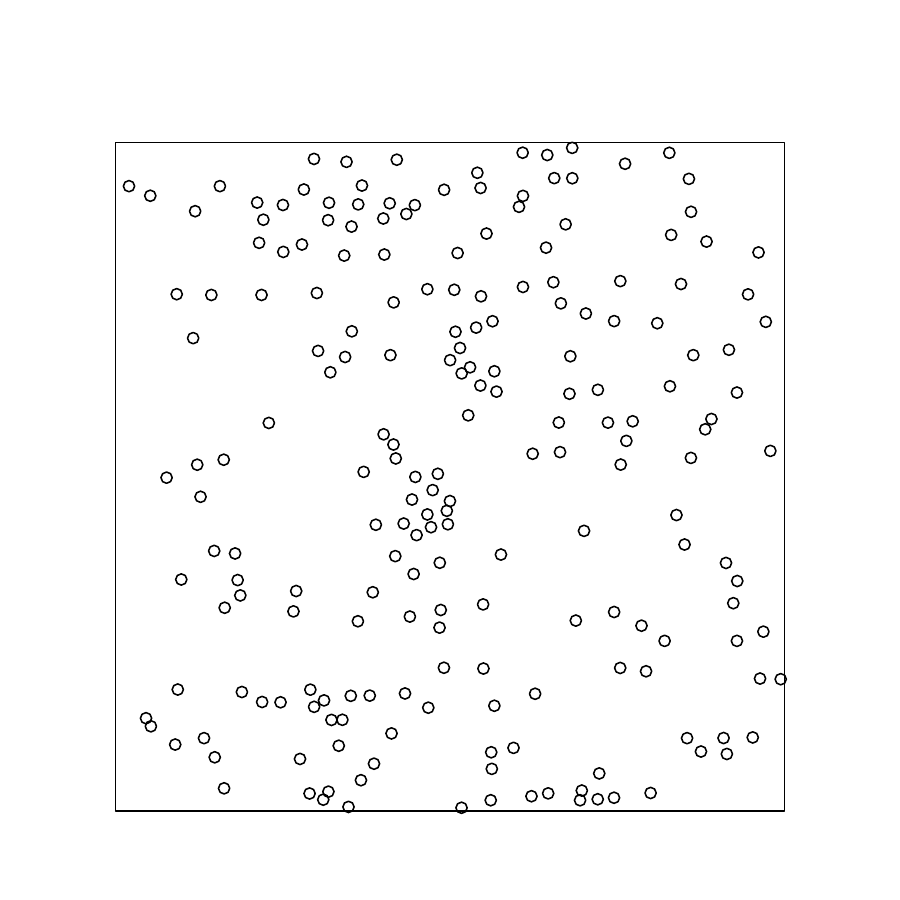}
 \label{subfig:data:Japanese}
 }
 \quad
 \subfloat[][Locations of 876 cells of a mucous membrane in a rectangular region rescaled to
 unit width and height 0.81.]{
 \includegraphics[angle=0,width=.4\textwidth]{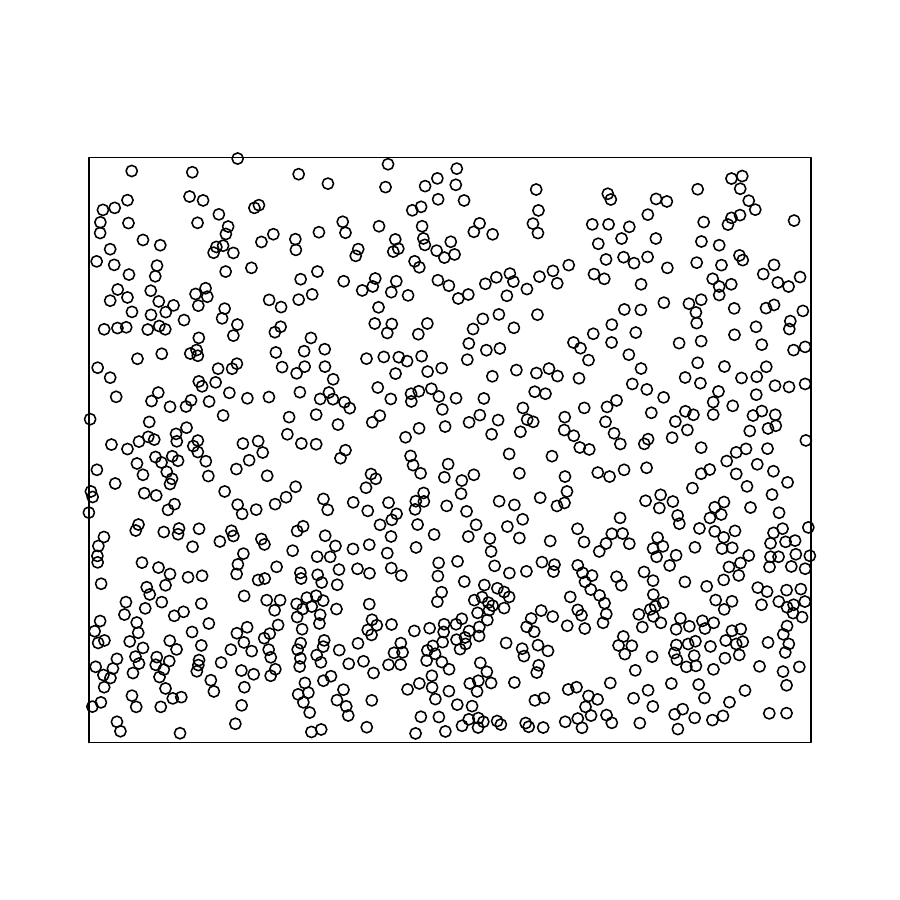}
 \label{subfig:data:mucosa}
 }
  \caption{
    Examples of spatial point patterns. 
  }
  \label{fig:data}
\end{figure}

For comparison Figure~\ref{fig:intro} shows realizations in the unit
square of three stationary DPPs with the same intensity of points.
Figure~\ref{fig:intro}\subref{subfig:poisson:sim} shows a simulation of a Poisson
process, which is a
special case of a DPP with no repulsion (or no
interaction). Figure~\ref{fig:intro}\subref{subfig:detgauss:sim}
 shows a simulation of a DPP with
moderate repulsion; in comparison with Figure~\ref{fig:intro}\subref{subfig:poisson:sim}
the point pattern looks more regular.
Figure~\ref{fig:intro}\subref{subfig:detjinc:sim} shows, in a sense made more
precise
in Appendix~\ref{sec:quantify},
a simulation of a DPP for the
strongest case of repulsiveness when the intensity is fixed.
The point pattern in Figure~\ref{fig:intro}\subref{subfig:detjinc:sim}
is clearly regular but not to the same extent as can be obtained in a Gibbs hard-core point
process.  

\begin{figure}[!htbp]%
\centering
\subfloat[][]{
 \includegraphics[angle=0,width=.27\textwidth]{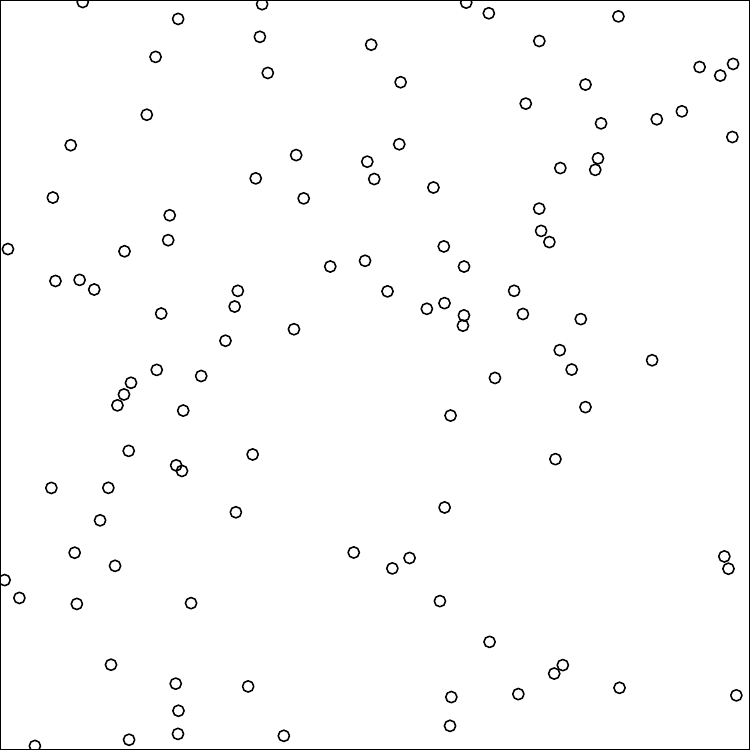}
 \label{subfig:poisson:sim}
 }
 \quad
\subfloat[][]{
 \includegraphics[angle=0,width=.27\textwidth]{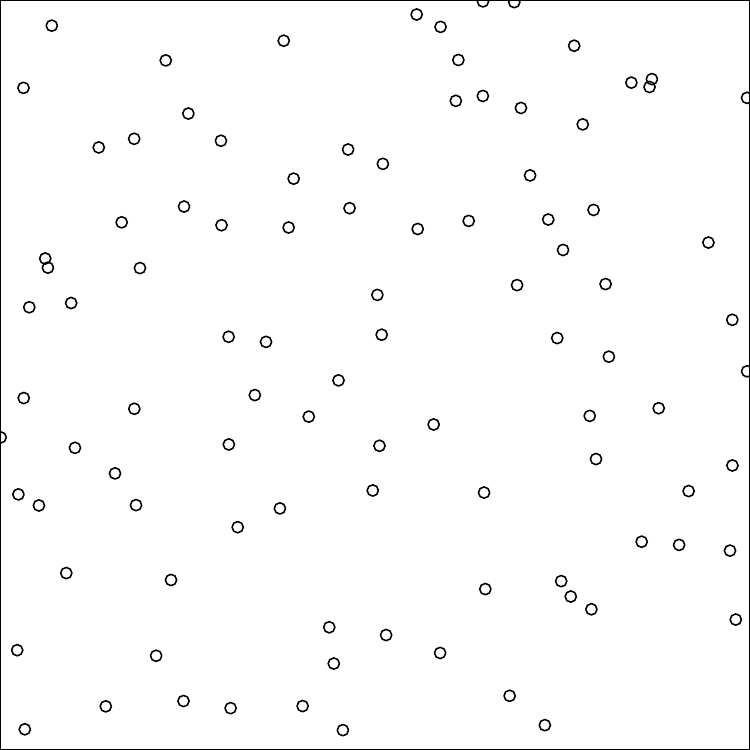}
 \label{subfig:detgauss:sim}
 }
 \quad
 \subfloat[][]{
 \includegraphics[angle=0,width=.27\textwidth]{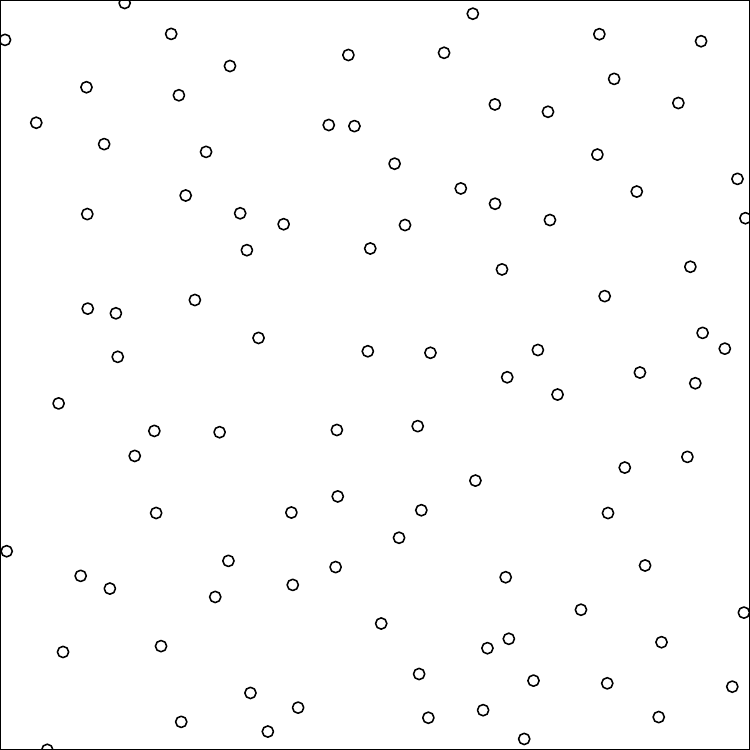}
 \label{subfig:detjinc:sim}
 }
  \caption{
    Realizations of stationary DPPs within a unit square:
    \subref{subfig:poisson:sim} a Poisson process;
    \subref{subfig:detgauss:sim} a
    DPP with moderate repulsion (a Gaussian DPP as described in Section~\ref{sec:examples});
    \subref{subfig:detjinc:sim} a stronger repulsive DPP (a
    jinc-like DPP as
    described in Section~\ref{sec:spectral}).}
  \label{fig:intro}
\end{figure}

\subsection{Gibbs point processes versus determinantal point processes}\label{e:vs}

 The usual class of
point processes used for modelling repulsiveness is the class of Gibbs
point processes, including Markov point processes and pairwise
interaction point processes
\citep{ripley:77,ripley:kelly:77,stoyan:kendall:mecke:95,lieshout:00,Diggle:03,moeller:waagepetersen:00,illian:penttinen:stoyan:stoyan:08,Gelfand:etal:10}.
In general for Gibbs point processes,
\begin{itemize}
\item moments are
not expressible in closed form;
\item likelihoods involve intractable
normalizing constants;
\item rather elaborate Markov chain Monte
Carlo methods are needed for simulations and approximate likelihood
inference;
\item when
dealing with infinite Gibbs point processes defined on $\real^d$,
`things' become rather complicated, e.g.\ conditions for
existence and uniqueness
as well as problems with edge effects;
\end{itemize}
see
\cite{moeller:waagepetersen:00,moeller:waagepetersen:07} and the
references therein. For Gibbs point processes, 
as maximum likelihood inference is complicated, 
the most popular and much quicker alternative inference
procedure is
based on pseudo-likelihood \citep{besag:77b,jensen:moeller:91,baddeley:turner:00,Gelfand:etal:10}. The pseudo-likelihood function is specified
in terms of the Papangelou conditional intensity which does not depend
on the normalizing constant from the likelihood.

In contrast, DPPs possess a number of
appealing properties: Considering a DPP
defined on $\real^d$ (with $d=2$ in most of our examples), its distribution is
 specified
 by a kernel (which we assume is a continuous complex covariance
 function) $C$ defined on $\real^d\times
 \real^d$ and which is properly
 scaled
 (these regularity
 conditions on $C$ are imposed to ensure existence of the process
as discussed in Section~\ref{sec:exist}). Then
\begin{enumerate}
\item[(a)] there are simple conditions for existence of the process
  and there is no phase transition: uniqueness of the DPP is ensured
  when it exists;
\item[(b)] moments are known:
by the very definition, all orders of moments are
described by certain determinants of matrices with entries given by
$C$ (Section~\ref{sec:defdet});
\item[(c)] edge effects is not a problem: the restriction of the
  DPP to a compact
  subset $S\subset\real^d$ is also a DPP with its
  distribution specified by the restriction of $C$ to $S\times S$;
 \item[(d)]
 the DPP restricted to $S$ ($S\subset\real^d$ compact)
has a
density (with respect to a Poisson
process): the density is
given by a normalizing constant, with a closed form expression, times
the determinant of a matrix with entries given
by a certain
kernel $\Ct$ which is
obtained by a simple
transformation of the eigenvalues in a spectral representation
of $C$ restricted to $S\times S$ (Section~\ref{sec:densities});
\item[(e)]   if such a spectral representation is not explicitly known, we can approximate it in practice by a Fourier series (Section~\ref{sec:approx2});
\item[(f)] the DPP can easily be simulated: basically because it
is a mixture of `determinantal projection point
processes' (Section~\ref{sec:sim}).
\end{enumerate}
Indeed, DPPs possess further useful properties, e.g.\
 a one-to-one smooth transformation or an independent thinning of the
DPP is also a DPP (Appendix~\ref{sec:trans-thin}); the reduced Palm measure of a DPP is also a DPP (Appendix~\ref{sec:Palm}).
Due to (a)-(f), modelling and
estimation for parametric families of DPPs become tractable as
discussed in Sections~\ref{sec:stationary}-\ref{sec:conremarks}. In
particular, we shall calculate likelihood functions, maximum
likelihood estimates, and likelihood ratio statistics for parametric
DPP models.

The link between Gibbs point processes and DPPs have been studied in
\cite{Georgii:Yoo:05}, where the key is the description of the
Papangelou conditional
intensity for a DPP. From a statistical perspective this link is of
limited interest, since for parametric families of DPPs,
the Papangelou conditional
intensity is not easier to handle than the likelihood, and the
pseudo-likelihood is in fact less easy to calculate than the likelihood.
Although DPPs may be considered as a subclass of Gibbs point processes,
at least when they are defined on a bounded region, we rather think of DPPs
as an interesting model class in itself.

\subsection{Software}

The statistical analyses in this paper have been conducted with R
\citep{R:11}. The software we have developed is freely available
as a supplement to the \texttt{spatstat} library
\citep{baddeley:turner:05} enabling users to both simulate and fit
parametric models of
stationary and inhomogeneous DPP models.

\subsection{Outline}

The paper is organized as follows. Section~\ref{sec:background}
is our tutorial (cf.\ (i) in Section~\ref{sec:aim}). In
Section~\ref{sec:stationary} we
study stationary DPPs for several purposes: to simplify the
general condition for existence of a DPP; to construct useful
parametric model classes of DPPs; and to understand to which extent
they can model repulsiveness. 
Using a Fourier basis approach, we derive in
Section~\ref{sec:approx2} 
approximations of the spectral representations
of the kernels $C$ and $\Ct$ (cf.\ (d)-(e) in Section~\ref{e:vs})
which make simulation and inference feasible for parametric models
of DPPs. Section~\ref{sec:hom} presents statistical
inference procedures for parametric models of
stationary DPPs and parsimonious parametric DPP
models are fitted to the datasets in
Figure~\ref{fig:data}\subref{subfig:data:spanish}-\subref{subfig:data:termites}.  
In Section~\ref{sec:inhom}, we discuss inference for inhomogeneous DPPs and we fit  parametric DPP
models to the datasets in
Figure~\ref{fig:data}\subref{subfig:data:Japanese}-\subref{subfig:data:mucosa}.
Section~\ref{sec:conremarks} contains our
concluding remarks. Finally,
Appendices~\ref{sec:trans-thin}-\ref{sec:approx3} contain the technical
proofs of our results
and provide supplementary
methods, examples, and remarks to the ones presented in the main text.

\section{Definition, existence, simulation, and densities
  for determinantal point processes}\label{sec:background}

The following provides the background material needed in this paper on
 the
definition (Sections~\ref{sec:mom}-\ref{sec:defdet}),
 existence (Section~\ref{sec:exist}), simulation (Section~\ref{sec:sim}), and
density expression for
 a general DPP
defined on a Borel set
$B\subseteq\real^d$ (Section~\ref{sec:densities}).
We shall mainly consider the cases $B=\real^d$ and
 $B=S$, where $S$ is compact.
We aim at a simple exposition, though
it is
unavoidable at some places to be a bit technical.

We denote by $X$ a simple
locally finite spatial point process defined on $B$,
i.e.\ we can view realizations of $X$ as
locally finite subsets of $B$
(for
measure theoretical details, see e.g.\ \cite{moeller:waagepetersen:00}
and the references therein). We refer to
the elements (or points) of $X$ as
events.

\subsection{Moments for spatial point processes}\label{sec:mom}

Since DPPs are defined in terms of their moment properties as expressed
by their so-called product density functions,
$\rho^{(n)}:B^{n}\to[0,\infty)$, $n=1,2,\ldots$,
 we start by
recalling this notion. 

Intuitively, for any pairwise distinct points $x_1,\ldots,x_n\in B$,
 $\rho^{(n)}(x_1,\ldots,x_n)\,\mathrm
dx_1\cdots\mathrm dx_n$ is the probability that  for each
$i=1,\cdots,n$, $X$ has a point in an infinitesimally
small region around $x_i$ of volume $\mathrm dx_i$.
Formally, $X$ has $n$'th order
product density function
$\rho^{(n)}:B^{n}\to[0,\infty)$ if this function is locally integrable (with
respect to Lebesgue measure restricted to $B^n$) and
for any Borel function $h:B^{n}\to[0,\infty)$,
\begin{equation}\label{e:def1}
\mean\sum^{\not=}_{x_1,\ldots,x_n\in
  X}h(x_1,\ldots,x_n)=\int_B\cdots\int_B\rho^{(n)}(x_1,\ldots,x_n)h(x_1,\ldots,x_n)\,\mathrm
dx_1\cdots\mathrm dx_n
\end{equation}
where $\not=$ over the summation sign means that $x_1,\ldots,x_n$ are
pairwise distinct events. See e.g.\ \cite{stoyan:kendall:mecke:95}.
Clearly, $\rho^{(n)}$ is only uniquely defined up to a
Lebesgue nullset.
We shall henceforth
require that
 $\rho^{(n)}(x_1,\ldots,x_n)=0$ if
$x_i=x_j$ for some $i\not=j$. This convention becomes consistent with
Definition~\ref{def1} below.

In particular, $\rho=\rho^{(1)}$ is the
intensity function and $g(x,y)=\rho^{(2)}(x,y)/[\rho(x)\rho(y)]$
 is the pair correlation function, where we set $g(x,y)=0$ if $\rho(x)$
 or $\rho(y)$ is zero. By our convention above,
 $g(x,x)=0$ for all $x\in B$.
The terminology
`pair correlation function' may be confusing, but it is adapted from
physics and commonly used
by spatial statisticians. 
For a Poisson point
process with an intensity function $\rho$, and for $x\not= y$,
we have 
$g(x,y)=1$ if $\rho(x)>0$
and $\rho(y)>0$. 

\subsection{Definition}\label{sec:defdet}

We need the following notation.
Let $\complex$ denote the complex plane. For a complex number
$z=z_1+\mathrm{i}z_2$ (where $z_1,z_2\in\real$ and
$\mathrm{i}=\sqrt{-1}$), denote $\conj{z}=z_1-\mathrm{i}z_2$ the
complex conjugate and $|z|=\sqrt{z_1^2+z_2^2}$ the modulus.
For a square complex matrix $A$, denote $\det A$ its determinant.
For any function $C:B\times B \to
\complex$, let
$[C](x_1,\dots,x_n)$ be the $n\times n$ matrix with $(i,j)$'th
entry $C(x_i,x_j)$. We refer to $C$ as a kernel. In most
examples of applications, the kernel will be real (the Ginibre DPP is
an exception). 
\begin{definition}\label{def1}
Suppose that a simple locally finite spatial point process $X$ on $B$ has
product density functions 
\begin{equation}\label{eq:productIntensity}
  \rho^{(n)}(x_1,\dots,x_n) = \det[C](x_1,\dots,x_n),\quad
  (x_1,\dots,x_n)\in B^{n},\quad n=1,2,\ldots.
\end{equation}
Then $X$ is called a
\emph{determinantal point process (DPP)} with kernel $C$, and we write
$X\sim\detproc_B(C)$.
\end{definition}

\begin{remark}\label{rem:DPPonS}
For $X\sim\detproc_B(C)$ and any Borel set $A\subseteq B$, define
  $X_A=X\cap A$ and denote its distribution by $\detproc_B(C;A)$. We
  also write $\detproc_A(C)$ for the distribution of the DPP on $A$ with kernel given by the restriction of $C$ to
  $A\times A$. Then property (c) in Section~\ref{e:vs}
follows
  directly from Definition~\ref{def1}, i.e.\
  $\detproc_A(C)=\detproc_B(C;A)$.
Further, when $B=\real^d$, we
  write $\detproc(C)$ for $\detproc_{\real^d}(C)$, and
  $\detproc(C;A)$ for $\detproc_{\real^d}(C;A)$.
\end{remark}

Some further remarks are in order. 

A Poisson process is the special
case where $C(x,y)=0$
whenever $x\not=y$.

Note that 
$C:\real^d\times\real^d\rightarrow\complex$ needs to be 
non-negative definite to ensure $\rho^{(n)}\geq
0$ in \eqref{eq:productIntensity}. Thus $C$ is
a complex  covariance function if and only if it is Hermitian,
 i.e.\ $C(x,y)=\conj{C(y,x)}$ for all
  $x,y\in\real^d$. 

Suppose $X\sim\detproc(C)$. 
Then there is no other point process
satisfying \eqref{eq:productIntensity}
(Lemma~4.2.6 in \cite{Hough:etal:09}).

By \eqref{eq:productIntensity}, the intensity function 
is
\begin{equation}\label{intensity}
  \rho(x) = C(x,x)
,\quad
  x\in \real^d,
\end{equation}
and the pair correlation function of $X$ is
\[g(x,y)=1-\frac{C(x,y)C(y,x)}{C(x,x)C(y,y)}\quad\mbox{if $C(x,x)>0$
  and $C(y,y)>0$}\]
while it is zero otherwise. 

Repulsiveness of the DPP is reflected by the following.
If $C$ is Hermitian, then $g\le1$ and for any $n=2,3,\ldots$, 
\[\rho^{(n)}(x_1,\ldots,x_n)\le\rho(x_1)\cdots\rho(x_n)\]
(the inequality is in general sharp and follows from the fact
that the determinant of a complex covariance matrix is less than or equal to
the product of its diagonal elements). Furthermore, if $C$ is continuous,
$\rho^{(n)}$ is also continuous and $\rho^{(n)}(x_1,\dots,x_n)$ tends
to zero as the Euclidean distance $\|x_i-x_j\|$ goes to zero for some $i\neq
j$, cf.\ \eqref{eq:productIntensity}. 

For later use,
let $R(x,y)=C(x,y)/[C(x,x)C(y,y)]^{1/2}$, where we set $R(x,y)=0$ if
$C(x,x)=0$ or $C(y,y)=0$. Note that 
\begin{equation}\label{eq:pcf}
  g(x,y) = 1-|R(x,y)|^2, \quad x,y\in\real^d,
\end{equation}
and when $C$ is a covariance
function, $R$ is its corresponding correlation function.

\subsection{Existence}\label{sec:exist}

Existence of a DPP
on $\real^d$ is ensured by the following
assumptions (C1)-(C2) on $C$.

First, as argued below, we find it natural to assume the
  following condition:
\begin{align*}
 &\mbox{(C1)}\qquad \mbox{$C$ is a continuous complex
  covariance function.}
\end{align*}
Then, if we
let $S\subset\real^d$ denote a generic compact set and  $L^2(S)$ the space
of square-integrable functions $h:S\to\complex$, we obtain the following
by Mercer's theorem
(see e.g.\ Section~98 in
\cite{riesz:nagy}). Under \mbox{(C1)},
$C$ restricted to $S\times S$ has a spectral representation,
\begin{equation}\label{eq:eigenrep}
  C(x,y) = \sum_{k=1}^\infty \lambda_k {\phi_k(x)}
  \conj{\phi_k(y)},\quad (x,y)\in S\times S,
\end{equation}
with absolute and uniform convergence of the
series, and where
\begin{itemize}
\item the set of eigenvalues $\{\lambda_k\}$
is unique, each non-zero eigenvalue is real and
has finite
multiplicity,
and the only possible accumulation point of the eigenvalues is $0$;
\item the eigenfunctions $\{\phi_k\}$ form an orthonormal basis of
$L^2(S)$, i.e.\
\begin{equation}
\int_S {\phi_k(x)}\conj{\phi_l(x)}\dee x = \left\{
  \begin{tabular}{ll}
    1 & if $k=l$,  \\
    0 & if $k\neq l$,
  \end{tabular}\right. \label{e:spectralrep}
\end{equation}
and any $h\in L^2(S)$ can be written as
$h=\sum_{k=1}^\infty\alpha_k\phi_k$, where $\alpha_k\in\complex$, $k=1,2,\ldots$.
Moreover, $\phi_k$ is continuous if $\lambda_k\not=0$.
\end{itemize}
When we need to stress that the
eigenvalue $\lambda_k$ depends on $S$, we
write $\lambda_k^S$.
Second, we consider the following condition:
\begin{align*}
& \mbox{(C2)}\qquad
\mbox{$\lambda_k^S\le1$ for all compact $S\subset\real^d$ and all $k$.}
\end{align*}
The following result is verified in Appendix~\ref{sec:existproof}.

\begin{theorem}\label{thm:existence}
Under (C1), existence of $\detproc(C)$ is equivalent to (C2).
\end{theorem}


\begin{assumption} In the remainder of this paper, $X\sim\detproc(C)$
  with $C$ satisfying the
  conditions (C1) and (C2).
\end{assumption}

Various comments on (C1) and (C2) are in order.

Usually, for statistical models of covariance functions, (C1) is satisfied,
and so (C2) becomes the essential
condition. As discussed in Section~\ref{sec:approx1}, (C2) simplifies in the
stationary case of $X$. It seems hard to provide an intuition 
why the eigenvalues need to be bounded by one, but they appear as probabilities for the simulation
algorithm in Section~\ref{sec:sim}.

As noticed in \cite{Hough:etal:09}, there are interesting examples of
DPPs with non-Hermitian kernels, but they do not
possess various general properties, and the results and methods in
our paper rely much on the spectral representation
\eqref{eq:eigenrep}. We therefore confine ourselves
 to the Hermitian case of $C$.

We find that (C1) is often a natural condition for several
reasons: statisticians are used to deal with
  covariance functions;
as seen in the proof of
Theorem~\ref{thm:existence},
the situation simplifies when $C$ is assumed to be continuous;
  continuity of $C$ implies continuity of the intensity
  function and the pair correlation function; conversely, if $C$ is
  real and non-negative, continuity of $\rho$ and $g$ implies
  continuity of $C$.

When we are only
  interested in considering a DPP $Y$ on a
  given compact set $S\subset\mathbb R^d$, then (C1)-(C2) can be
  replaced by the assumption that C is a continuous complex covariance
  function defined on $S\times S$ such that $\lambda_k^S\le 1$ for all
  $k$. The results in Sections~\ref{sec:sim}-\ref{sec:densities} are
  then valid for $Y$, even if there is no continuous
  extension of $C$ to $\mathbb R^d\times\mathbb R^d$ which satisfies
  (C1)-(C2). However, it is convenient to assume (C1)-(C2) as we in
  Sections~\ref{sec:stationary}-\ref{sec:hom} consider stationary
  DPPs.

Though a Poisson process is determinantal from Definition \ref{def1}, it
is excluded by our approach where $C$ is continuous. In
particular, \eqref{eq:eigenrep} does not hold for a Poisson process 
(therefore many of our results as well as those 
established in \cite{Hough:etal:09} do not
hold for a Poisson process).




\subsection{Simulation}\label{sec:sim}

  An algorithm for simulating a finite DPP in a very general setup
is provided in \cite{Hough:etal:06}.
There are special cases of DPPs
which may be simulated in a different manner, e.g.\ the Ginibre ensemble,
see Section 4.3 in \cite{Hough:etal:09}.

We
explain and prove the simulation
algorithm of \cite{Hough:etal:06} in the specific case where we want to
simulate $X_S\sim\detproc(C;S)$ with $S\subset\real^d$ compact.
Our
implementation of the algorithm becomes more efficient than the one
in \cite{Scardicchio:etal:09}, and
our description and proof
use
mainly linear algebra and are less technical than that in
\cite{Hough:etal:06}. We notice the following definition.

\begin{definition}\label{def:projection}
Let $S\subset \real^d$ be compact and assume
all non-zero eigenvalues $\lambda_k^S$ are
one. Then $C$ restricted to $S\times S$ is called a projection kernel,
and $X_S$ is
 called a {\em determinantal
  projection point process}.
\end{definition}

The terminology in Definition~\ref{def:projection}
seems commonly used
(e.g.\ \cite{Hough:etal:06} and \cite{Hough:etal:09});
\cite{McCullagh:Moeller:06} call a determinantal
  projection point process a special DPP because of its special
  properties as discussed below.
Now, consider the spectral representation
\eqref{eq:eigenrep} of
$C$ restricted to $S\times S$.
The simulation algorithm is based on the following result (Theorem~7 in \citet{Hough:etal:06}; see also
Theorem~4.5.3 in \cite{Hough:etal:09}).

\begin{theorem}\label{thm:simulation}
  For $k=1,2,\dots$, let $B_k$ be independent Bernoulli variables with
  mean $\lambda_k$. Define the random projection kernel
  $K:S\times S \to \complex$ by
  \begin{equation}\label{eq:projectionKernelInfty}
    K(x,y) = \sum_{k=1}^\infty B_k {\phi_k(x)} \conj{\phi_k(y)}.
  \end{equation}
  Then
  \begin{equation}\label{e:imp}
    \detproc_S(K)\sim\detproc(C;S).
  \end{equation}
\end{theorem}

In other words, if we first generate the independent Bernoulli
  variables, and second generate a determinantal projection point
  process on $S$ with kernel $K$, then the resulting point process
  follows $\detproc(C;S)$. Note that if $N(S)=n(X_S)$ denotes
the number of events in $S$, then
\begin{equation}\label{e:count}
N(S) \sim \sum_{k=1}^\infty B_k,\quad \mean[N(S)] = \sum_{k=1}^\infty \lambda_k,\quad \var[N(S)]
  =\sum_{k=1}^\infty \lambda_k(1-\lambda_k).
\end{equation}
The first result in \eqref{e:count} follows from \eqref{e:imp} and
Theorem~\ref{thm:algorithmWorks} below (or from Lemma~4.4.1 in
\cite{Hough:etal:09}), and the first result immediately implies the
two other results.

\subsubsection{Simulation of Bernoulli variables}\label{sec:Bern}

This and the following section descibe a two step simulation procedure based
on Theorem~\ref{thm:simulation}.

Recall that $\prob(B_k=1)=1-\prob(B_k=0)=\lambda_k$, $k=1,2,\ldots$, and
define
$B_0=\lambda_0=1$. With probability one, $\sum
B_k<\infty$, since $\sum\lambda_k=\int_SC(x,x)\,\mathrm dx<\infty$ as
$S$ is bounded and $C$ is continuous.
Consequently, with probability one, the random variable
$M=\max\{k\ge0:B_k\neq0\}$ is finite.
For any integer $m>0$, it is easily
verified that $B_0,\ldots,B_{m-1}$ are independent of the
event $\{M=m\}$. Therefore the strategy is first to generate a
realization $m$ of $M$, second independently generate
realizations of the Bernoulli variables $B_k$ for $k=1,\ldots, m-1$
(if $m=0$ we do nothing), and third
set $B_m=1$ and $B_k=0$ for $k=m+1,m+2,\ldots$.
Simulation of these Bernoulli variables is
of course easily done. For simulation of $M$, we use the inversion
method described in Appendix~\ref{sec:inversionM}.

\subsubsection{Simulation of  determinantal projection point process}
\label{sec:simproc}

Suppose we have generated a realization of the Bernoulli variables
$B_k$ as
described in Section~\ref{sec:Bern} and we now want to generate a
realization from $\detproc_S(K)$ with $K$ given by
\eqref{eq:projectionKernelInfty}.

Let $n=\sum_{k=1}^\infty B_k$ denote
the number of non-zero $B_k$'s with $k\ge1$ (as foreshadowed in
connection to \eqref{e:count}, $n$ can be considered as a realization of the
count $N(S)$).
If $n=0$, then $K=0$ and
a realization from $\detproc_S(K)$
 is simply equal to the empty point configuration.
 Assume that $n>0$ and 
without
loss of generality that
\begin{equation}\label{eq:projectionKernelN}
  K(x,y) = \sum_{k=1}^n {\phi_k(x)}\conj{\phi_k(y)} = \bv(y)^*\bv(x)
\end{equation}
where $\bv(x)=
(\phi_1(x),\dots,\phi_n(x))^T$, and where
$^T$ and $^*$ denote the transpose and conjugate transpose
of a vector or a matrix. For $n$-dimensional complex column vectors
such as $\bv(x)$ and $\bv(y)$, we consider their usual inner product
$\langle\bv(x),\bv(y)\rangle=\bv(y)^*\bv(x)$.

\begin{algorithm}[H]
\caption{Simulation of determinantal projection point process}
\label{alg:sim}
\begin{algorithmic}
  \STATE {\bf sample} $X_n$ from the distribution with density
  $p_n(x)=\|\bv(x)\|^2/n$, $x\in S$
  \STATE {\bf set} $\be_1=\bv(X_n)/\|\bv(X_n)\|$
  \FOR{$i=(n-1)$ to $1$}
\STATE {\bf sample} $X_i$ from the distribution with density
   \begin{equation}\label{eq:p_i}p_i(x)=\frac{1}{i}\left[\|\bv(x)\|^2-\sum_{j=1}^{n-i}|\be_j^*\bv(x)|^2\right],\quad x\in S \end{equation}
  \STATE {\bf set} $\bw_i=\bv(X_i)-\sum_{j=1}^{n-i}\left(\be_j^*\bv(X_i)\right)\be_j$, $\be_{n-i+1}=\bw_i/\|\bw_i\|$
  \ENDFOR
  \STATE {\bf return} $\{X_1,\dots,X_n\}$
\end{algorithmic}
\end{algorithm}

The following theorem is proved in Appendix~\ref{sec:algproof}. 
It follows from the proof that
with probability one, $p_i(x)$ is a density, where we are
conditioning on $(X_n,\dots,X_{i+1})$ if $i<n$.

\begin{theorem}\label{thm:algorithmWorks} If $n>0$ and $K(x,y) = \sum_{k=1}^n
{\phi_k(x)}\conj{\phi_k(y)}$ for all $x,y\in S$, then
  $\{X_1,\dots,X_n\}$ generated by Algorithm~\ref{alg:sim}
is distributed as $\detproc_S(K)$.
\end{theorem}

To implement Algorithm~\ref{alg:sim} we need to sample from
  the densities $p_i$, $i=n,\ldots,1$. This may simply be done by
  rejection sampling with a uniform instrumental density and
  acceptance probability given by $p_i(x)/U$,
where $U$ is an upper bound on $p_i(x)$ for $x\in S$. We use the bound
$U=\sup_{y\in S}\|\bv(y)\|^2/i$, cf.\ \eqref{eq:p_i}, which simplifies
to $U=n/i$ for the Fourier
basis considered in Section~\ref{sec:approx2}.
Appendixes~\ref{sec:algproof}-\ref{sec:closeupper}
  discuss rejection sampling for
this and other choices of the instrumental distribution.

\subsection{Densities}\label{sec:densities}

This section briefly discusses the density expression for $X_S\sim\detproc(C;S)$
when $S\subset\real^d$ is compact. Recall that the eigenvalues
$\lambda_k=\lambda_k^S$ are assumed to be less than or equal to one.

In general, when some eigenvalues $\lambda_k$ are allowed
  to be one, the density of $X_S$ is not available.  But we can
  condition on the Bernoulli variables $B_k$ from
  Theorem~\ref{thm:simulation}, or just condition on $K(x,y)$ for all
  $x,y\in S$, to obtain the conditional density.
Note that the trace $\tr_S(K)=\int_S K(x,x)\dee x =
\sum_{k=1}^\infty B_k$ is almost surely finite. Conditional on $K$,
when $\tr_S(K)=n>0$, the ordered $n$-tuple of events of the
determinantal projection point process $X_S$ has density
\[p(x_1,\ldots,x_n)=\det[K](x_1,\ldots,x_n)/n!,\quad
(x_1,\ldots,x_n)\in S^n,\]
as verified in \eqref{e:p-den}. Moreover,
by Algorithm~1 and Theorem~\ref{thm:algorithmWorks},
\[p_n(x)=K(x,x)/n,\quad x\in S,\]
is the density for an arbitrary
selected event of $X_S$.
This is in agreement with the simple fact that
in the homogeneous
case, i.e.\ when the intensity $K(x,x)$ is constant on $S$,
any event of $X_S$ is uniformly distributed on $S$.

The most interesting case occurs when
  $\lambda_k<1$ for all $k=1,2,\ldots$, which means
that no $B_k$ is almost surely one. Then the  density
  of $X_S$ exists and is specified in Theorem~\ref{thm:density}
  below, where the following considerations and notation are used. If $P(N(S)=n)>0$, then
$P(N(S)=m)>0$ for $m=0,\ldots,n$, cf.\ \eqref{e:count}. Thus
\[\prob(N(S)=0)=\prod_{k=1}^\infty(1-\lambda_k)\]
is strictly positive, and we can
define
\begin{equation}\label{e:D-def}
  D = -\log\prob(N(S)=0)= - \sum_{k=1}^\infty
  \log(1-\lambda_k).
\end{equation}
Further, define
$\Ct:S\times S \to \complex$ by
\begin{equation}\label{e:defCt}
  \Ct(x,y) = \sum_{k=1}^\infty \lambdat_k {\phi_k(x)}
  \conj{\phi_k(y)}
\end{equation}
where
\begin{equation*}
  \lambdat_k={\lambda_k}/{(1-\lambda_k)}, \quad
  k=1,2,\dots.
\end{equation*}
Let $|S|=\int_S\mathrm d x$, and set $\det
  [\Ct](x_1,\ldots,x_n)=1$ if $n=0$. Then we have the following
  result, cf.\
  Appendix~\ref{sec:proofdensitythm}.

\begin{theorem}\label{thm:density} Assuming $\lambda_k<1$,
  $k=1,2,\ldots$, then $X_S$ is absolutely continuous
with respect to the homogeneous Poisson process on $S$
with unit intensity, and has
density
\begin{equation}\label{e:unconddensity}
  f(\{x_1,\ldots,x_n\}) = \exp({|S|-D}) \det
  [\Ct](x_1,\ldots,x_n)
\end{equation}
for all $(x_1,\ldots,x_n)\in S^n$ and $n=0,1,\ldots$.
\end{theorem}

Section~\ref{sec:approx2} and Appendix~\ref{sec:approx3} discuss
efficient ways of approximating $\Ct$ and $D$ when $X$ is stationary.

\section{Stationary models}\label{sec:stationary}

To the best of our knowledge,
parametric families of
DPP models have yet not been studied in the literature from a
statistical perspective. In the sequel we focus
on the stationary case of DPPs, discuss isotropy
(Section~\ref{sec:isotropyetc}),
give a simple condition for the
existence of a stationary DPP (Section~\ref{sec:approx1}),
construct various classes of parametric models
(Sections~\ref{sec:examples}-\ref{sec:spectral}).
 Inhomogeneous models of DPPs are discussed in Section~\ref{sec:inhom}.

Throughout this section,
$X\sim\detproc(C)$ where $C$ is of the form
\begin{equation}\label{stat-cov}
C(x,y)=C_0(x-y),\quad x,y\in\real^d.
\end{equation}
This condition  implies that  $X$ is stationary, i.e.\ its distribution is
invariant under translations. If $C$ is real, \eqref{stat-cov} is equivalent to the stationarity of $X$.

We also refer to $C_0$ as a covariance function.
Note that
$C_0(0)$, the variance corresponding to $C$,
equals $\rho$, the intensity  of $X$,
cf.\ \eqref{intensity}.

In light of Propositions~\ref{prop:trans}-\ref{prop:thinning},
 as inhomogeneous DPPs can be
obtained by transforming or thinning $X$, stationarity
is not a very restrictive
assumption. For example, by \eqref{e:C-trans},
if we transform $X$ by a one-to-one
continuous differentiable mapping $T$ such that its Jacobian matrix is
invertible, then $T(X)$ is a DPP with kernel
\begin{equation}\label{e:C-trans-again}
C_{{\mathrm{trans}}}(x,y)=|J_{T^{-1}}(x)|^{1/2}C_0(T^{-1}(x)-T^{-1}(y))|J_{T^{-1}}(y)|^{1/2}.
\end{equation}

\subsection{Isotropy}\label{sec:isotropyetc}

It is often convenient to require that $C_0$ is
 isotropic, meaning that $C_0(x)=\rho R_0(\|x\|)$ is invariant
 under rotations about the origin in $\real^d$. 
This is a natural simplification, since 
any stationary and an\-iso\-tro\-pic 
covariance function can be obtained from some stationary and isotropic
covariance function using some rotation followed by some rescaling, see e.g.\
\cite{goovaerts:97}.

Suppose $C_0$ is
 isotropic. Then $C_0$ is real, and the pair
 correlation function
depends only on the distance between pairs of points,
$g(x,y)=g_0(\|x-y\|)$, cf.\ \eqref{eq:pcf}. Hence commonly used
 statistical procedures
 based on the pair correlation function or the closely related
$K$-function apply (see
\cite{ripley:76,ripley:77} and \cite{moeller:waagepetersen:00}). In particular,
using the relation
\begin{equation}\label{e:r0}
|R_0(r)|=\sqrt{1- g_0(r)}
\end{equation}
we can define a `range of correlation', i.e.\ a distance $r_0>0$ such that
$g_0(r)\approx 1$ for $r\ge r_0$, as exemplified
later in \eqref{eq:ranges}. For many specific models for isotropic covariance,
including those studied in Section~\ref{sec:examples}, $R_0$ is a
decreasing function. By \eqref{e:r0} $g_0$ is then
 an increasing function from zero
to one. 

Examples of stationary and isotropic covariance
functions are studied in Sections~\ref{sec:examples}-\ref{sec:spectral}.
However, the following
Section~\ref{sec:approx1} does not involve an assumption of
isotropy, and the approximation of
$C_0$ studied in
Section~\ref{sec:approx2} is only approximately isotropic when $C_0$ is isotropic.

\subsection{A simple spectral condition for existence}\label{sec:approx1}

The following Proposition~\ref{prop:c1c2} simplifies condition
(C2). As it involves the spectral density for $C_0$, we 
start by recalling this and related notions.

For any number $p>0$ and Borel set $B\subseteq\real^d$,
let $L^p(B)$ be the class of $p$-integrable
 functions $h:B\to\complex$, i.e.\ $\int_B |h(x)|^p\,\mathrm
 dx<\infty$.
 Denote $\cdot$
the usual inner product in
$\real^d$.
For any Borel function $h:\real^d\to\complex$,
define the
Fourier transform $\mathcal F(h)$ of $h$ by
\[\mathcal F (h)(x)=
\int h(y)\mathrm{e}^{-2\pi \mathrm{i} x\cdot y}\,\mathrm dy,\quad x\in\real^d,\]
provided the integral exists,
 and the inverse Fourier transform $\mathcal F^{-1}(h)$  of $h$ by
\begin{equation*}\label{e:invfour}
\mathcal F^{-1}(h)(x)=
\int h(y)\mathrm{e}^{2\pi \mathrm{i} x\cdot y}\,\mathrm dy,\quad
x\in\real^d,
\end{equation*}
provided the integral exists. For instance, if $h\in L^1(\real^d)$, then $\mathcal
F(h)$ and $\mathcal F^{-1}(h)$ are well-defined.
Recall that $L^2(\real^d)$ is a Hilbert space with inner product
\[\langle h_1,h_2\rangle=\int h_1(x)\conj{h_2(x)}\,\mathrm dx\]
and the Fourier and inverse Fourier operators
initially defined on $L^1(\real^d)\cap L^2(\real^d)$ extend by
continuity to $\mathcal F:L^2(\real^d)\to L^2(\real^d)$ and
$\mathcal F^{-1}:L^2(\real^d)\to L^2(\real^d)$.  Furthermore,
these are unitary operators that
preserve the
inner
product, and $\mathcal F^{-1}$ is the inverse of $\mathcal F$.
See e.g.\ \cite{stein-weiss}.

By Khinchin's (or Bochner's) theorem, since $C_0$ is a continuous covariance
function, a spectral distribution function $F$
exists, i.e.\ $F$ defines a finite measure so that
\[C_0(x)=\int \mathrm{e}^{2\pi \mathrm{i} x\cdot y}\,\mathrm dF(y),\quad x\in\real^d.\]
If $F$ is differentiable, then the derivative $\varphi(x)=\mathrm
dF(x)/\mathrm dx$ is the
spectral density for $C_0$. In this case, $\varphi$ is non-negative, $\varphi\in
L^1(\real^d)$, and
$C_0=\mathcal F^{-1}(\varphi)$. On the other hand, if $C_0\in
L^1(\real^d)$ and $C_0$ is continuous (as assumed in this paper),
then the spectral density necessarily exists
(equivalently $F$ is differentiable),
$\varphi=\mathcal F(C_0)$, and $\varphi$ is
continuous and bounded.
See e.g.\ pages 331-332 in \cite{Yaglom:87}.

Alternatively, if $C_0\in L^2(\real^d)$ and $C_0$ is continuous,
the spectral density $\varphi$ also exists, since we can define
$\varphi=\mathcal F(C_0)$ in $L^2(\real^d)$ as explained above. In this case, $\varphi$ is
non-negative, belongs to $L^1(\real^d)\cap L^2(\real^d)$, but is not
necessarily continuous or bounded.
Note that if $C_0\in
L^1(\real^d)$, then $C_0\in
L^2(\real^d)$ by continuity of $C_0$.

The following is proved in Appendix~\ref{proofProp_spectral}.

\begin{proposition}\label{prop:c1c2}
Under (C1) and \eqref{stat-cov}, if $C_0\in L^2(\real^d)$, then
 (C2) is equivalent to that
\begin{equation}\label{e:C0-cond}
\varphi\le 1.
\end{equation}

\end{proposition}

\begin{assumption}\label{e:assumption}
Henceforth, in addition to (C1), we assume that $C_0\in L^2(\real^d)$
and that \eqref{e:C0-cond} holds.
\end{assumption}

The following corollary, verified in Appendix~\ref{proofCoro_spectral}, becomes useful in Section~\ref{sec:spectral}
 where we
discuss a spectral approach for constructing stationary DPPs.

\begin{corollary}\label{cor:first}
Under \eqref{stat-cov} the following two statements are equivalent.
\begin{enumerate}
\item[(i)] There exists $\varphi\in L^1(\real^d)$ with
  $0\le\varphi\le1$ and $C_0=\mathcal F^{-1}(\varphi)$.
\item[(ii)] Conditions (C1) and (C2) hold and $C_0\in L^2(\real^d)$.
\end{enumerate}
\end{corollary}

\begin{remark}\label{rem:trade-off} 
There is a trade-off between how large the intensity
and how repulsive a stationary DPP can be: 
Consider a parametric model for $C_0$ with parameters $\rho$
and $\theta$.
For each fixed value of
 $\theta$, (C2) is equivalent to $0\le\rho\le\rho_{\max}$ where
$\rho_{\max}=\rho_{\max}(\theta)$ may depend on $\theta$ and is determined by
\eqref{e:C0-cond}.
 As exemplified in Section~\ref{sec:examples},
 $\rho_{\max}$ will be a decreasing function
of the range of correlation (which only depends on $\theta$).
On the other hand, it may be more natural to
determine
the range of $\theta$ in terms of $\rho$. 
\end{remark}

\subsection{Examples of covariance function models}\label{sec:examples}

Numerous examples of stationary and isotropic covariance functions exist
\citep[see e.g.][]{Gelfand:etal:10},
while examples of stationary and anisotropic covariance functions are
discussed in \cite{laco:palma:posa:03}. This section starts by
considering the simple example of the circular covariance function
and continues with a brief discussion of the broad class of stationary
isotropic covariance functions obtained by scaling in normal-variance
mixture distributions, where a few specific examples of such
models are considered in more detail. Section~\ref{sec:spectral}
discusses further examples based on a spectral approach.

Examples of isotropic covariance
functions $C_0(x)$, where the range
\begin{equation}\label{e:kors}
\delta=\sup \{\|x\|:C_0(x)\neq 0\}
\end{equation}
is finite are given in \cite{Wu:95} and
\cite{Gneiting:02}. Then, by Definition~\ref{def1}, 
$X_A$ and $X_B$
are independent DPPs
if $A,B\subset\real^d$
are separated by a distance larger than $\delta$. 
In this
paper we only consider the circular covariance function to understand well
the quality of our approximations in Section~\ref{sec:approx2}.
For $d=2$, the
circular covariance function with finite range $\delta>0$ is given by
\begin{equation}\label{e:circcovfct}
  C_0(x) = \rho \frac{2}{\pi} \left(\arccos(\|x\|/\delta) - \|x\|/\delta
    \sqrt{1-(\|x\|/\delta)^2}\right),\quad  \|x\|<\delta.
\end{equation}
Note that $\pi\delta^2 C_0(x)/(4\rho)$
is the area of the
intersection of two discs, each with diameter $\delta$, and with
distance $\|x\|$ between the centers. Since this area is equal to the
autoconvolution of the indicator function of the disc with center at
the origin and with diameter $\delta$, the associated spectral density
becomes
\begin{equation*}
  \varphi(x) = (\rho/\pi) (J_1(\pi \delta \|x\|)/\|x\|)^2
\end{equation*}
where $J_1$ is the Bessel function of the first kind with parameter
$\nu=1$. This spectral density has maximal value
$\varphi(0)=\rho\pi\delta^2/4$,
so by \eqref{e:C0-cond}, a stationary and isotropic 
DPP with kernel \eqref{e:circcovfct}  
exists if and only if
$0\le\rho\le\rhomax$, where $\rhomax=4/(\pi\delta^2)$. Therefore, we require
\begin{equation}\label{eq:circularmax}
    \rho\delta^2\le4/\pi.
\end{equation}

In the sequel we focus on more interesting classes of covariance
functions. Let $Z$ be a $d$-dimensional standard normally distributed
random variable,
and $W$ be a strictly positive
random variable with $\mean(W^{-d/2})<\infty$, where $Z$ and $W$ are independent.
Then $Y=\sqrt W Z$ follows a normal-variance mixture
distribution, with density
\[h(x)=\mean\left[
  W^{-d/2}\exp\left(-\|x\|^2/(2W)\right)\right]/(2\pi)^{d/2},\quad x\in\real^d.\]
Note that $h(0)=\sup h$, and define
\[C_0(x)=\rho h(x)/h(0),\quad x\in\real^d.\]
The Fourier transform of $C_0$ is
\[\varphi(x)=
\rho\,\mean\left[\exp\left(-2\pi^2\|x\|^2W\right)\right]/h(0),\quad x\in\real^d\]
which is positive, showing that $C_0$ is  a stationary and isotropic
covariance function.
Note that  $\varphi$ is given
by the Laplace transform of $W$.
By \eqref{e:C0-cond}, a stationary DPP with kernel $C_0$  exists if
$0\le\rho\le\rhomax$,
where $\rho$ is the intensity and
\[\rho_{\max}= h(0)=\mean(W^{-d/2})/(2\pi)^{d/2}.\]

\cite{gneiting:97} presents several examples of pairs $h$ and
$\mathcal F(h)$ in the one-dimensional case $d=1$, and these examples
can be generalized to the multivariate case.  Here we restrict
attention to the following three examples, where $Y$ follows either a
multivariate normal distribution or two special cases of the
multivariate generalized hyperbolic distribution
\citep{OBN:77,OBN:78}. We let $\Gamma(a,b)$ denote the
Gamma-distribution with shape parameter $a>0$ and scale parameter
$b>0$.

First, taking $\sqrt{2W}=\alpha$, where $\alpha>0$ is a parameter, we obtain
the Gaussian (or squared exponential) covariance
function
\begin{equation}\label{e:gaussian}
  C_0(x)=\rho\exp\left(-\|x/\alpha\|^2\right),\quad x\in\real^d,
\end{equation}
and
\[\varphi(x)=\rho(\sqrt\pi\alpha)^d\exp\left(-\|\pi\alpha
  x\|^2\right),\quad x\in\real^d.\]
Hence
\begin{equation}\label{eq:maxrhoGauss}
  \rho_{\max}=(\sqrt\pi\alpha)^{-d}
\end{equation}
is a decreasing function of $\alpha$.

Second, suppose that
$W\sim\Gamma(\nu+d/2,2\alpha^2)$ where $\nu>0$ and $\alpha>0$.
Then
\[h(x)=\frac{\|x/\alpha\|^\nu
  K_\nu(\|x/\alpha\|)}{2^{\nu+d-1}(\sqrt{\pi}\alpha)^{d}\Gamma(\nu+d/2)},\quad
x\in\real^d,\]
where $K_\nu$ is the modified Bessel function of the second kind (see
Appendix~\ref{matern-approx}).
Hence
\begin{equation}\label{e:matern}
C_0(x)={\rho}\, \frac{2^{1-\nu}}{\Gamma(\nu)}\|x/\alpha\|^\nu
K_\nu(\|x/\alpha\|),\quad x\in\real^d,
\end{equation}
is the Whittle-Mat{\'e}rn covariance function, where for $\nu=1/2$,
$C_0(x)={\rho}\exp(-\|x\|/\alpha)$
is the exponential covariance function.
Moreover,
\[\varphi(x)=\rho\, \frac{\Gamma(\nu+d/2)}{\Gamma(\nu)}
\frac{(2\sqrt\pi\alpha)^d}{(1+\|2\pi\alpha x\|^2)^{\nu+d/2}}\, ,\quad
x\in\real^d,\]
so
\begin{equation}\label{eq:maxrhoWM}
\rho_{\max}=\frac{\Gamma(\nu)}{\Gamma(\nu+{d}/{2})(2\sqrt\pi\alpha)^d}
\end{equation}
is a decreasing function of $\nu$ as well as of $\alpha$.

Third, suppose that $1/W\sim\Gamma(\nu,2\alpha^{-2})$ where $\nu>0$ and
$\alpha>0$. Then
\[h(x)=\frac{\Gamma(\nu+d/2)}{\Gamma(\nu)(\sqrt{\pi}\alpha)^{d}
\left(1+\|x/\alpha\|^2\right)^{\nu+d/2}},\quad x\in\real^d,\]
is the density of a multivariate
$t$-distribution, and
\begin{equation}\label{e:C-cauchy}
C_0(x)=\frac{\rho}{\left(1+\|x/\alpha\|^2\right)^{\nu+d/2}}\,
,\quad x\in\real^d,
\end{equation}
is the generalized Cauchy covariance function.
Furthermore,
\[\varphi(x)=\frac{\rho(\sqrt\pi
  \alpha)^d2^{1-\nu}}{\Gamma(\nu+d/2)}\|2\pi\alpha x\|^\nu
K_\nu(\|2\pi\alpha x\|)
,\quad x\in\real^d,\]
so
\begin{equation}\label{eq:maxrhoCauchy}
\rho_{\max}=\frac{\Gamma(\nu+d/2)}{\Gamma(\nu)(\sqrt\pi\alpha)^{d}}
\end{equation}
is an increasing function of $\nu$ and a decreasing function of
$\alpha$.

For later use, notice that the Gaussian covariance function
\eqref{e:gaussian} with $\alpha=1/\sqrt{\pi\rho}$ is the limit of both
\begin{enumerate}
\item[(i)] the \wm{} covariance function \eqref{e:matern} with
  $\alpha=1/\sqrt{4\pi\nu\rho}$, and
\item[(ii)] the Cauchy covariance function \eqref{e:C-cauchy} with
  $\alpha=\sqrt{\nu/(\pi\rho)}$
\end{enumerate}
as $\nu\to\infty$.

We refer to a DPP model with kernel
\eqref{e:gaussian}, \eqref{e:matern}, or \eqref{e:C-cauchy} as the
Gaussian, Whittle-Mat{\'e}rn, or Cauchy model, respectively. In all
three models, $\alpha$ is a scale parameter of $C_0$, and
for the Whittle-Mat{\'e}rn and Cauchy models,
$\nu$ is a shape parameter of $C_0$. Their isotropic pair correlation
functions are as follows.

  For the Gaussian model: $g_0(r)=1-\exp\left(-2(r/\alpha)^2\right),\quad r\ge0$.

  For the \wm{} model: $g_0(r)=1-\left[2^{1-\nu}(r/\alpha)^\nu
      K_\nu(r/\alpha)/\Gamma(\nu)\right]^2,\quad r\ge0.$ 
      
 For the Cauchy model: $g_0(r)=1-\left[1+(r/\alpha)^2\right]^{-2\nu-d},\quad r\ge0.$\\
In each case, $g_0(r)$ is a strictly increasing function from zero to
one.
Moreover, for fixed $\rho$ and
$\nu$, the upper limit  $\alpha_{\max}$ of $\alpha$ is given
by
\begin{equation}\label{eq:alphamax}
\alpha_{\max}=1/\sqrt{\pi\rho}, \quad \alpha_{\max}=1/\sqrt{4\pi\nu\rho}, \quad \alpha_{\max}=\sqrt{\nu/(\pi\rho)}
\end{equation}
for the Gaussian, \wm{}, and Cauchy models, respectively.

In the sequel, let $d=2$. For a given model as above, we choose the range of
correlation $r_0$ such that $g_0(r_0)=0.99$, whereby the isotropic correlation
function given by \eqref{e:r0} has absolute value $0.1$. While it is
straightforward to determine $r_0$ for the Gaussian and Cauchy model,
$r_0$ is not expressible on closed form for the \wm{} model, and in
this case we use the empirical result of \cite{Lindgren:etal:11}. The ranges of correlation for
the Gaussian, \wm{} and Cauchy models are then
\begin{equation}\label{eq:ranges}
  r_0 = \alpha \sqrt{-\log(0.1)}, \quad r_0 = \alpha \sqrt{8\nu}, \quad r_0 = \alpha \sqrt{0.1^{-1/(\nu+1)}-1},
\end{equation}
respectively. In each case, $r_0$ depends linearly on $\alpha$, and
when $\nu$ is fixed,  $\rhomax$ 
decreases as $r_0$ increases, since $\rho_{\max}$ is proportional to
$r_0^{-d}$, cf.\ \eqref{eq:maxrhoGauss}, \eqref{eq:maxrhoWM}, and
\eqref{eq:maxrhoCauchy}. There is a similar trade-off between how
large the intensity and the range of the circular covariance function
can be, cf.\ \eqref{eq:circularmax}.

\begin{figure}[!htbp]%
  \centering
  \subfloat[][]{
    \includegraphics[scale=.45]{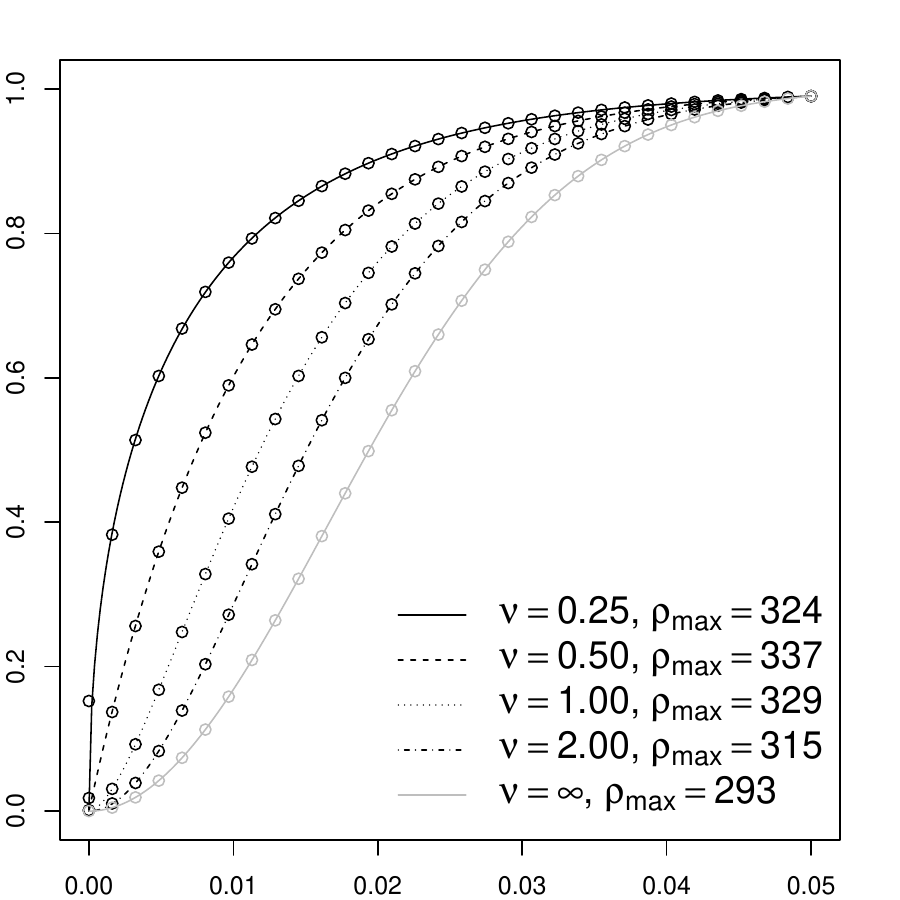}
    \label{subfig:pcf:matern}
  }
  \quad
  \subfloat[][]{
    \includegraphics[scale=.45]{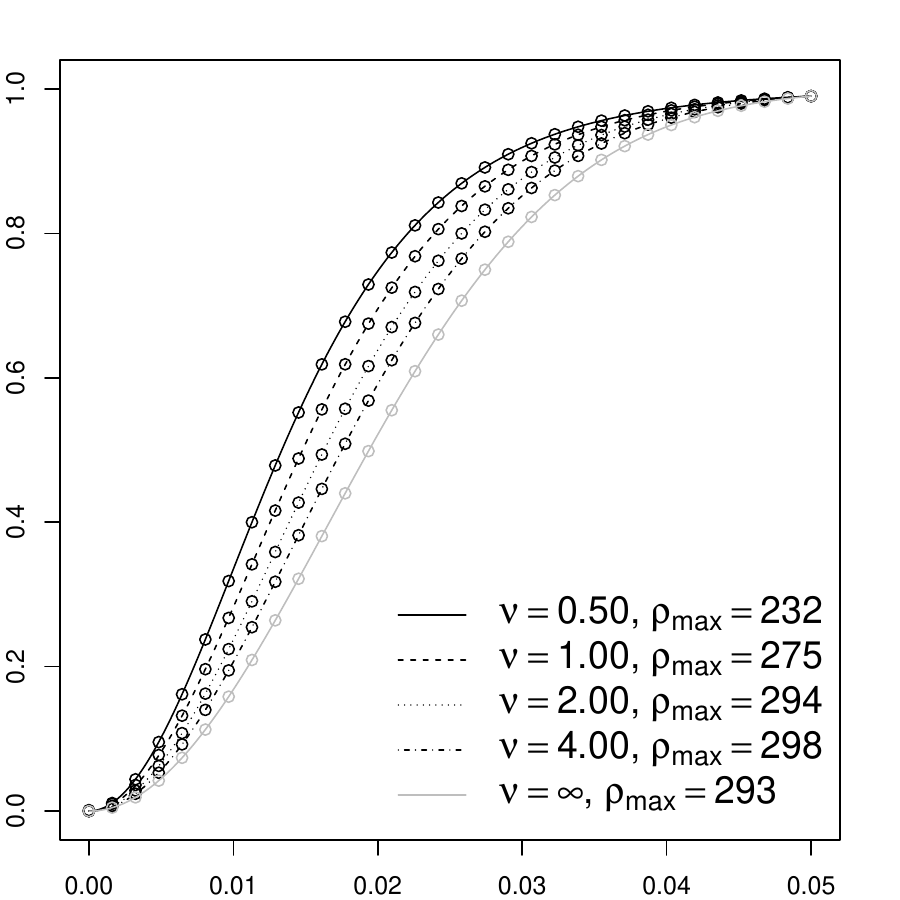}
    \label{subfig:pcf:cauchy}
  }
  \caption{ Isotropic pair correlation functions for
    \subref{subfig:pcf:matern} the Whittle-Mat{\'e}rn model and
    \subref{subfig:pcf:cauchy} the Cauchy model. Each black line
    corresponds to a different value of the shape parameter $\nu$.
    The pair correlation function for the Gaussian
    model ($\nu=\infty$) is shown in gray in both plots. For each
    model, the scale parameter $\alpha$ is chosen such that the range of
    correlation is fixed at $r_0=0.05$, and the
    corresponding value of $\rho_{\max}$ is reported in the legend. The
    circles show values of the approximate isotropic pair correlation function
    obtained by using the approximation $C_{{\mathrm{app}}}$
    described in Section~\ref{sec:approx2}. }
  \label{fig:pcfs}
\end{figure}

\begin{figure}[!htbp]%
  \centering
  \includegraphics[width=.45\textwidth]{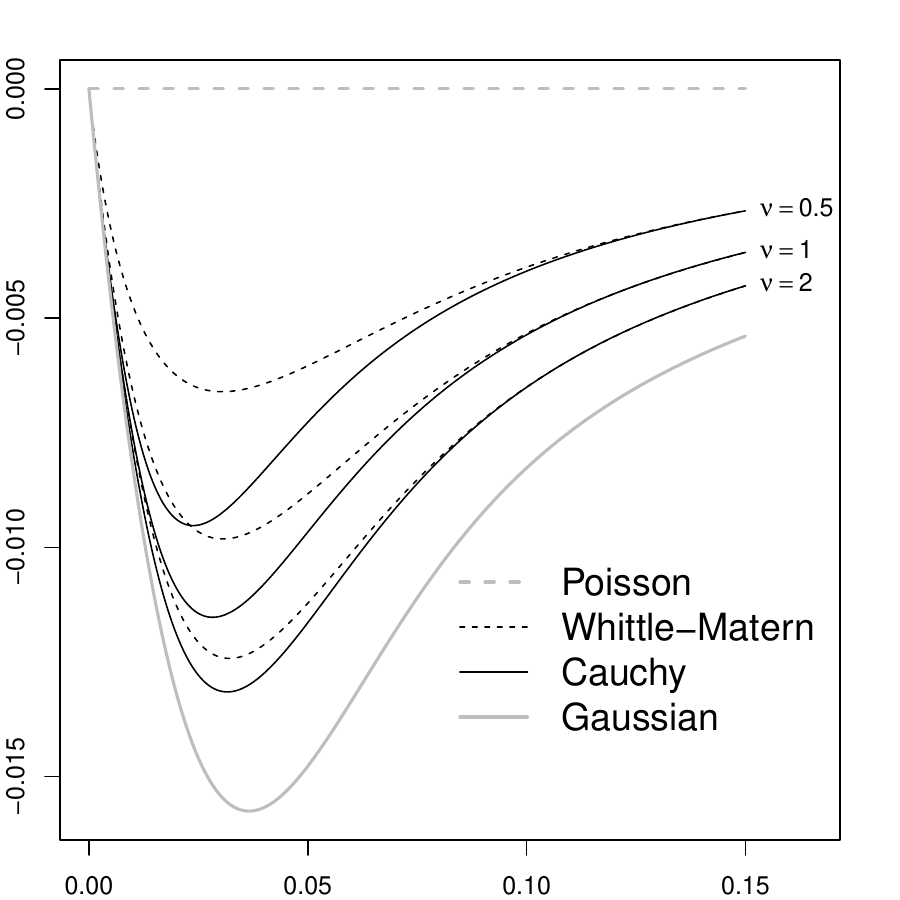}
  \caption{
    Plots of $L(r)-r$ vs.\ $r$ for the \wm{}, Cauchy, and Gaussian
    model with $\alpha=\alpha_{\max}$ and $\rho=100$. For the \wm{}
    and Cauchy models, $\nu\in\{0.5,1,2\}$. The horizontal line at
    zero is $L(r)-r$ for a stationary Poisson process.
  }
  \label{fig:compareLdiff}
\end{figure}

Figure~\ref{fig:pcfs} shows examples of the isotropic pair
correlation functions with a fixed range of correlation. In particular
the Whittle-Mat{\'e}rn DPPs have
several different shapes of pair correlation
functions and so they may
constitute a quite flexible
model class for repulsiveness. 
From the figure it is also evident that the value of
$\rho_{\max}$ is of the same order of magnitude for all these models,
indicating that the range of interaction has a major effect on the
maximal permissible intensity of the model.

Ripley's $K$-function \citep{ripley:76,ripley:77} is for $d=2$ given by
\begin{equation}\label{eq:defKfun}
  K(r) = 2\pi\int_0^r t g_0(t) \dee t, \quad r\ge0,
\end{equation}
and we obtain the following.

  For the Gaussian model: $K(r) = \pi r^2 -
  \frac{\displaystyle \pi\alpha^2}{\displaystyle 2}\left(1-\exp\left(\frac{\displaystyle -2r^2}{\displaystyle \alpha^2}\right)\right). 
 $ 
  
  For the Cauchy model: $K(r) =\pi r^2-\frac{\displaystyle \pi\alpha^2}{\displaystyle 2 \nu+1}
\left(1-\left(\frac{\displaystyle \alpha^2}{\displaystyle \alpha^2+r^2}\right)^{2\nu+1}\right).$

 For the \wm{} model:  The integral in \eqref{eq:defKfun} has to be
evaluated by numerical methods.


We consider the variance stabilizing transformation
of the $K$-function, $L(r)=\sqrt{K(r)/\pi}$ \citep{Besag:77a}, and
recall that $L(r)=r$ for a stationary Poisson process.
Figure~\ref{fig:compareLdiff} shows $L(r)-r$ for seven different
DPPs. Figures~\ref{fig:pcfs} and \ref{fig:compareLdiff} illustrate the dependence between the degree of repulsiveness and
$\nu$, which will be discussed in more detail in Appendix~\ref{sec:quantify}.

\subsection{Spectral approach}\label{sec:spectral}

As an alternative of specifying a stationary covariance function
$C_0$, involving the need for checking positive semi-definiteness and for controlling the upper bound of its Fourier transform, we
may simply specify an integrable function $\varphi:\real^d\to[0,1]$,
which becomes the spectral density, cf.\ Corollary~\ref{cor:first}.
In fact knowledge about $\varphi$
is all we need for the approximate simulation procedure and density
approximation in Section~\ref{sec:approx2}. However, the disadvantage
is that it may then be difficult to determine $C_0=\mathcal
F^{-1}(\varphi)$, and hence closed form expressions for $g$ and $K$ may
not be available. Furthermore, it may be more difficult to interpret
parameters in the spectral domain.

In the following we first describe a general method for constructing
isotropic models via the spectral approach. Second, this method is
used to construct a model class displaying a higher degree of
repulsiveness (in a sense made precise in Appendix~\ref{sec:quantify})
than the Gaussian model which appears as a special case.

Let $f:[0,\infty)\to[0,\infty)$ be any Borel function such that $\sup
f<\infty$ and $0<c<\infty$, where
\begin{equation}\label{eq:specdenconst}
  c = \int_{\real^d} f(\|x\|) \dee x =
  \frac{d\pi^{d/2}}{\Gamma(d/2+1)} \int_0^\infty r^{d-1}f(r)
  \dee r.
\end{equation}
Then we can define the spectral density of a stationary and 
isotropic DPP model as
\begin{equation}\label{eq:spectralapproach}
  \varphi(x) = \rho f(\|x\|)/c, \quad x \in \real^d,
\end{equation}
where $\rho$ is the intensity parameter.
The model is well-defined whenever
\begin{equation}\label{eq:specdenrhomax}
  \rho \le \rho_{\max} = c/\sup f.
\end{equation}
Below we give an example of a parametric model class for such
functions $f$, where the integral in \eqref{eq:specdenconst} and
the supremum in \eqref{eq:specdenrhomax} can be evaluated
analytically.

Assume $Y\sim\Gamma(\gamma, \beta)$ and let $f$ denote the
density of $Y^{1/\nu}$, where $\gamma>0$, $\beta>0$, and $\nu>0$ are
parameters. Let $\alpha=\beta^{-1/\nu}$, then by \eqref{eq:specdenconst} and
\eqref{eq:spectralapproach},
\begin{equation*}
  c = \frac{d\pi^{d/2}\Gamma(\gamma+\frac{d+1}{\nu})}{\Gamma(d/2+1)\Gamma(\gamma)}\alpha^{1-d}
\end{equation*}
and
\begin{equation}\label{eq:specdengengamma}
  \varphi(x) = \rho \frac{\Gamma(d/2+1)\nu\alpha^d}{d\pi^{d/2}\Gamma(\gamma+\frac{d-1}{\nu})}
  \|\alpha x\|^{\gamma\nu-1} \exp(-\|\alpha x\|^\nu).
\end{equation}
We have $\rhomax=0$ if $\gamma\nu<1$, and
\begin{equation}\label{eq:rhomaxgengamma}
  \rho_{\max} = \frac{c}{f((\gamma-1/\nu)^{1/\nu})} =
  \frac{d\pi^{d/2}\alpha^{-d}\Gamma(\gamma+\frac{d-1}{\nu})\exp(\gamma-1/\nu)}{\Gamma(d/2+1)\nu(\gamma-1/\nu)^{\gamma-1/\nu}}\quad\mbox{if
  $\gamma\nu\ge1$.}
\end{equation}
 We call a DPP model with a spectral density of
the form \eqref{eq:specdengengamma} a generalized gamma model. For
$\gamma\nu>1$, the spectral density \eqref{eq:specdengengamma} attains
its maximum at a non-zero value, which makes it fundamentally different
from the other models considered so far where the maximum is
attained at zero.

In the remainder of this section, we consider the special case
$\gamma=1/\nu$, so
\begin{equation}\label{eq:specdenpowerexp}
  \varphi(x) = \rho
  \frac{\Gamma(d/2+1)\alpha^d}{\pi^{d/2}\Gamma(d/\nu+1)}
  \exp(-\|\alpha x\|^\nu).
\end{equation}
We call a DPP model with a spectral density of the form
\eqref{eq:specdenpowerexp} a power exponential spectral model. For $\nu=2$,
this is the Gaussian model of Section~\ref{sec:examples}.

\begin{figure}[!htbp]%
  \centering
  \subfloat[][]{
    \includegraphics[width=.28\textwidth,bb=18 18 410 410]{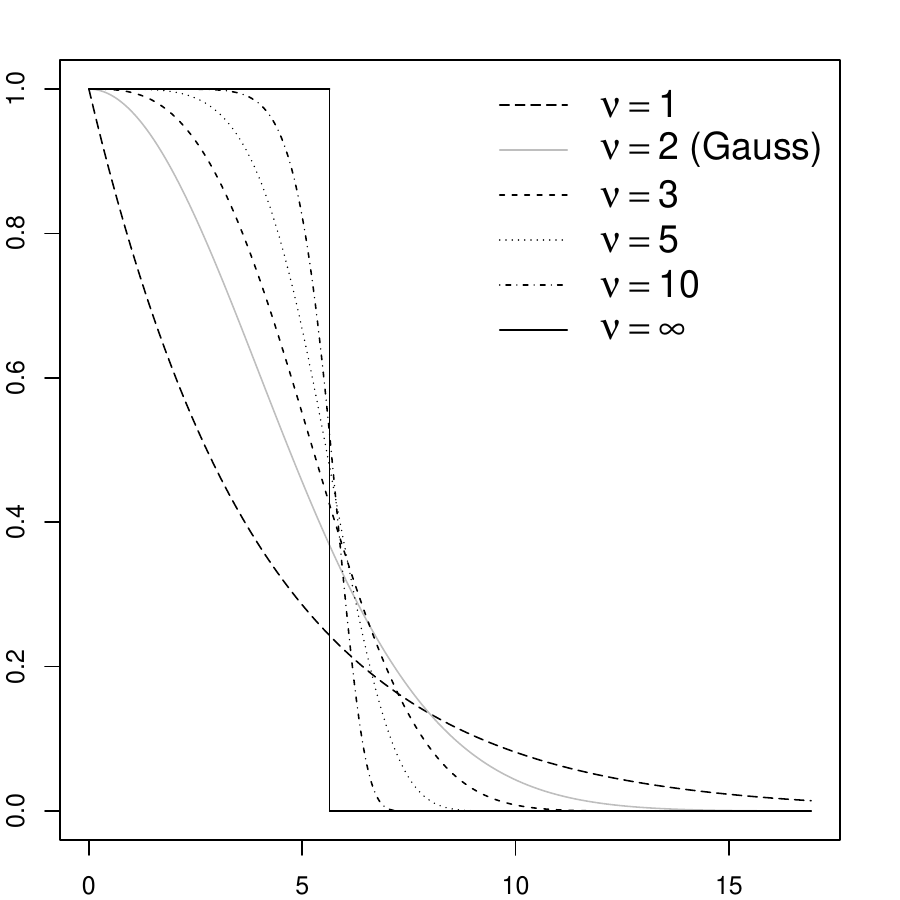}
    \label{fig:gengamma:specden}
  }
  \quad
  \subfloat[][]{
    \includegraphics[width=.28\textwidth,bb=18 18 410 410]{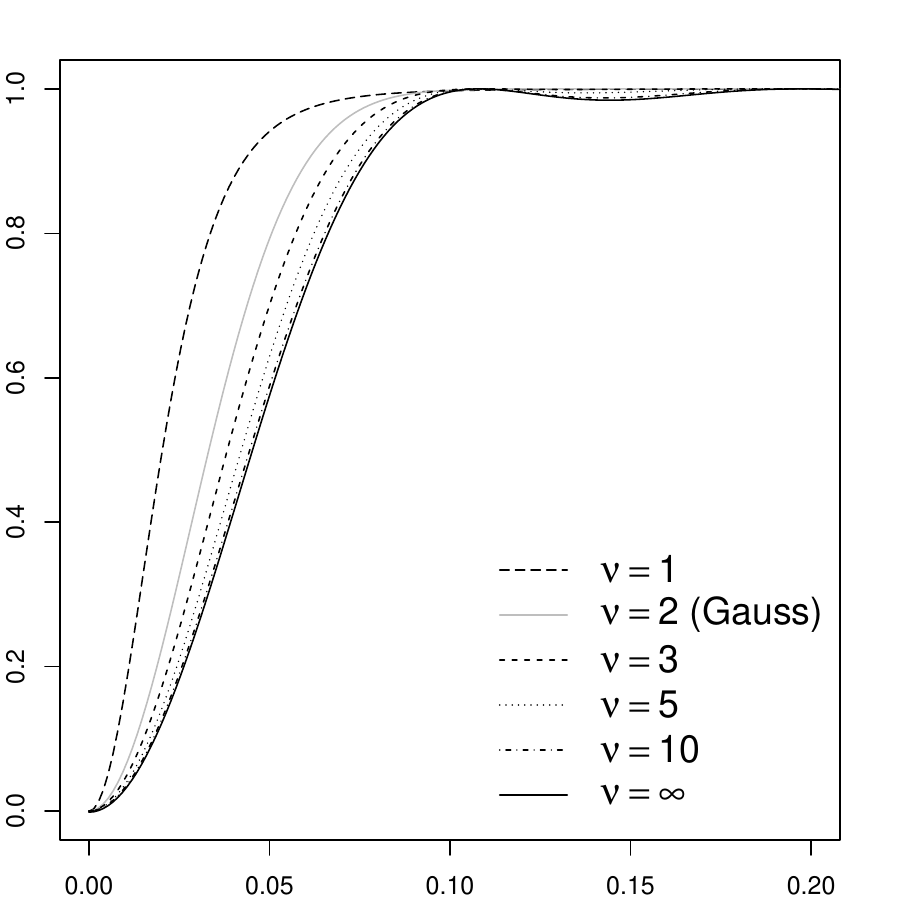}
    \label{fig:gengamma:pcf}
  }
  \quad
  \subfloat[][]{
    \includegraphics[width=.28\textwidth,bb=18 18 410 410]{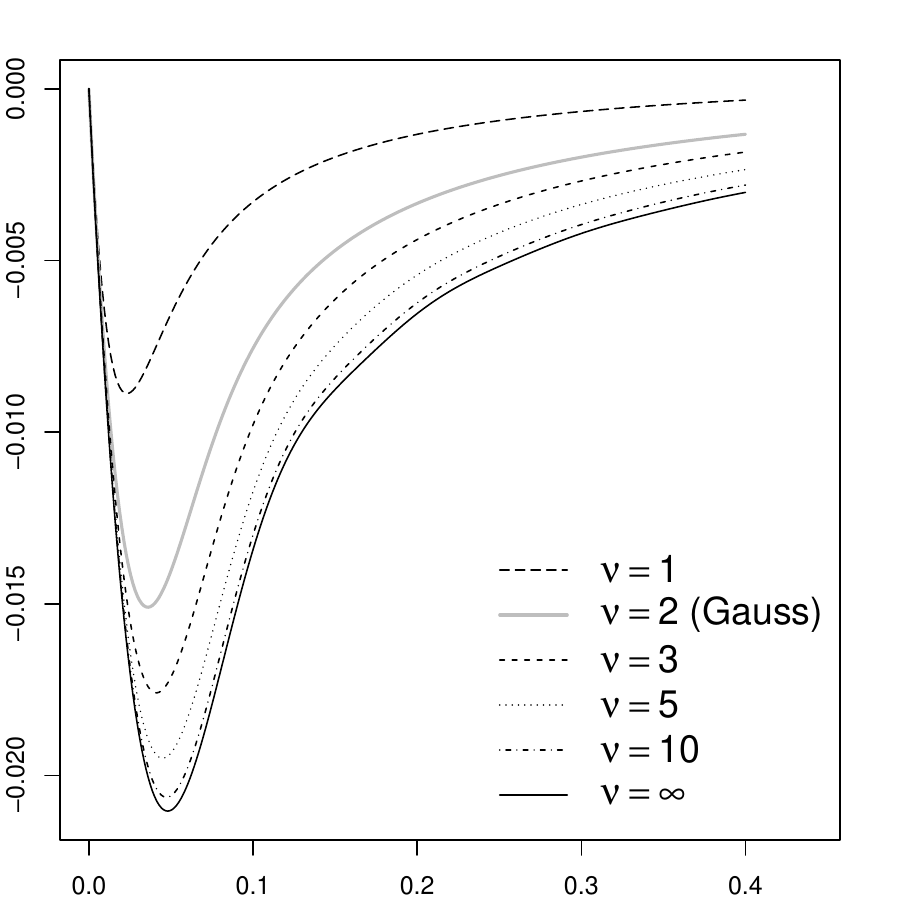}
    \label{fig:gengamma:L}
  }
  \caption{ \subref{fig:gengamma:specden} Isotropic spectral densities,
    \subref{fig:gengamma:pcf} approximate isotropic pair correlation
    functions, and \subref{fig:gengamma:L} approximate $L(r)-r$ functions
    for power exponential spectral models with $\rho=100$,
    $\nu=1,2,3,5,10,\infty$, and $\alpha=\alpha_{\max}$ the maximal
    permissible value determined by \eqref{eq:rhomaxgengamma}.}
  \label{fig:gengamma}
\end{figure}

For the power exponential spectral model,   for fixed $\rho$ and
$\nu$, $\alpha$ has an upper limit $\alpha_{\max}$ given by $\alpha^d_{\max}=\Gamma(d/\nu+1)\tau^{-d}$, where $\tau^d=\rho \Gamma(d/2+1) / \pi^{d/2}$. For the choice $\alpha=\alpha_{\max}$ in
\eqref{eq:specdenpowerexp}, the spectral density of the power
exponential spectral model becomes
\begin{equation}\label{e:527}
\varphi(x)=\exp(- \|\Gamma(d/\nu+1)^{1/d} x/\tau\|^\nu).
\end{equation}
Note that this function tends to the indicator function over the set $\{\|x\|\leq \tau\}$
 as $\nu$ tends to $\infty$. This limiting case corresponds
 to the most repulsive stationary DPP as specified in Appendix~\ref{sec:quantify}.
 For $d=2$, $C_0$ is then
proportional to
a 'jinc-like' function:
\begin{equation}\label{e:2trekant}
C_0(x)=\sqrt{\rho/\pi}\,
J_1(2\sqrt{\pi\rho}\|x\|)/\|x\| \quad \mbox{if $d=2$}.
\end{equation}

Figure~\ref{fig:gengamma} illustrates some properties of the power
exponential spectral
model when $\alpha=\alpha_{\max}$ and $\nu=1,2,3,5,10,\infty$. Recall
that $\nu=2$ is the Gaussian model.
Figure~\ref{fig:gengamma}\subref{fig:gengamma:specden} shows the
spectral density for these cases. Note that
the spectral density approaches an indicator function as
$\nu\to\infty$.
Since we are not aware of a close form
expression for $C_0=\mathcal F^{-1}(\varphi)$ when $\varphi$ is given
by \eqref{e:527},
we approximate $C_0$ by the periodic method discussed in
Section~\ref{sec:approx2}, leading to
approximating pair correlation functions shown in
Figure~\ref{fig:gengamma}\subref{fig:gengamma:pcf}.
Figure~\ref{fig:gengamma}\subref{fig:gengamma:L} shows the
corresponding approximations
of $L(r)-r$ (analogously to Figure~\ref{fig:compareLdiff} in
Section~\ref{sec:examples}).
 Figure~\ref{fig:gengamma} is discussed
in further detail in Appendix~\ref{sec:quantify}.

\section{Approximations}\label{sec:approx2}

Let again $X\sim\detproc(C)$ where $C(x,y)=C_0(x-y)$, cf.\
\eqref{stat-cov}, so that $X$ is stationary. By
Remark~\ref{rem:DPPonS} the restriction of $X$ to a bounded set $S$ is
the finite DPP $X_S \sim DPP_S(C)$. To simulate and evaluate the
density for such a process the spectral representation
\eqref{eq:eigenrep} is needed. Unfortunately, analytic expressions for
such representations are only known in a few simple cases (see e.g.\
\cite{Macchi:75}), which we believe are insufficient to describe the
interaction structure in real datasets. In the following we propose various
approximations which are easy to apply for any $C_0$ with a
known spectral density. Throughout this section we consider
$S=\unitbox$, but we also explain how 
the methods easily generalize to any rectangular set
due to the transformation property \eqref{e:C-trans-again}. 
Section~\ref{sec:approxkernel} concentrates on
approximation of the kernel,
Section~\ref{sec:appsim}
on how to
simulate an approximation of the DPP, and
Section~\ref{sec:approxdensity} on how to approximate the density of
the DPP.

\subsection{Approximation of the kernel $C$}\label{sec:approxkernel}

Denoting as before $\varphi$ the Fourier transform of $C_0$, we
consider $\Xapp{}\sim DPP_S(\Capp)$ where
\begin{equation}\label{e:approximation}
  \Capp(x,y)=\sum_{k\in\Z^d}\varphi(k) \mathrm{e}^{2\pi \mathrm{i}
    k\cdot (x-y)},\quad x,y\in S.
\end{equation} 
This is a spectral decomposition of $\Capp$ on $S$, as defined in
\eqref{eq:eigenrep}, and $\Xapp{}$ is well defined since
$\varphi\leq1$.

As justified in the following, for any $x,y\in S$,
\begin{equation}\label{e:stjerne}
  C(x,y)\approx \Capp(x,y)\quad \text{if }x-y\in S.
\end{equation}
Indeed, if $x-y\in S$, the Fourier expansion of $C_0$ on $S$
yields $C(x,y)=C_0(x-y)= \sum_{k\in\Z^d}\alpha_k\mathrm{e}^{2\pi
  \mathrm{i} k\cdot (x-y)}$ where
\begin{equation}\label{eq:alphadef}
  \alpha_k=\int_{S}C_0(t)\mathrm{e}^{-2\pi \mathrm{i} k\cdot t}\,\mathrm dt.
\end{equation}
Substituting $S$ by $\mathbb R^d$ in this integral, we obtain
$\varphi(k)$. Therefore, if $C_0(t)\approx0$ for $t\in\real^d\setminus
S$, then $\alpha_k\approx\varphi(k)$ and consequently
$C(x,y)\approx\Capp(x,y)$ when $x-y\in S$. This is 
the case for reasonable parameter values of the models, i.e.\ when the
expected number of points in $X_S$ is not very low (see Appendix~\ref{matern-approx} for the \wm{} model). For instance, 
Figure~\ref{fig:pcfs} 
indicates that the
approximation is accurate 
as the
approximate pair correlation functions marked by circles in the plot
are very close to the true curves. 
Furthermore, for covariance functions with finite range $\delta<1/2$
(see \eqref{e:kors}), $C_0(t)=0$ for $t\in\real^d\setminus S$, and so
$C(x,y)=\Capp(x,y)$ if $x-y\in S$, i.e.\
the approximation 
\eqref{e:stjerne} is then exact. For instance, considering the
circular covariance function \eqref{e:circcovfct} and the existence condition
\eqref{eq:circularmax}, we have $\delta<1/2$ if $\rho>16/\pi$, which
indeed is not a restrictive requirement in practice.


If $x,y\in S$ and $x-y\notin S$, then $\Capp(x,y)$ is no longer an
approximation of $C(x,y)$, but rather an approximation of its periodic
extension given by $\sum_{k\in\Z^d}\alpha_k\mathrm{e}^{2\pi \mathrm{i}
  k\cdot (x-y)}$ for all $x,y\in S$.

\subsection{Approximate 
simulation}\label{sec:appsim}

It is straightforward to simulate $\Xapp{}$, since
\eqref{e:approximation} is of the form required for the simulation
algorithm of Section~\ref{sec:sim}. 
Figure~\ref{fig:sim}\subref{subfig:simalg} shows the
acceptance probability for a uniformly distributed proposal (used for
rejection sampling when simulating from one of the densities $p_i$)
when $\Xapp{}$ is simulated. The interior region in the figure is
$S/2=\halfbox$. If $x,y\in S/2$, then $x-y\in S$ and based on the
arguments above we expect that $\Xapp{S/2}$ is a good approximation of
$X\cap S/2$. In practice it turns out that the approximation works
well for the entire region $S$. This may intuitively be explained by
Figure~\ref{fig:sim} where the qualitative behavior of $\Xapp{}$ and
$\Xapp{S/2}$ are similar e.g.\ in the sense that there are regions at
the borders where the acceptance probability is low. For 
$\Xapp{S/2}$, this is due to the influence of points outside
$S/2$. For $\Xapp{}$, this is created artificially by points at the
opposite border. Therefore, we use $\Xapp{}$ as an
approximate simulation of $X_S$. We refer to this approach as the
periodic method of simulation.

More generally suppose we want to simulate $X_R$ where
$R\subset\real^d$ is a rectangular set. We then define an affine
transformation $T(x)=Ax+b$ such that $T(R)=S$. Then $Y=T(X)$ is a
stationary DPP, with kernel given by
\eqref{e:C-trans-again} and spectral density $\varphi_{Y}(x) =
\varphi(A^Tx)$. Let $\Yapp{}$ be the DPP on $S$
with kernel \eqref{e:approximation} where $\varphi$ is replaced by
$\varphi_Y$. Then we simulate $\Yapp{}$ and return $T^{-1}(\Yapp{}\cap
S)$
as an approximate simulation
of $X_R$. We refer to this simulation
procedure as the border method for simulating $X_R$.



\begin{figure}[!htbp]%
  \centering
  \subfloat[][]{
    \includegraphics[width=.45\textwidth]{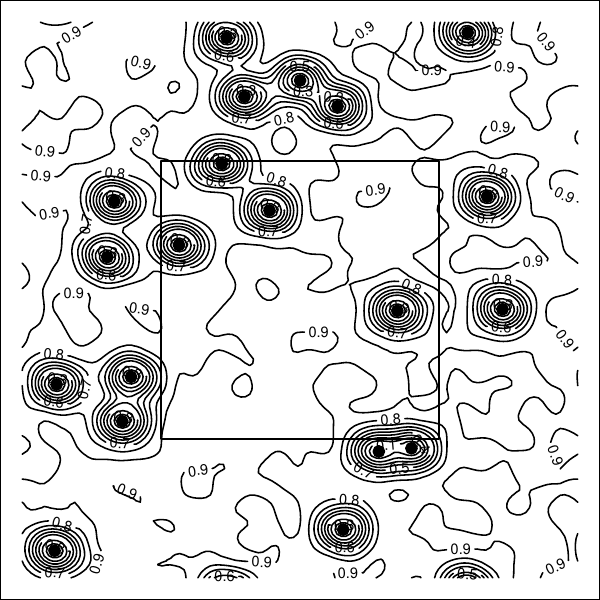}
    \label{subfig:simalg}
  }
  \quad
  \subfloat[][]{
    \includegraphics[width=.45\textwidth,bb=18 18 410 410]{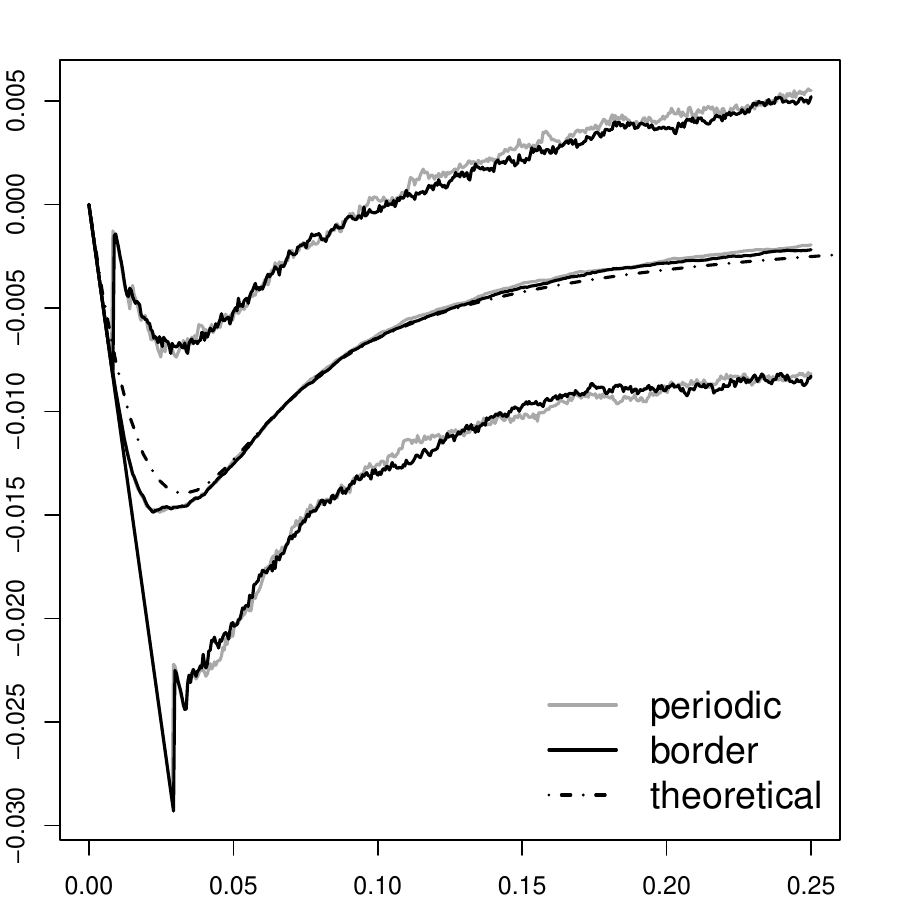}
    \label{subfig:comparebordper}
  }
  \caption{ \subref{subfig:simalg} Acceptance probability for a
    uniformly distributed proposal at an intermediate step of
 Algorithm~1 (Section~\ref{sec:sim}) when 
    simulating a realization of $\Xper{}$ on $S=\unitbox$. The interior box
    is the region $S/2=\halfbox$. The black points
    represent previously generated points, and the acceptance
    probability is zero at these
    points. \subref{subfig:comparebordper} Empirical means and $2.5\%$
    and $97.5\%$ pointwise quantiles of $L(r)-r$ using either the
    periodic method (gray lines) or the border method (black lines), and based on 1000
    realizations of a Gaussian model with $\rho=100$ and
    $\alpha=0.05$. The dashed line corresponds to the theoretical $L(r)-r$ function for this
    Gaussian model. }
  \label{fig:sim}
\end{figure}

An alternative way of approximately
simulating $X_R$ is the border method described as follows.
We simply redefine the affine transformation above such that
$T(R)=S/2$. Then $Y=T(X)$ is again a stationary DPP with spectral
density $\varphi_{Y}(x) = \varphi(A^Tx)$, and $T^{-1}(\Yapp{S/2})$ is an
approximate simulation of $X_R$. While this approximation is
mathematically well founded it is computationally more expensive (it
needs to simulate $2^d$ times more points in average). In
our experience, the periodic method works equally well in practice.
In particular we
have compared the two methods for simulating DPPs with kernels given
by circular covariance functions. In this case the border method
involves no approximation and comparison of plots of the empirical distribution
of various summary statistics revealed almost no difference between
the two methods (these plots are omitted to save space).

For a Gaussian covariance function,
Figure~\ref{fig:sim}\subref{subfig:comparebordper} shows empirical
means and $2.5\%$ and $97.5\%$ pointwise quantiles of $L(r)-r$ using
either the periodic method (gray lines) or the border method (black
lines), and based on 1000 realizations of a Gaussian model with
$\rho=100$ and $\alpha=0.05$. The corresponding curves for the two
methods are in close agreement, which suggests that the two methods
generate realizations of nearly the same DPPs. This was also concluded
when considering other covariance functions and functional summary
statistics (plots not shown here). In
Figure~\ref{fig:sim}\subref{subfig:comparebordper} the
empirical means of $L(r)-r$ are close to the theoretical $L(r)-r$
function for the Gaussian model, indicating that the two approximations of
the Gaussian model are appropriate.

The computational efficiency of the periodic method makes it our
preferred method of simulation. The 1000 realizations used in
Figure~\ref{fig:sim}\subref{subfig:comparebordper} were generated in
approximately three minutes on a laptop with a dual core processor.

\subsection{Approximation of the density 
}\label{sec:approxdensity}

First, consider the density $f$ for $X_S$ as specified in
Theorem~\ref{thm:density} (so we assume $\varphi<1$). We may use the
approximation $f\approx\fapp$, where 
\begin{equation}\label{e:unconddensityper}
  \fapp(\{x_1,\ldots,x_n\}) = \exp(|S|-\Dapp) \det
  [\Ctapp](x_1,\ldots,x_n),\quad \{x_1,\ldots,x_n\}\subset S,
\end{equation}
denotes the density of
$\Xapp{}$, with
\begin{equation}\label{e:(a)}
\tilde\varphi(u) = \varphi(u)/(1-\varphi(u)),\quad u\in S,
\end{equation}
\begin{equation}\label{e:(b)}
\Ctapp(x,y) = \Ctappstat(x-y) = \sum_{k\in\Z^d}
\tilde\varphi(k) \mathrm{e}^{ 2\pi \mathrm{i} k\cdot(x-y)},\quad
x,y\in S,
\end{equation}
\begin{equation}\label{e:(c)}
\Dapp = \sum_{k\in\Z^d} \log\left(1+\tilde\varphi(k)\right).
\end{equation}

Second, consider the density of $X_R$, where $R\subset\real^d$ is
rectangular. Then we use the affine
transformation from above with $T(R)=S$ to define $Y=T(X)$. If $\fapp_Y$ 
denotes the approximate density of $Y$ as specified by the
right hand side of \eqref{e:unconddensityper}, we can approximate
the density of $X_R$ by
\begin{equation*}
\fapp(\{x_1,\dots,x_n\}) =
|R|^{-n}\exp(|R|-|S|)\fapp_Y (T(\{x_1,\dots,x_n\})),\quad
\{x_1,\ldots,x_n\}\subset R.
\end{equation*}

We call $\fapp$ the periodic approximation of $f$. 
To evaluate this
approximation in practice we need to use truncated versions of 
\eqref{e:(b)} and \eqref{e:(c)}. For a given integer $N>0$ (the choice of $N$ is discussed
in Section~\ref{sec:likest}), let
$\Z_N=\{-N,-N+1,\dots,N-1,N\}$ and define
\begin{equation}\label{eq:Dtrunc}
  D_N = \sum_{k\in\Z_N^d} \log(1+\tilde\varphi(k))
\end{equation}
and
\begin{equation}\label{eq:ctildetrunc}
  \Ct_N(u) = \sum_{k\in\Z_N^d} \tilde\varphi(k) \mathrm{e}^{ 2\pi
    \mathrm{i} k \cdot u},\quad u\in\real^d.
\end{equation}
While it is feasible to evaluate \eqref{eq:Dtrunc} for large values of
$N$, the evaluation of \eqref{eq:ctildetrunc} is more problematic
since it needs to be carried out for every distinct pair of points in
$\{x_1,\dots,x_n\}$. For moderate $N$ (few hundreds) direct calculation of
\eqref{eq:ctildetrunc} can be used. In this case, we can exploit the
fact that $\tilde\varphi$ often is an even function (corresponding to
a real-valued $C_0$) such that all imaginary parts in
\eqref{eq:ctildetrunc} cancel. This allows to reduce the number of
terms in the sum by a factor $2^d$, which speeds up calculations
considerably when evaluating the approximate density. For large $N$
(hundreds or thousands) we use the fast Fourier transform (FFT) of $\tilde\varphi$. The FFT
yields values of $\Ct_N$ at a discrete grid of values and we
approximate $\Ct_N(x_i-x_j)$ by bilinear interpolation.
The simulation
study in Section~\ref{sec:hom} shows that likelihood inference based
on $\fapp$ works well in practice for the examples in this
paper. Appendix~\ref{sec:approx3} introduces a convolution
approximation of the density which in some cases may be
computationally faster to evaluate. However, as discussed in
Appendix~\ref{sec:approx3},
this approximation appears to be poor in some situations and
in general we prefer the periodic approximation.

\section{Inference for stationary models}\label{sec:hom}

In this section, we discuss and exemplify how to estimate parameters of
stationary
DPP models and how to
do model comparison.  Section~\ref{sec:likest} focuses on maximum
likelihood based inference, Section~\ref{sec:nonlikest} considers
alternative ways of performing inference,
and Section~\ref{sec:simstudy} discusses a simulation study of the
approaches of Sections~\ref{sec:likest}-\ref{sec:nonlikest}.  Examples
of the estimation and model comparison procedures when modelling real
datasets are given in Sections~\ref{sec:spanish}-\ref{sec:termites} where we also discuss
model checking. Furthermore, in Sections~\ref{sec:likest1} and
\ref{sec:simstudy}, we discuss the commonly used non-parametric
intensity estimate $\hat\rho=n/|S|$ in comparison to the maximum
likelihood estimate (MLE) of $\rho$.


Throughout this section we assume that $\{x_1,\dots,x_n\}$ is a realization of $X \sim \detproc(C)$
restricted to a compact set $S$, where $C(x,y)=C_0(x-y)=\rho R_0(\|x-y\|)$ is modelled by one of
the parametric models of
Sections~\ref{sec:examples}-\ref{sec:spectral}, namely the Gaussian,
\wm{}, Cauchy, and power exponential spectral models --- for short we
refer to these as the four parametric models of DPPs. Recall that
$\rho$ is the intensity parameter, $\theta$ denotes the parameter of
the correlation function $R_0$, and a given value of $\rho$ introduces a bound on
the parameter space for $\theta$ which is denoted $\Theta_\rho$. 

\subsection{Maximum likelihood based inference}\label{sec:likest}



Eventhough it is feasible to estimate $\rho$ by the MLE, for
computational reasons we prefer the estimate $\hat\rho=n/|S|$. This
choice is discussed in Section~\ref{sec:likest1}. 
Further, we only consider the
(approximate) MLE for $\theta$ as the value that maximizes the
approximate log-likelihood, i.e.\ the truncated version of the density $\fapp$
in \eqref{e:unconddensityper},
\begin{equation*}
  \ell_N(\theta) = \log\det[\Ct_N](x_1,\ldots,x_n) - D_N, \quad \theta\in\Theta_{\hat\rho},
\end{equation*}
where $[\Ct_N](x_1,\ldots,x_n)$ is the $n\times n$ matrix with
$(i,j)$'th element $\Ct_N(x_i-x_j)$. Here we suppress in the
notation that $\Ct_N$ and $D_N$ depend on $(\rho,\theta)$ through
$\tilde\varphi$, cf.\ \eqref{eq:Dtrunc}-\eqref{eq:ctildetrunc}.
 If $\theta$ is one
dimensional, the maximum of $\ell_N(\theta)$ can be determined by a simple
search algorithm, otherwise the simplex algorithm by
\cite{Nelder:Mead:65} can be used. Note that these methods do not require
explicit knowledge of the derivatives of $\ell_N(\theta)$.

Concerning the choice of $N$, note that the sum
\begin{equation*}
  S_N = \sum_{k\in\Z_N^d} \varphi(k)
\end{equation*}
tends to $\rho$ from below as $N$ tends to infinity. Hence, for any
value of $\theta$, one criterion for choosing $N$ is to require
e.g.\ $S_N>0.99\hat\rho$. However, this may be insufficient as $N$
also determines the grid resolution when FFT is used, and a high
resolution may be required to obtain a good approximation of the
likelihood. Therefore, in the FFT case we use increasing values of $N$
until the approximate MLE stabilizes.

When comparing the four parametric models fitted to the
same dataset,
we prefer the model with the largest value of
$\ell_N(\theta)$. The comparison of $\ell_N(\theta)$ between different models is valid, since the dominating measure is the same.

\subsubsection{MLE for the intensity}\label{sec:likest1}

Rather than fixing the estimate of the intensity to $\hat\rho=n/|S|$,
we may estimate both $\rho$ and $\theta$ by maximum likelihood. This has
been done for the simulated Gaussian model given in
Section~\ref{sec:simstudy}, where we observed that the MLE of $\rho$ is
very close to $n/|S|$. This has further been done for each DPP model
fitted to the real datasets in Sections~\ref{sec:spanish}-\ref{sec:termites}, where the
largest relative difference between the non-parametric estimate and
the MLE of $\rho$ was $4\%$.

When $\rho$ is not too close to $\rhomax$,
the fact that the MLE appears to be close to
 $n/|S|$ may be
understood in the following
way. By applying the convolution approximation in
Appendix~\ref{sec:convolutionapprox}, rough
approximations $\tilde C(x,y)\approx C_0(x-y)$ and $D\approx|S|\rho$
 are obtained by considering only the first terms
in \eqref{eq:Ctilde0conv} and \eqref{eq:Dconv}. Hence a rough
approximation of the log-likelihood is
\[\ell(\rho,\theta;\{x_1,\ldots,x_n\})\approx -|S|\rho + n\log\rho +
\log \det [C^\dagger](x_1,\ldots,x_n) \]
where $C^\dagger(x,y)=C_0(x-y)/\rho$ depends only on $\theta$ and not
on $\rho$. The maximum point of this approximate
log-likelihood has $\rho=n/|S|$.

On the other hand, for very repulsive DPPs, the number of points has a
small variance, and so we may expect the intensity to be close to the
observed $n/|S|$. In particular, the approximation to the most
repulsive DPP is a determinantal projection process and the observed
intensity is then non-random.

\subsection{Alternative approaches for inference}\label{sec:nonlikest}

Given a parametric DPP model there are several feasible approaches for
inference which are not based on maximum likelihood. For example,
parameter estimation can be based on composite likelihood, Palm likelihood,
generalized estimating equations, or minimum contrast methods. See
 \cite{moeller:waagepetersen:07}, \cite{Prok:Jensen:10},
and the references therein.
Here we
only briefly recall how the minimum contrast estimate (MCE)
\citep{Diggle:Gratton:84} is calculated.

Let $s(r;\theta)$, $r\ge0$, denote a functional summary statistic for
which we have a closed form expression, where
$\theta\in\Theta_{\hat\rho}$ and $\hat\rho=n/|S|$. In our examples, this will be
either the pair correlation function $g$ or the $K$-function. Further,
let $\hat s(r)$ be a non-parametric estimate of $s$ based on the data
$\bx$. The MCE based on the functional summary statistic $s$ is the
value of $\theta$ which minimizes
\begin{equation*}
  D(\theta) = \int_{r_l}^{r_u} |\hat s(r)^q - s(r;\theta)^q|^p \dee r
\end{equation*}
where 
$r_l<r_u$, 
$p>0$, and $q>0$ are user-specified parameters. Following the
recommendations in \cite{Diggle:03}, we let $q=1/2$, $p=2$, and
$r_u$ be one quarter of the minimal side length of $S$. It is
customary to use $r_l=0$ and we do this when the MCE is based on the
$K$-function. However, when the MCE is based on $g$, we let $r_l$ be one
percent of the minimal side length of $S$ 
to avoid numerical
instabilities of the estimate of $g$ close to zero.
To minimize $D(\theta)$
 we use the same method as was used for maximizing
$\ell_N(\theta)$ in Section~\ref{sec:likest}, which avoids the use of
derivatives of $D(\theta)$.

Finally, when several different models are
fitted to the same dataset, the one with minimal value of $D(\theta)$
is preferred.

\subsection{Simulation study}\label{sec:simstudy}

We have generated 500 realizations in the unit square of the following
five models: Gaussian, \wm{} with $\nu=0.5$, \wm{} with $\nu=1$,
Cauchy with $\nu=0.5$, and Cauchy with $\nu=1$. For all models,
$\rho=200$ and $\alpha=\alpha_{\max}/2$, where $\alpha_{\max}$ is
given by \eqref{eq:alphamax}. In our experience it is difficult to
identify the parameters $\nu$ and $\alpha$ simultaneously, which is a
well-known issue for the \wm{} covariance function \citep[see
e.g.][]{Lindgren:etal:11}. Here we consider $\nu$ known such that the
remaining parameter to estimate is one dimensional, i.e.\
$\theta=\alpha$.

Table~\ref{tab:esttotal} provides the empirical means and standard
deviations of the MCE based on $K$, the MCE based on $g$, and the
MLE, where for each model,
the MLE is calculated for several different values of $N$.
In general, we see that as long as the truncation is sufficiently
large the MLE outperforms the MCE since the former
 has smaller biases and smaller
standard deviations.

\begin{table}[ht]
\begin{center}
  \caption{Empirical means and standard deviations (in parentheses) of
    parameter estimates based on 500 simulated datasets for each of 5
    different models with intensity $\rho = 200$. Model 1: Gauss;
    Model 2: \wm{} ($\nu=0.5$); Model 3: \wm{} ($\nu=1$); Model 4:
    Cauchy ($\nu=0.5$); Model 5: Cauchy ($\nu=1$). The columns from
    left to right are: The true value of $\alpha$, MCE based on the
    $K$-function, MCE based on $g$, MLE with $N=256$, MLE with
    $N=512$, MLE with $N=1024$, and MLE with $N=2048$. All entries are
    multiplied by 100 to make the table more compact.}
\label{tab:esttotal}
\resizebox{\textwidth}{!} {
\begin{tabular}{rlllllll}
  \hline
 & $\alpha$ & $K$ & $g$ & MLE256 & MLE512 & MLE1024 & MLE2048 \\ 
  \hline
1 & 2.00 & 2.05 (0.58) & 1.99 (0.51) & 1.42 (0.25) & 2.01 (0.43) & 2.01 (0.43) & 2.01 (0.43) \\ 
  2 & 1.40 & 1.59 (0.88) & 1.48 (0.92) & 1.77 (0.11) & 1.62 (0.56) & 1.55 (0.63) & 1.52 (0.67) \\ 
  3 & 1.00 & 1.02 (0.46) & 0.95 (0.54) & 0.97 (0.18) & 1.00 (0.36) & 1.00 (0.37) & 1.00 (0.37) \\ 
  4 & 1.40 & 1.48 (0.68) & 1.30 (0.87) & 1.39 (0.23) & 1.37 (0.54) & 1.38 (0.54) & 1.38 (0.55) \\ 
  5 & 2.00 & 2.07 (0.83) & 1.91 (0.97) & 1.69 (0.29) & 2.01 (0.61) & 2.01 (0.62) & 2.02 (0.61) \\ 
   \hline
\end{tabular}}
\end{center}
\end{table}

The quality of the likelihood approximation is closely related to the
decay rate of the spectral density of the model, or equivalently to
the rate of convergence of $S_N$. Figure~\ref{fig:truncsum} shows
$S_N$ for different values of $N$ for each of the five models. It is
clear that the two \wm{} models approach the theoretical limit
$\rho=200$ at a slower rate than the other models, and this makes the
likelihood approximation inaccurate for small $N$ leading to bias in
the estimates shown in Table~\ref{tab:esttotal}.

\begin{figure}[!htbp]%
  \centering
    \includegraphics[scale=.45]{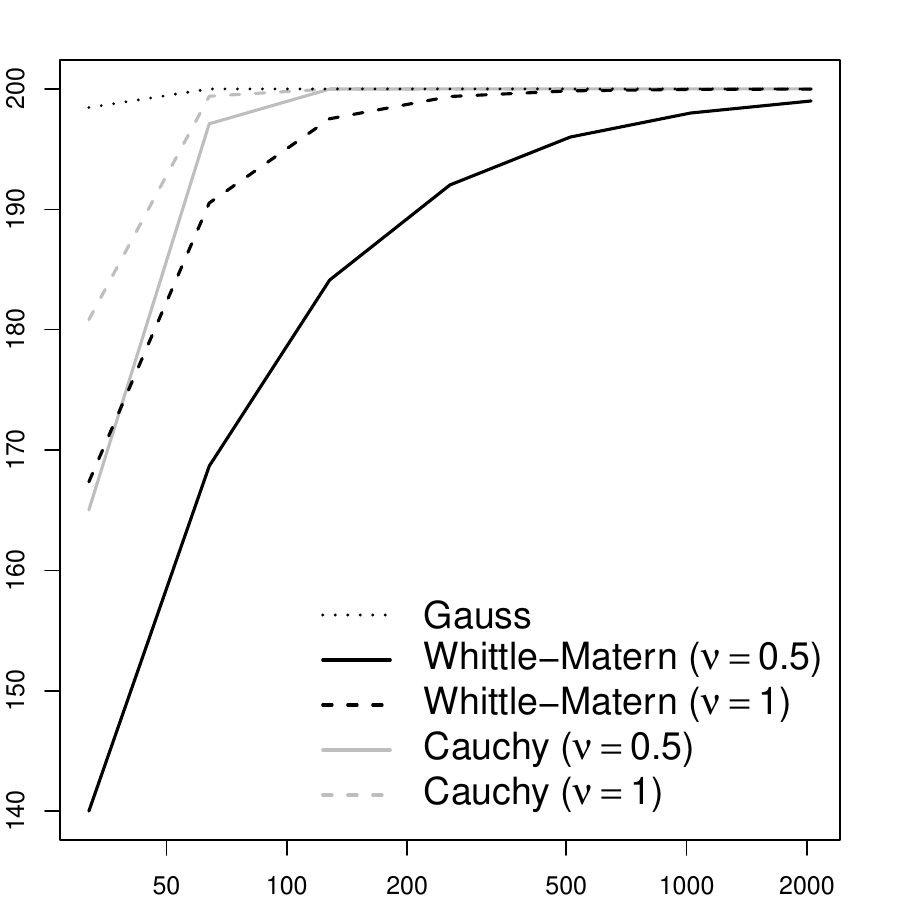}
  \caption{ $S_N$ as a function of $N$. }
  \label{fig:truncsum}
\end{figure}

For the Gaussian model above, we have tried to include $\rho$ as a
freely varying parameter when maximizing the likelihood with
$N=2048$. For each realization the MLE of $\rho$ was very close to the
non-parametric estimate $n/|S|$, and the largest relative difference
between the two estimates of $\rho$ for the 500 realizations was
$0.3\%$. This fits well with Section~\ref{sec:likest1} where heuristic arguments
suggest that the MLE for $\rho$ is close to $n/|S|$.

\subsection{Spanish towns dataset}\label{sec:spanish}

The Spanish towns dataset (Figure~\ref{fig:data}\subref{subfig:data:spanish}) was first analysed in
\cite{Glass:Tobler:71}. In a subsequent analysis, \cite{Ripley:88}
used a Strauss hard-core model specified by four parameters: A
hard-core distance $\hc$, an interaction distance $R$, an abundance
parameter $\beta$, and an interaction parameter $\gamma$. The MLE for
$\hc$ is the minimal observed distance, $r_{\min}$, but more commonly
$\frac{nr_{\min}}{n+1}$ is used. Estimation of $R$ is typically based
on an ad-hoc method such as maximum profile pseudo-likelihood (see
e.g.\ \cite{moeller:waagepetersen:00}). Conditionally on $\hc$ and
$R$, the parameters $\beta$ and $\gamma$ can e.g.\ be estimated by the
maximum pseudo-likelihood method or much more computationally
demanding Markov chain Monte Carlo methods to approximate the
likelihood can be used. Following an analysis in
\cite{illian:penttinen:stoyan:stoyan:08}, we use their estimates
$\hcest=0.83$ and $\hat{R}=3.5$. Further, we estimate $\beta$ and
$\gamma$ using the approximate likelihood method of
\cite{Huang:Ogata:99} available in \texttt{spatstat}. The estimates
are $\hat{\beta}=0.13$ and $\hat{\gamma}=0.48$.

Figure~\ref{fig:spanishenvelopes} is used to assess the
goodness of fit for the Strauss hard-core model. The dashed central
lines show non-parametric estimates of $L(r)-r$, the nearest neighbour distribution function
$G(r)$, and 
the empty space function $F(r)$ (for definitions of $F$ and $G$, see e.g.\
\cite{moeller:waagepetersen:00}). The plots in the top row also show $2.5\%$ and $97.5\%$ pointwise quantiles (gray lines) for these
summary statistics based on 4000 simulations of the fitted Strauss hard-core
model. Overall the model appears to provide
an acceptable fit, but the characteristic cusp of the envelopes of
$L(r)-r$ at $r=\hat{R}=3.5$ seem to be a somewhat artificial model effect that
the dataset does not exhibit (see also Example~3.14 in
\cite{illian:penttinen:stoyan:stoyan:08}). The plots in the bottom row
show the rank envelopes of \cite{myllymaki:etal:13} using the
authors' R package \texttt{spptest} with a
significance level of 5\%. In contrast to the pointwise envelopes
these are global envelopes and if the non-parametric estimate exits
the envelopes the deviation is significant at the 5\% level. The rank
envelope test also yields an interval for the $p$-value, and the
$p$-value intervals for all three summary
statistics are above 5\% (which corresponds to
the non-parametric estimates staying within the envelopes).

As an alternative to the Strauss hard-core model, we now consider the
four parametric classes of DPP models. The intensity estimate is
$\hat{\rho}=0.043$, and the fitted \wm{} model (with
$\hat\nu=2.7$ and $\hat\alpha=0.819$) has the highest value of the
likelihood and it is therefore preferred over the Cauchy and power
exponential spectral models which also have three parameters.
 As the Gaussian model has only two parameters and is a
(limiting) special case of the \wm{} model, we carry out a simulation based
likelihood-ratio test as follows. Using 400 simulated realizations
under the fitted Gaussian model (with $\hat{\alpha}=2.7$), we fit both the Gaussian model and
the (alternative) \wm{} model. Then, for each sample, we evaluate
$D=-2\log(Q)$ where $Q$ is the ratio of the Gaussian and the \wm{}
likelihoods. We finally compare the distribution of $D$ over the 400
simulated realizations with the observed value of $D$ for the
dataset. The resulting $p$-value is 0.03 and we reject the Gaussian
model in favor of the \wm{} model. Again
Figure~\ref{fig:spanishenvelopes} is used to assess the goodness of
fit which appears to be quite good as none of the non-parametric
estimates exit neither the 95\% pointwise envelopes or the 5\% rank
envelopes obtained from simulations
under the fitted \wm{} model.

\begin{figure}[!htbp]%
  \centering
  \includegraphics[width=.92\textwidth]{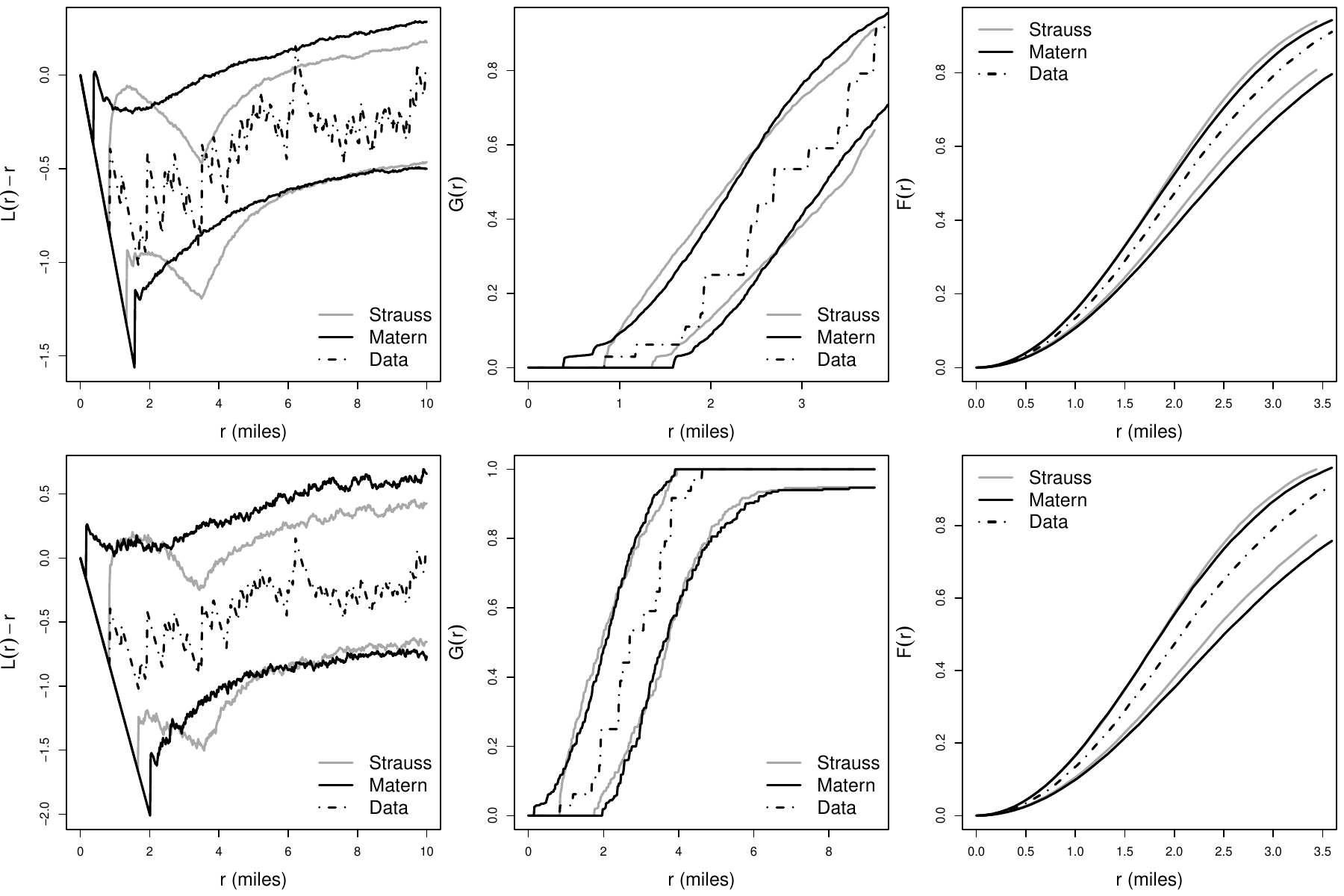}
  \caption{ Left to right: Non-parametric estimate of $L(r)-r$, $G(r)$,
    and $F(r)$ for the Spanish towns dataset with simulation based
    envelopes for both the fitted Strauss hard-core model (gray lines)
    and the fitted \wm{} model (black lines). Top row: The envelopes
    are $2.5\%$ and $97.5\%$ pointwise quantiles. Bottom row: Rank
    envelopes with a $5\%$ significance level. For both models the
    envelopes are based on 4000 simulated realizations. }
  \label{fig:spanishenvelopes}
\end{figure}

In summary the \wm{} model both has less parameters and
arguably provides a better fit than the Strauss hard-core model. Furthermore, we
have direct access to the moments of the \wm{} model such as the
intensity and the pair correlation function which can only be obtained
by simulation for the Strauss hard-core model.

\subsection{Hamster cells dataset}\label{sec:hamster}

Figure~\ref{fig:data}\subref{subfig:data:hamster} shows a plot of the
locations of 303 cells of two types in a \SI{0.25}{\milli\metre} by
\SI{0.25}{\milli\metre} region of a histological section of the kidney
of a hamster. The data has been rescaled to a unit square and there
are 226 dividing (living) cells marked by circles and 77 pyknotic
(dying) cells marked by pluses. The dataset was analysed in
\cite[][Section 6.4.1]{Diggle:03} where it was concluded to be in
agreement with independent labelling of a simple sequential inhibition
(SSI) process with hardcore distance $\delta=0.0012$. As noted by
Diggle, this model is not strictly valid, since there are a few pairs
of data points violating the hardcore condition. However, Diggle
considered the good overall fit as an indication that the SSI model,
together with random labelling of cell types, provides a reasonable
approximate description of the data.

Under the assumption of random labelling, the two sub-point patterns
consisting of respectively the dividing and the pyknotic cells
correspond to respectively the retained and thinned points of an independent
thinning of the full unmarked point pattern. Using a stationary DPP model with
kernel $C_0$ for the
full unmarked point pattern, this
implies that each individual
sub-point pattern should follow the same type of DPP model with
different values of $\rho_1$ and $\rho_2$ for the intensities (with
$\rho=\rho_1+\rho_2$), and
kernels $(\rho_1/\rho) C_0$ and $(\rho_2/\rho) C_0$, cf.\ Proposition~\ref{prop:thinning}.
We may exploit this to test the hypothesis of random labelling: first
we fit
a parametric class of DPP
models to the full unmarked point pattern, second we fit the same
model class to each of the sub-point patterns. If
random labelling is true, all fitted models should coincide, up to the
intensity.

For the full unmarked point pattern, all four of the parametric
classes of DPPs
 fit well (judged by envelopes of summary statistics not shown here) with
very similar values of the approximate likelihood. Simulation based
likelihood ratio tests cannot reject
the null hypothesis of the Gaussian model against alternatives given
by either the \wm{}, Cauchy, or power exponential spectral model. We
therefore use the simpler fitted Gaussian model with estimates
$(\hat\rho,\hat\alpha)=(303,0.0181)$ for the full unmarked point
pattern, $(\hat\rho_1,\hat\alpha_1)=(226,0.0188)$ for the
dividing cells, and $(\hat\rho_2,\hat\alpha_2)=(77,0.00816)$ for the
pyknotic cells.  The relevant
hypotheses to check for random labelling, based on the thinning
characterization explained above, is thus
$H_0$: $\alpha_1=\alpha_2=\alpha$ against $H_1$:
  $\alpha_1\neq\alpha$ or $\alpha_2\neq\alpha$.
Several test statistics can be proposed to perform this test.  We
choose to base our decision on
$\Pi=|\hat\alpha-\hat\alpha_1||\hat\alpha-\hat\alpha_2|$. The
distribution of $\Pi$ under $H_0$ is evaluated from 1000 realizations
of a Gaussian DPP model with $(\rho,\alpha)=(303,0.0181)$. On one hand,
we fit a Gaussian DPP model to each realization to get an estimate of
$\alpha$, and on the other hand, we apply an independent thinning with
retention probability $\hat\rho_1/\hat\rho=0.75$ and fit a Gaussian
DPP model to both the retained points and the thinned points to obtain
estimates of both $\alpha_1$ and $\alpha_2$. The distribution of $\Pi$
over the 1000 simulations is compared to the empirical value of $\Pi$
for the dataset. The resulting $p$-value is $0.55$
and there is no reason to reject the null hypothesis.

\subsection{Oak and beech trees dataset}\label{sec:oakbeech}

Figure~\ref{fig:data}\subref{subfig:data:oakbeech} shows a plot of the
locations of 244 trees of the species oak and beech. The dataset
originates from \cite{Pommerening:02} and was analysed in
\cite{Mecke:Stoyan:05} and
\cite{illian:penttinen:stoyan:stoyan:08}. In the following analysis we
ignore the species type. In this case \cite{Mecke:Stoyan:05} noted
that only powerful tests can reject the null hypothesis of a
homogeneous Poisson model.

Among the four parametric classes of DPPs, the \wm{}
model with $\hat{\rho}=n/|S|=0.038$, $\hat \nu=0.4$, and
$\hat\alpha=2.28$ has the highest likelihood and shows a reasonable
fit, cf.\ Figure~\ref{fig:oakbeech-summarystat}. The estimate
$\hat\nu=0.4$ could indicate that we are close to the Poisson
model. However, when we performed a simulation based likelihood-ratio
test similar to the one in Section~\ref{sec:spanish}, all 400 
simulated values of the test
statistic were below the observed value, so the null hypothesis of a
homogeneous Poisson model is clearly rejected.

\begin{figure}[!htbp]%
  \centering
  \includegraphics[width=.95\textwidth]{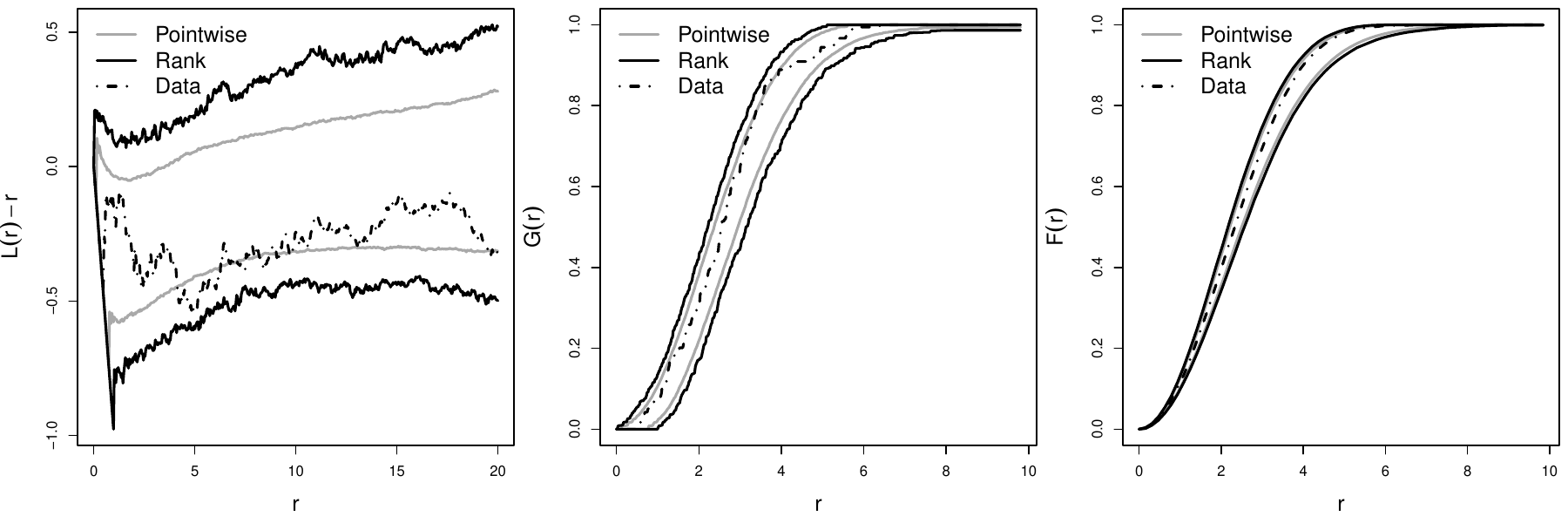}
  \caption{ Left to right: Non-parametric estimate of $L(r)-r$, $G(r)$
    and $F(r)$ for the oak and beech dataset with simulation based
    envelopes for the fitted \wm{} DPP model. The envelopes are
    $2.5\%$ and $97.5\%$ pointwise quantiles (gray lines) and rank
    envelopes with a $5\%$ significance level (black lines). Both sets
    of envelopes are based on 4000 simulated realizations.}
    \label{fig:oakbeech-summarystat}
\end{figure}

\subsection{Termite mounds dataset}\label{sec:termites}

Figure~\ref{fig:data}\subref{subfig:data:termites} shows a plot of the
locations of 48 termite mounds. The full dataset also contains an
associated pattern of palm tree locations which is omitted here. The
dataset originates from \cite{barot:etal:99} and was analysed in
\cite{illian:penttinen:stoyan:stoyan:08} where the focus was on the
interaction between palms and termite mounds. The palm locations were
modelled conditional on the mound locations assuming that the palm
locations constitute a cluster process with the mounds as cluster
centers. It may therefore be interesting to model and analyse the
locations of the mounds separately to know which process generated the
cluster centers.

When we fit the four parametric classes of DPPs, the power exponential
spectral DPP model has the highest likelihood and
an estimate $(\hat\rho,\hat\alpha,\hat\nu)=(0.00128,46.8,6.1)$. This is on the
border of the parameter space and the fitted model is close to the jinc-like DPP
with kernel \eqref{e:2trekant}. The fitted model is
judged to provide a good fit, see
Figure~\ref{fig:termites-summarystat}. We also performed a simulation based likelihood-ratio
test similar to the ones in
Sections~\ref{sec:spanish}-\ref{sec:oakbeech} based on 1000 
simulations from the fitted Gaussian model and the resulting p-value
was 2.6\%, so the null hypothesis of the simpler Gaussian model is rejected.


\begin{figure}[!htbp]%
  \centering
  \includegraphics[width=.95\textwidth]{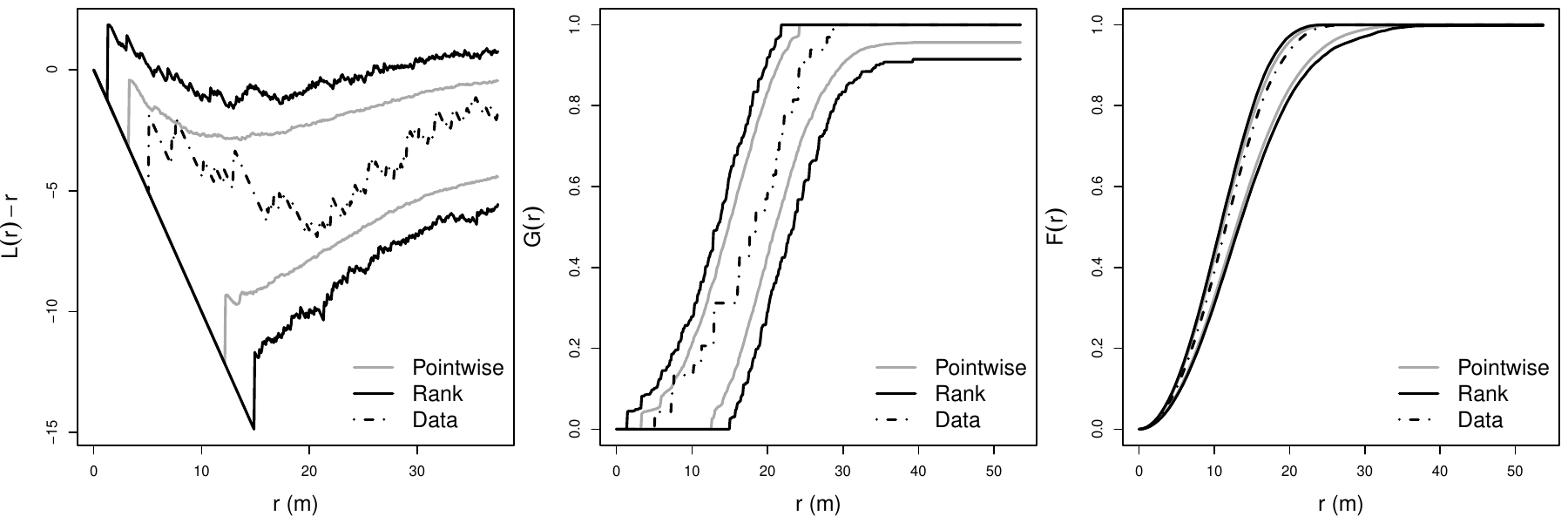}
  \caption{ Left to right: Non-parametric estimate of $L(r)-r$, $G(r)$
    and $F(r)$ for the termites dataset with simulation based
    envelopes for the fitted power exponential spectral model. The envelopes are
    $2.5\%$ and $97.5\%$ pointwise quantiles (gray lines) and rank
    envelopes with a $5\%$ significance level (black lines). Both sets
    of envelopes are based on 4000 simulated realizations.}
    \label{fig:termites-summarystat}
\end{figure}

\section{Inference for non-stationary models}\label{sec:inhom}


There are at least two different possible strategies for constructing
a non-sta\-ti\-o\-na\-ry DPP model. In Section~\ref{sec:inhom-reweighted}, we
assume the DPP is second-order intensity-reweighted stationary
\citep{baddeley:moeller:waagepetersen:97} which allows us to devise a
quite generally applicable strategy of analysis. Alternatively
we could construct a non-stationary DPP model by transforming a
stationary DPP, in which case \eqref{e:C-trans-again} gives a closed
form expression for the kernel of the non-stationary DPP. This
approach is exemplified in Section~\ref{sec:mucous}, but it is very specific to
the analysed dataset and it appears to be difficult to devise a
general strategy for modelling non-stationarity this way. In both
approaches, we exploit that the kernel (covariance function) $C$ of a
DPP can be written as
\begin{equation}\label{eq:C-R}
  C(x,y)=\sqrt{\rho(x)} R(x,y) \sqrt{\rho(y)}
\end{equation}
where $\rho$ is the intensity function and $R$ is the corresponding
correlation function to $C$.

\subsection{Second-order intensity-reweighted stationary
  models}\label{sec:inhom-reweighted}

In this section, we use the Japanese pines dataset (Figure~\ref{fig:data}\subref{subfig:data:Japanese}) from \cite{numata:64} to illustrate how
a non-stationary DPP model can be fitted to a point pattern dataset in
practice. The data consists of the locations of 204 seedlings and
saplings of Japanese black pines (\emph{Pinus Thunbergii}) in a
\SI{10}{\metre} by \SI{10}{\metre} region. The data has previously
been analysed by among others \cite{ogata:tanemura:86} using an
inhomogeneous Gibbs model where the logarithm of the first order term
was assumed to be a cubic polynomial in the Cartesian
coordinates. Since we have no access to external covariates (such as
soil quality, type of terrain, etc.) we use the Cartesian coordinates
as an artificial way of accounting for spatial heterogeneity. However,
the methodology described below works equally well for datasets with
external covariates.

Let $X$ be a DPP observed on a compact set $S$, with intensity function $\rho$
and correlation function $R$. Then $X$ is second-order
intensity-reweighted stationarity if $R$ is invariant by
translation, since this implies that $g$ is translation invariant,
cf.\ \eqref{eq:pcf}. If we further assume $R$ is isotropic, we have
$R(x,y)=R_0(\|x-y\|)$. Let $\rho_\text{b}$ be an upper bound of
$\rho(x)$, $x\in S$, and assume a stationary isotropic DPP
$\Xdom$ with kernel
\begin{equation*}
    \Cdom(x,y)=\Cdomstat(\|x-y\|)=\rho_\text{b} R_0(\|x-y\|)
\end{equation*}
is well defined. By Proposition~\ref{prop:thinning}, $X$
corresponds to an independent thinning of $\Xdom$ with retention
probability $\rho(x)/\rho_\text{b}$. A parametric model for $X$ 
can then be
constructed using the parametric models of
Sections~\ref{sec:examples}-\ref{sec:spectral} for $R_{0,\theta}$
(i.e.\ when $R_0$ is parametrized by $\theta$) and
specifying a parametric model $\rho_\psi$ for $\rho$, which possibly
depends on spatial covariates. We estimate $\psi$ by the Poisson
maximum likelihood estimator $\hat\psi$ \citep{Schoenberg05}. 
Using
$\rho_{\hat\psi}$ for reweighting, we can non-parametrically estimate
the pair correlation function
\citep{baddeley:moeller:waagepetersen:97}.
Then we estimate
$\theta$ by the minimum contrast estimate $\hat\theta$ based on the
pair correlation function $g_{0,\theta}(r)=1-|R_{0,\theta}(r)|^2$,
$r\ge0$, cf.\ Section~\ref{sec:nonlikest}.

Specifically, for the Japanese pines dataset, we
assume $\log\rho_\psi$ is a cubic polynomial in the Cartesian
coordinates. This is in close analogy with the analysis of
\cite{ogata:tanemura:86}, but we have the advantage of modelling
the intensity directly, while the intensity is unknown in their
analysis. The left panel in Figure~\ref{fig:japanese-reweighted} shows
the fitted intensity function with the points of the dataset
overlayed.
\begin{figure}[!htbp]%
  \centering
    \includegraphics[width=\textwidth]{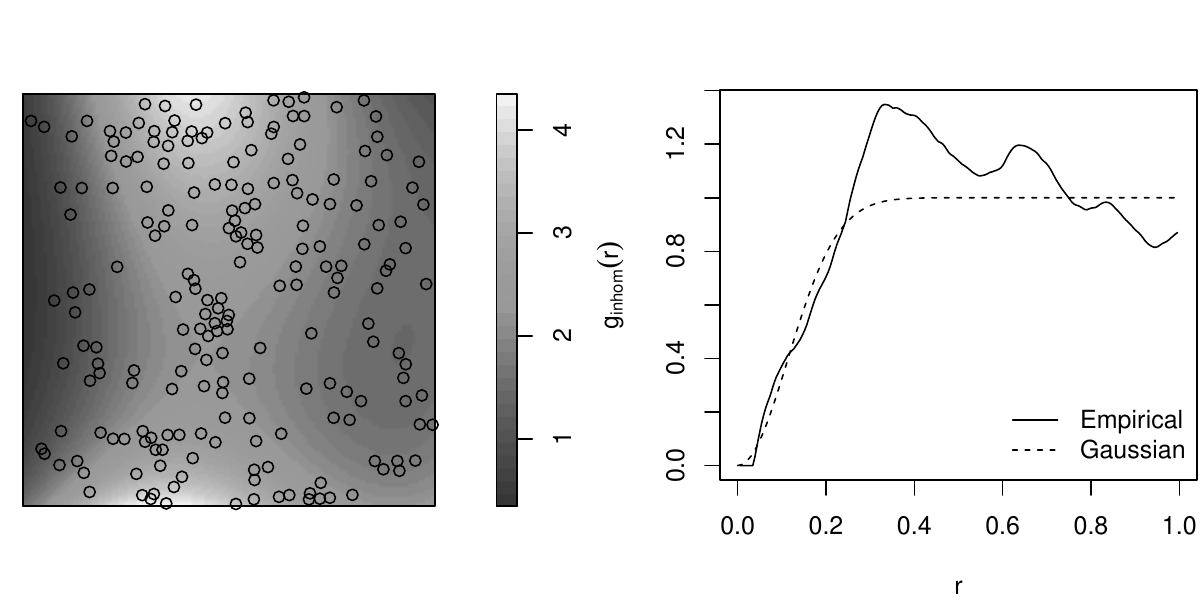}
    \caption{ Japanese pines dataset. Left: Estimated intensity
      function $\rho_{\hat\psi}$ with the points of the dataset
      overlayed. Right: Empirical pair correlation function and theoretical
      pair correlation function for the fitted Gaussian DPP model.}
  \label{fig:japanese-reweighted}
\end{figure}
Further, given $\rho_{\hat\psi}$, we define $\hat\rho_\text{b} =
\text{sup}_{x\in S}\{ \rho_{\hat\psi}(x)\}$. Then the existence
condition \eqref{e:C0-cond} for $\Xdom$ with kernel $\hat\rho_\text{b}R_{0,\theta}$
defines the valid parameter space $\Theta_{\hat\rho}$. 

Since we do not have a closed form expression for the pair correlation
function of the power exponential spectral model, we have omitted that
model from this analysis. The fitted Gaussian
 model has $\hat\alpha=0.226$ and the right panel in Figure~\ref{fig:japanese-reweighted}
 shows the fitted pair correlation function.
 The estimates for both the \wm{}
 and Cauchy models have very large values of $\nu$ indicating
 that the fitted model is close to the Gaussian model; in fact plots
 of the fitted pair correlation functions (not shown here) coincide with the one for
 the Gaussian model. Also the values of $D(\theta)$ are very similar for
 all three models. Therefore we prefer the simpler Gaussian model.

\begin{figure}[!htbp]%
  \centering
  \includegraphics[width=.95\textwidth]{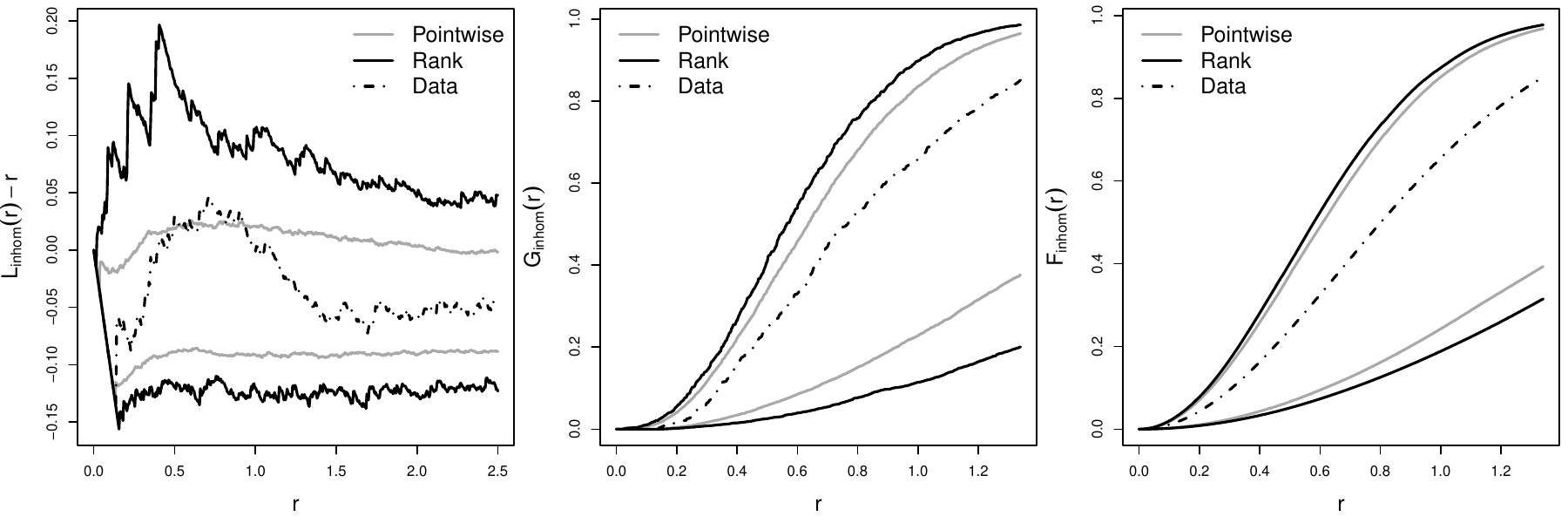}
  \caption{ Left to right: Non-parametric estimate of inhomogeneous
    versions of $L(r)-r$, $G(r)$
    and $F(r)$ for the Japanese pines dataset with simulation based
    envelopes for the fitted Gaussian DPP model. The envelopes
    are $2.5\%$ and $97.5\%$ pointwise quantiles (gray lines) and rank
    envelopes with a $5\%$ significance level (black lines). Both sets
    of envelopes are based on 4000 simulated realizations.}
  \label{fig:japaneseenvelopes}
\end{figure}

Similarly to
Figures~\ref{fig:spanishenvelopes}-\ref{fig:termites-summarystat}
 in Section~\ref{sec:hom}, 
Figure~\ref{fig:japaneseenvelopes} is used to assess the
goodness of fit for the Gaussian model. 
Here we use the inhomogeneous versions of the summary statistics, see
\cite{moeller:waagepetersen:00} and \cite{Lieshout:11}. 
Since the estimated intensity is used when estimating the summary
statistics, to ensure the validity of the envelope test,
it is important that the intensity is reestimated for each
simulation. 
Overall the model appears to provide
an acceptable fit. The estimate of $L(r)-r$ exits the
pointwise envelopes but stays within the rank envelopes, so this
deviation is not significant at the 5\% level.
 
\subsection{Inhomogeneity by transformation}\label{sec:mucous}


The mucous membrane dataset shown in
Figure~\ref{fig:data}\subref{subfig:data:mucosa} consists of the most
abundant type of cell in a bivariate point pattern analysed in
\cite{moeller:waagepetersen:00}. 
For this dataset we propose another way of handling inhomogeneity in a DPP
model. Assume $X$ is a planar DPP with kernel $C$ and observed within
a rectangular window $W=[0,A]\times[0,B]$, where
$A>0$ and $B>0$. Suppose $X_W$ has a separable intensity function: $\rho(x_1,x_2)=\rho_1(x_1)\rho_2(x_2)$,
where $\rho_1$ (resp.\ $\rho_2$) is positive and integrable on $[0,A]$
(resp.\ on $[0,B]$).  Let $T_1(x_1)=\int_0^{x_1} \rho_1(u)\,\mathrm du$,
$0\le x_1\le A$, and $T_2(x_2)=\int_0^{x_2} \rho_2(u)\,\mathrm du$, $0\le x_2\le
B$. Define $T(x_1,x_2)=(T_1(x_1), T_2(x_2))$ for $(x_1,x_2)\in W$. The transformed point process
$Y=T(X_W)$ is a DPP defined on $T(W)=[0,T_1(A)]\times[0,T_2(B)]$ and its kernel $C_Y$ is deduced
from Proposition~\ref{prop:trans}:
\begin{equation}\label{eq:C_Y}
C_{Y}(x,y) = R(T^{-1}(x), T^{-1}(y)), \quad x,y\in T(W),
\end{equation}
with $R$ as in \eqref{eq:C-R}.  For $x\in W$, $Y$ has intensity
$C_{Y}(x,x)=R(T^{-1}(x), T^{-1}(x))=1$, since
$R$ is a correlation function. Now, assume that $Y$ corresponds to
the restriction of a stationary DPP $Z$ to $T(W)$, i.e.\ $Y=Z_{T(W)}$
and for $x,y\in \mathbb R^2$, $C_Y(x,y)= C_{Y,0}(y-x)$. Thus by
\eqref{eq:C_Y}, 
for $x,y\in W$, $R(x,y) =
C_{Y,0}(T(y)-T(x))$ and the kernel of $X_W$ follows from \eqref{eq:C-R}:
\begin{equation}\label{CX}
C(x,y)=\sqrt{\rho(x)}\;  C_{Y,0}(T(y)-T(x)) \; \sqrt{\rho(y)}, \quad
x,y\in W.
\end{equation}
Note that $X_W$ has pair correlation function
$g(x,y)=1-|C_Y(T(y)-T(x))|^2$. This implies in particular that $X$ is
not second-order intensity-reweighted stationary.

In summary we fit an inhomogeneous DPP $X_W$ with separable
intensity as follows.
\begin{itemize}
\item Fit the intensity function $\rho$ restricted to $W$ (e.g.\
  by the Poisson maximum likelihood estimator).
\item Apply the transformation $T$ introduced above to obtain
  $Y=T(X_W)$.
\item Fit a stationary DPP model for $Z$ based on $Y=Z_{T(W)}$,
  and use \eqref{CX} to obtain the kernel of $X_W$.
\end{itemize}

We apply this procedure to the mucous membrane dataset, where $A=1$
and $B=0.81$. Considering
Figure~\ref{fig:data}\subref{subfig:data:mucosa} it seems reasonable
to assume horizontal homogeneity, i.e.\
$\rho(x_1,x_2)=\rho_1\rho_2(x_2)$ with $\rho_1$ a positive
constant. We simply model $\rho_2$ as piecewise constant on the nine
intervals $[0.09(i-1),0.09i)$, $i=1,\dots,9$ (though $\rho$ is then
not continuous, $X_W$ becomes a DPP according to
Definition~\ref{def1}). 
Thus $T_1$ is linear and
$T_2$ is piecewise linear making the transformation $T$ very
simple. Note that we can choose any positive value for $\rho_1$ as
this choice just amounts to rescaling $\rho_2$. We fix
$\hat\rho_1=\sqrt{n}/A$, where $n=876$ is the number of points,
whereby $T_1$ is determined.  If $T_2$ determined by an estimate of
$\rho_2$ satisfies $T_1(A)T_2(B)=n$, then $T(W)=[0,\sqrt{n}]^2$. We
therefore estimate $\rho_2$ on each interval by the frequency of
points with first coordinate in the interval divided by $0.09\sqrt
n$. This gives the estimates $34, 59, 46, 39, 34, 36, 32, 34,
16$. The left panel of 
Figure~\ref{fig:mucous-rhopcf} 
shows the fitted piecewise constant intensity $\hat\rho$. The dataset
transformed by $T$ is shown in
Figure~\ref{fig:mucous} 
and is
modelled as the restriction of a stationary DPP using each of the four
parametric classes of DPPs.

\begin{figure}[!htbp]%
  \centering
    \includegraphics[width=.45\textwidth]{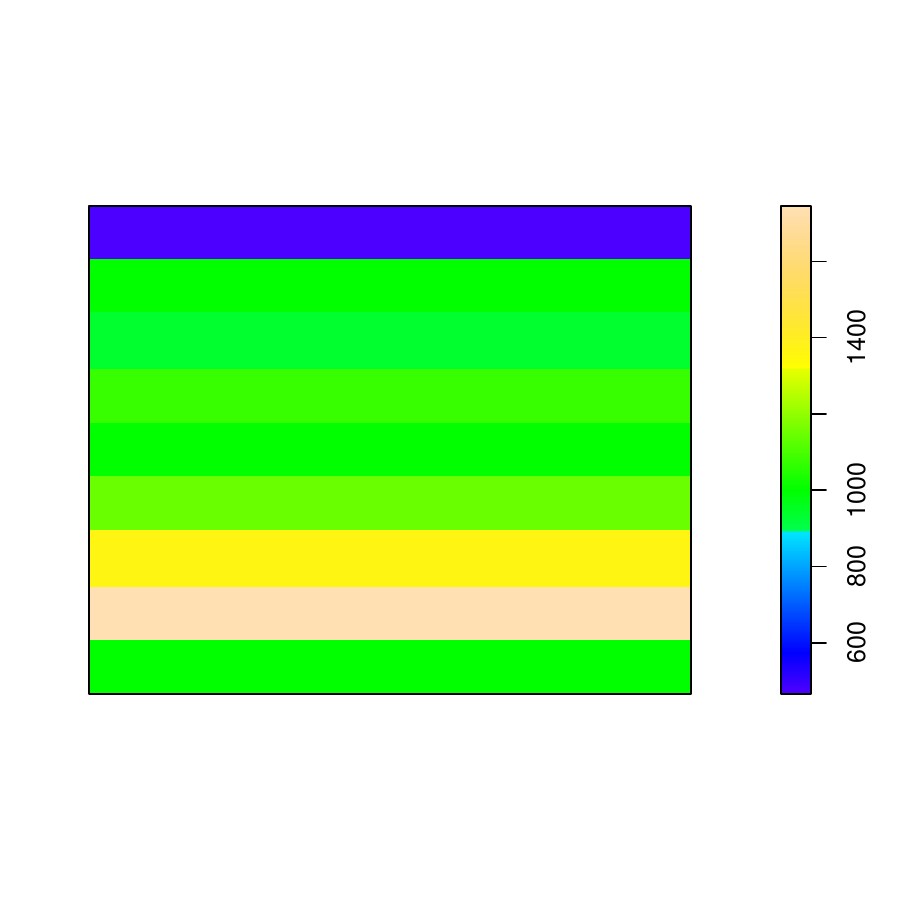}
  \quad
    \includegraphics[width=.45\textwidth]{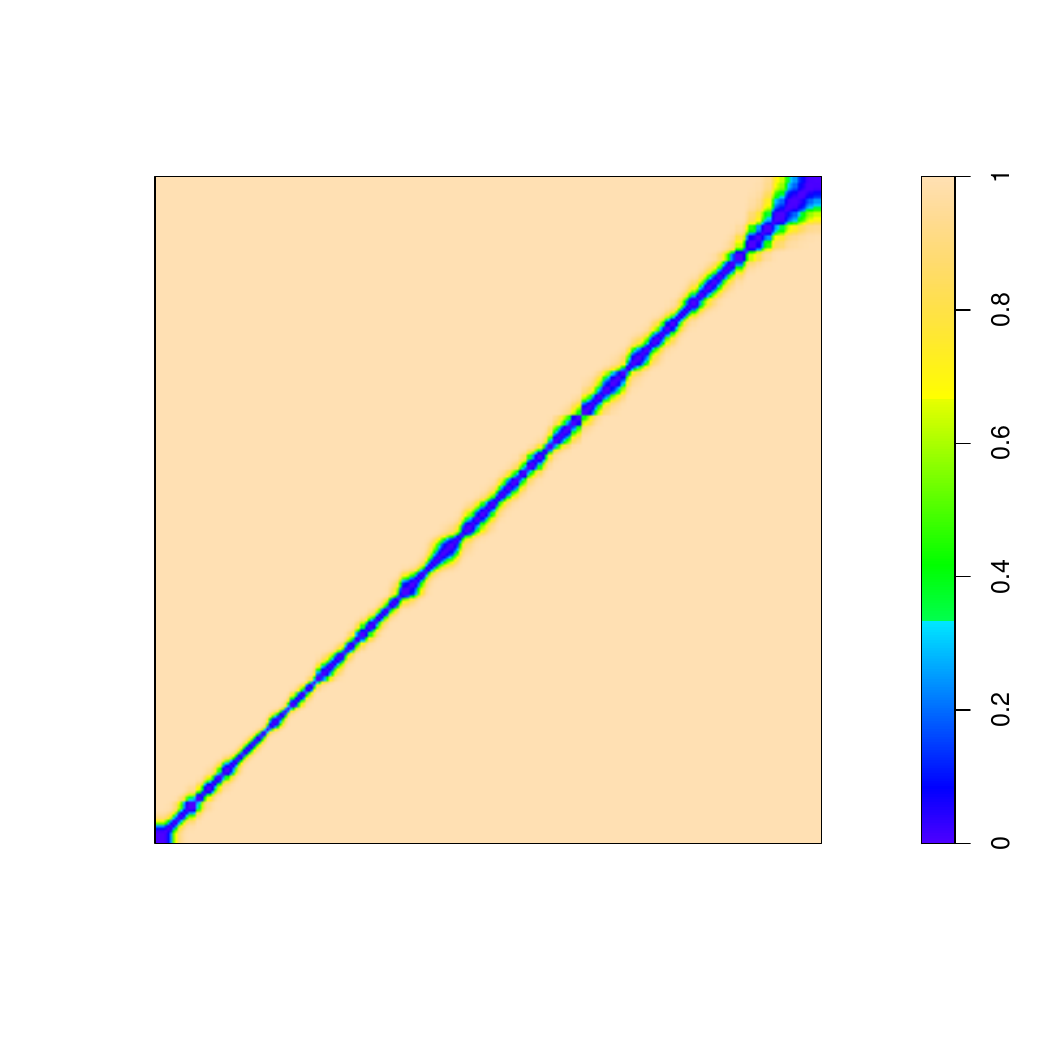}
  \caption{ Mucous membrane dataset. Left: 
    Fitted piecewise constant intensity function $\hat\rho$.
Right: Fitted pair correlation function
    with $x_2$ along the abscissa, $y_2$ along the
    ordinate, and $y_1-x_1=0$ (see the text).}
  \label{fig:mucous-rhopcf}
\end{figure}

The fitted power exponential spectral model has the highest likelihood
of the four models, but the likelihood value is only slightly larger
than for the Gaussian model. The simulated $95\%$ envelopes in
Figure~\ref{fig:mucous-summary} 
indicate that these two models are very close in terms of the
considered functional summary statistics. A simulation based
likelihood-ratio test comparing the Gaussian null model with the power
exponential spectral model yielded a $p$-value of 0.10, and we thus
prefer the Gaussian model for $Z$ with fitted parameter $\hat
\alpha=0.48$ (and intensity one as imposed by the transformation
$T$). However, the non-parametric estimate of $L(r)-r$ exits the rank
envelope, indicating a significant lack of fit at the 5\% level (the
departure from the envelope is around $r=0.25$, and may be difficult
to see in the figure).

\begin{figure}[!htbp]%
  \centering
    \includegraphics[width=.45\textwidth,bb=36 36 395 395]{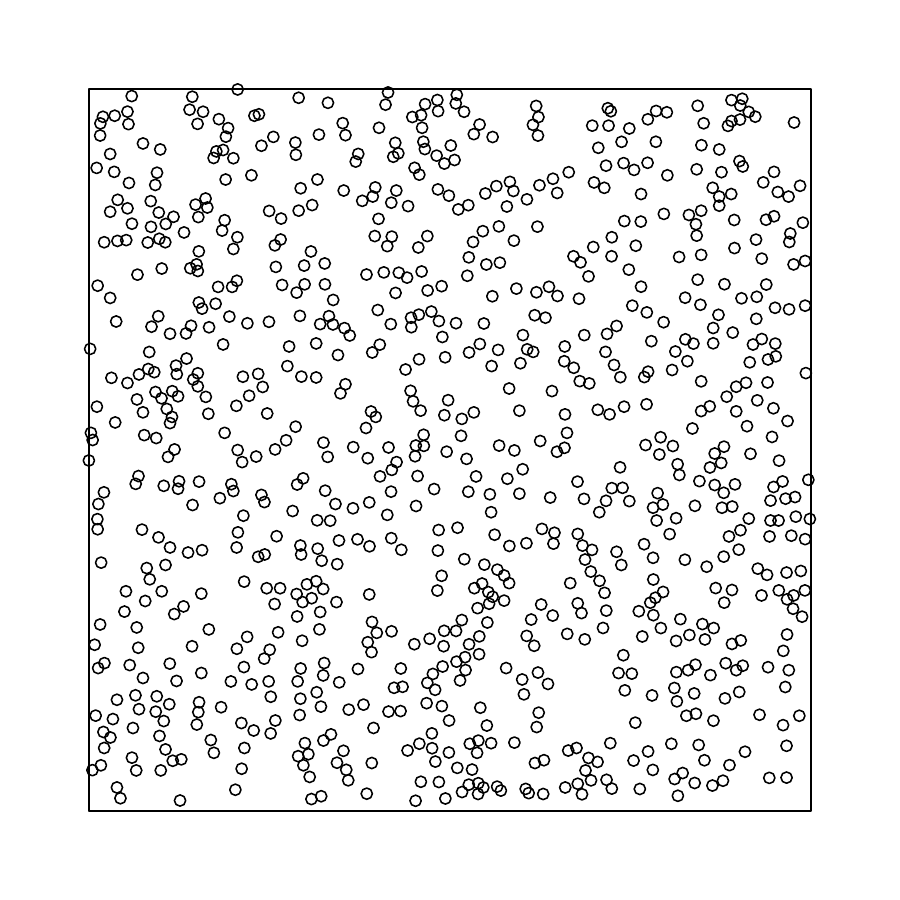}
  \caption{ 
Mucous membrane dataset in $[0,\sqrt{n}]^2$ after transformation.
  }
  \label{fig:mucous}
\end{figure}

\begin{figure}[!htbp]%
  \centering
    \includegraphics[width=.95\textwidth]{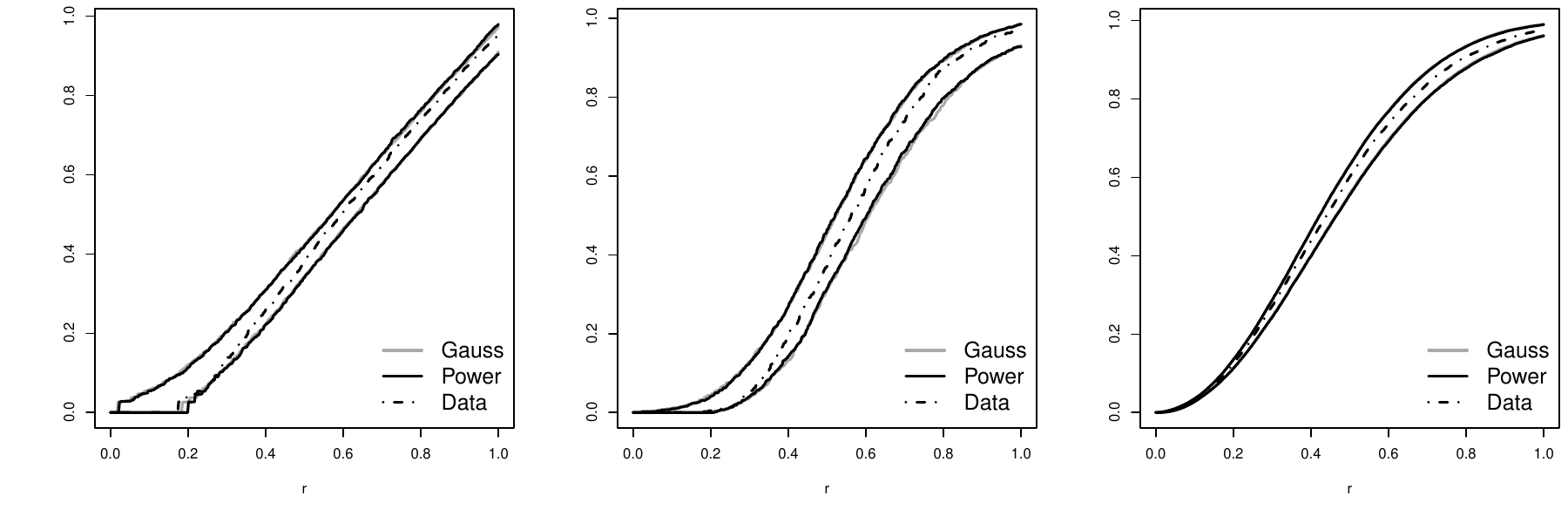}
    \caption{ Left to right: Non-parametric estimate of $L(r)-r$,
      $G(r)$ and $F(r)$ for the mucous membrane dataset with rank
      envelopes with a $5\%$ significance level for both the fitted
      Gaussian model (gray lines) and the fitted power exponential
      spectral model (black lines). Both sets of
      envelopes are based on 3000 simulated realizations.}
  \label{fig:mucous-summary}
\end{figure}

According to the third step of the procedure, we deduce the fitted
kernel \eqref{CX} of the mucous dataset and in particular its fitted
pair correlation function. The latter is hard to visualize since it is
not invariant by translation and depends on two two-dimensional
vectors $x=(x_1,x_2)$ and $y=(y_1,y_2)$. However, in our case it only
depends on $y_1-x_1$, $x_2$, and $y_2$ because $T_1$ is
linear. The right panel of
Figure~\ref{fig:mucous-rhopcf} 
shows the fitted pair correlation function as a function of
$(x_2,y_2)$ when $y_1-x_1=0$.

\section{Concluding remarks}\label{sec:conremarks}

In this paper we have
introduced several parametric
models for DPPs and discussed to which degree they can model
repulsiveness. In analogy with a zero-mean Gaussian process, the law of a DPP is
determined by a function, viz.\ the kernel (or covariance function)
$C$, which as
illustrated in our examples of applications can be chosen in many
ways. We have also
derived approximations which allow us in practice to deal with
likelihoods and simulation for DPP models, and we have demonstrated how
likelihood and moment based inference procedures work for simulated and
real point pattern datasets. In comparison to general Gibbs point
processes, DPPs are much easier to handle.

We do not think of a DPP as a mechanistic model, i.e.\ a
model which
describes a physical process generating a spatial point process
dataset. Rather our overall purpose of fitting DPP models is to
provide empirical models with a parsimonious parametrization, where we
can compare different spatial point pattern datasets by comparing
their estimated DPP model parameters, their maximized likelihoods,
their intensities and pair correlation
functions as well as other summaries. 
This is also possible when fitting parametric Poisson
process models, Poisson cluster process models, and Cox process models
(see e.g.\ \cite{moeller:waagepetersen:07} and the references
therein), however, these are not models for repulsion (or regularity
or inhibition) but rather for no interaction in the Poisson case and
clustering or aggregation in the other cases. Moreover, as mentioned,
 for Gibbs point processes it is in
general more complicated to use a maximum likelihood approach and it
is not possible to find the intensity, the pair correlation
function or other moment properties except by using time consuming simulations.

 At several places we
discussed repulsiveness of DPPs. 
However, there is a trade-off between how large the intensity
and how repulsive a stationary DPP can be, cf.\
Section~\ref{sec:examples}.
In Appendix~\ref{sec:quantify} we
suggested the quantity $\mu$ in \eqref{eq:rough}
as a rough way of quantifying repulsiveness for stationary point
processes with a pair correlation function not greater than one, but
it may
be worthwhile to consider
other ways of quantifying repulsiveness in DPPs.
In particular, we have characterized the 'most repulsive stationary DPP' with fixed intensity $\rho$, the kernel of which is a jinc-like function in the planar case, see \eqref{e:2trekant}.

DPPs cannot be as repulsive as Gibbs hard-core point processes; e.g.\ 
the cell dataset in \cite{ripley:77} is too regular to be fitted by a DPP model.
 Nevertheless, the jinc-like DPP
exhibits strong repulsiveness as seen in
Figure~\ref{fig:intro}\subref{subfig:detjinc:sim} and allows us to fit
quite regular datasets such as the termite mounds dataset, see
Section~\ref{sec:termites}.  For comparison consider a Strauss process
\citep{strauss:75,kelly:ripley:76} which is a standard example of a
repulsive Gibbs point process. Ignoring edge effects, the Strauss
process restricted to a bounded window has a density with respect to
the unit rate Poisson process which is proportional to
$\beta^{n(\bx)}\gamma^{s_R(\bx)}$, where $\beta>0$, $0\le\gamma\le1$,
$R>0$, $n(\bx)$ is the number of points in $\bx$, and $s_R(\bx)$ is
the number of (unordered) $R$-close pairs of points in $\bx$. The
Strauss process fitted by maximum pseudo-likelihood to the termite
mounds dataset (Section~\ref{sec:termites}) gives $\hat R=23.4$,
$\hat\beta=0.019$, and $\hat\gamma=0.18$. The fitted range of repulsion
($\hat R=23.4$) is in agreement with the practical range of repulsion
of the jinc-like DPP, which roughly corresponds to 
$0.96/\sqrt{\rho}$---see
Figure~\ref{fig:gengamma}\subref{fig:gengamma:pcf} 
where
$r_0=0.96/\sqrt\rho$ is the first $r$-value such that for the
jinc-like DPP, $g_0(r) = 0.99$---that is, $\hat r_0 =
0.96/\sqrt{\hat\rho}=26.8$ for the termite mounds data when the
estimate $\hat\rho=n/|S|$ is used.  Comparing the fitted Strauss
process with the fitted jinc-like DPP, not only the moments of the
jinc-like DPP can be expressible in closed form, unlike the Strauss
model, but it turns out that simulation of this jinc-like DPP is very
fast, while long Markov chain Monte Carlo simulations are needed for
this Strauss process.

\wm{} and Cauchy models with low values of $\nu$ (e.g.\ $\nu<0.5$) are
very close to Poisson, and in our experience it requires a rather
large point pattern dataset before the null hypothesis of Poisson is
rejected by a likelihood ratio test.  Such models (close to Poisson)
are difficult to estimate and simulate, since very large values of $N$
will be needed in the truncations discussed in
Section~\ref{sec:likest} in order to obtain satisfactory
approximations of $C$ and $\Ct$.

In general there is an inverse relationship between the range of
correlation and the spread of the spectral density: if $C_0$ decays
rapidly then $\varphi$ decays slowly, and if $C_0$ decays slowly then
$\varphi$ decays rapidly. This is in line with the following fact: the
(generalized) Fourier transform of the Dirac delta function (over
$\real^d$) is one and vice versa, and the Dirac function is the kernel
of the Poisson process.  From an `end user' point of view this is very
important: everything works well and is fast for DPPs except in the
less interesting cases which are close to Poisson. In such cases very
weakly repulsive Gibbs point processes (e.g.\ a Strauss process with
$\gamma$ close to one) become interesting competitors to DPPs, unless
some other and more efficient approximations of $C$ and $\Ct$ are
developed (we leave this problem for future research).


In
Section~\ref{sec:approx2} we discussed useful
approximations of $C$ and $\Ct$ restricted to
$R\times R$ when $R\subset\real^d$ is rectangular.
Frequently in the spatial point process
literature, including the present paper,
 spatial point pattern datasets observed within a
rectangular region are considered. However,
applications with non-rectangular observation windows are not
uncommon, see e.g.\ \cite{Harkness:Isham:83}. It remains to clarify
how such cases should be handled when fitting DPP models. We could
embed a non-rectangular observation window
$W$ into a rectangular region $R$ and consider the situation as a
missing data problem, since we are missing the events in $R\setminus W$, and
at least the `complete likelihood' can be handled. In our opinion this
seems a difficult approach for maximum likelihood. However, moment
based estimation and other simple alternatives to maximum
likelihood as discussed in this paper 
will easily
apply.

Large point pattern datasets may exhibit aggregation on the large
scale and repulsiveness on the small scale. For this purpose  
DPPs models which depend
on spatial covariates may be sufficient in some cases. As another possibility 
we are currently developing models for dependent thinnings of DPPs.

Generalizations of DPPs to weighted DPPs, which also are models for
repulsion, and to the closely related permanental and weighted
permanental point processes, which are models for attraction, are
studied in \cite{shirai:takahash:03} and
\cite{McCullagh:Moeller:06}. Since determinants have a geometric
meaning, are multiplicative, and there are algorithms for fast
computations, DPPs are much easier to deal with, not at least from a
statistical and computational perspective.  The approximations of $C$
and $\Ct$ using a Fourier basis approach (Section~\ref{sec:approx2})
apply as well for weighted DPPs and weighted permanental point
processes, but the practical usefulness of the approximations is yet
unexplored in these cases.

 Though a DPP is  broadly speaking a special case of a Gibbs point process, we
are not aware of a simple Hammersley-Clifford-Ripley-Kelly
representation (see \cite{ripley:kelly:77}) in terms of a product of interaction
functions (or, using the terminology of statistical physics, a sum of
potentials). Gibbsianness of
DPPs has been studied in
\cite{Georgii:Yoo:05}, see
Remark~\ref{rem:Panpangelou}. In our opinion the
Markovian properties
of DPPs is still an interesting area of research.

 DPPs models as studied in this may be extended to other spaces and
 settings. In an ongoing project we are studying DPPs on the
 sphere. Other interesting cases are multivariate and marked DPPs, DPP
 models for preferential sampling \citep{Diggle10}, incompletely observed DPPs
 \citep{Chakraborty11}, and not at least space-time DPPs. 
In the continuous time case of a space-time DPP, formally
we are just dealing with a DPP defined on $\real\times\real^d$
(where $\real$ is considered to be the time axis),
but the natural direction on the time axis should be taken into
consideration when
developing parametric families of
 space-time covariance functions and understanding how they can be used
 for modelling repulsion between events
 in time or space or both time and space. 
One possibility for building space-time covariance functions
is a spectral approach \citep{Stein05}.
Also
the development of statistical inference procedures for such
models is a challenge.
Recently, in a discrete time setting,
\cite{Affandi:Kulesza:Fox:12}
have constructed a Markov chain of DPPs with a finite state space.
It would be interesting to
study a similar Markov chain
construction for our case with state space $\real^d$.

\subsection*{Acknowledgments}

We are grateful to Philippe Carmona, Morten Nielsen, and Rasmus
Waagepetersen for helpful comments and to Adrian Baddeley for
supplying the Japanese pines dataset, which Yosihiko
Ogata and Masaharu Tanemura kindly granted us permission to
use. Supported by the Danish Council for Independent Research | Natural
Sciences, grant 09-072331, "Point Process Modelling and Statistical
Inference", and grant 12-124675, "Mathematical and Statistical
Analysis of Spatial Data".  Supported by the Centre for Stochastic
Geometry and Advanced Bioimaging, funded by a grant from the Villum
Foundation.

\appendix
\section*{Appendices}

\section{Smooth transformations and
independent thinning of DPPs}\label{sec:trans-thin}

\begin{proposition}\label{prop:trans}
Let $B,U\subseteq\real^d$ be Borel sets and 
$T:B\to U$ a
diffeomorphism such that its inverse $T^{-1}$
has a non-zero Jacobian determinant $J_{T^{-1}}(x)$
for all $x\in U$. If $X_1\sim\detproc_B(C_1)$ and $X_2=T(X_1)$, then
$X_2\sim\detproc_U(C_2)$ with
\begin{equation}\label{e:C-trans}
C_2(x,y)=|J_{T^{-1}}(x)|^{1/2}C_1(T^{-1}(x),T^{-1}(y))
|J_{T^{-1}}(y)|^{1/2}.
\end{equation}
\end{proposition}

\begin{proof} Follows immediately from \eqref{e:def1}
and \eqref{eq:productIntensity}.
\end{proof}

\begin{proposition}\label{prop:thinning}
If
$X_1\sim\detproc(C_1)$ and $X_2$ is obtained as an independent thinning of
$X_1$ with retention probabilities $p(x)$, $x\in\real^d$, then
$X_2\sim\detproc(C_2)$ with $C_2(x,y)=\sqrt{p(x)}C_1(x,y)\sqrt{p(y)}$.
\end{proposition}

\begin{proof}
Let $U=\{U(x):x\in \real^d\}$ be a random field of independent Bernoulli
variables where $\prob(U(x)=1)=p(x)$ and $U$ is independent of
$X_1$. Then $X_2$ is distributed as $\{x\in X_1:U(x)=1\}$, so from
\eqref{e:def1} and \eqref{eq:productIntensity}
it is clear that $X_2\sim\detproc(C_2)$.
\end{proof}

\section{Proof of Theorem~\ref{thm:existence}}\label{sec:existproof}

Let the situation be as in Theorem~\ref{thm:existence}
(a slightly different result where the eigenvalues are
  strictly less than one was first given in Theorem~12 of
  \cite{Macchi:75}).
Recall that $C$ is of local trace
class if \begin{equation*}\label{e:trace}
\tr_S(C)=\sum_{k=1}^\infty |\lambda_k^S|<\infty
\quad\mbox{for all compact $S\subset\real^d$}.
\end{equation*}
We apply Theorem~4.5.5 in
  \cite{Hough:etal:09}, where $C:\real^d\times \real^d \to \complex$
  is Hermitian, locally square integrable, of local trace class, and,
  as \eqref{eq:eigenrep} may not hold on a Lebesgue nullset, that $C$
  is simply given by \eqref{eq:eigenrep}. Then existence of
  $\detproc(C)$ is equivalent to that for all compact $S\subset\real^d$,
  $0\le\lambda_k^S\le1$, $k=1,2,\ldots$. When $C$ is continuous, this
  nullset vanishes and local square integrability is satisfied.  When
  $C$ is Hermitian and non-negative definite, the eigenvalues are
  non-negative, and so continuity of $C$ implies the local trace class
  assumption, since the trace $\sum_{k=1}^\infty
  |\lambda_k^S|=\sum_{k=1}^\infty \lambda_k^S=\int_S C(x,x)\,\mathrm
  dx$ is finite. Thereby Theorem~\ref{thm:existence} follows.

\section{Reduced Palm distributions for DPPs}\label{sec:Palm}

Recall that for any simple locally finite
spatial point process $Y$ on $\real^d$ with
intensity function $\rho$,
there exist unique reduced Palm distributions
$\prob_x^!$ for
Lebesgue almost all $x\in\real^d$ with $\rho(x)>0$, which
 are determined by that
\begin{equation*}\label{e:redPalmdist}
  \mean\sum_{x\in Y}h(x,Y\setminus\{x\}) =
  \int\int\rho(x)h(x,\bx)\,\dee\prob^!_x(\bx)\,\dee x
\end{equation*}
for any non-negative Borel function $h$, where $\bx$ denotes a locally
finite subset of $\real^d$. See e.g.\
\cite{stoyan:kendall:mecke:95} and Appendix~C.2 in
\cite{moeller:waagepetersen:00}.
Intuitively,
$\prob_x^!$ is the conditional distribution of
$Y\setminus\{x\}$ given that $Y$ has an event at $x$.
When all $n$'th order
product density functions $\rho^{(n)}$ of $Y$ exist, $n=1,2,\ldots$, then
for
Lebesgue almost all $x\in\real^d$ with $\rho(x)>0$,
 $\prob_x^!$
has
 $n$'th order
product density function
\begin{equation}\label{palm-moments}
\rho^{(n)}_x(x_1,\ldots,x_n)=\rho^{(n+1)}(x,x_1,\ldots,x_n)/\rho(x)
\end{equation}
and otherwise we can take $\rho^{(n)}_x(x_1,\ldots,x_n)=0$. See
e.g.\ Lemma~6.4 in \cite{shirai:takahash:03}.

For $X\sim\detproc(C)$, using
\eqref{eq:productIntensity} it can be shown that for all $x\in\real^d$
with $C(x,x)>0$, we
can take
$\prob_x^!=\detproc(C^!_x)$ where
\[C^!_x(u,v)=\det[C](u,x;v,x)/C(x,x),\quad u,v\in\real^d,\]
and where $[C](x_1,x_2;y_1,y_2)$ is the $2\times2$ matrix with entries
$C(x_i,y_j)$, $i,j=1,2$. See Theorem~6.5 in \cite{shirai:takahash:03}
(where their condition~A is implied by
the conditions in our Theorem~2.4). Moreover, \eqref{palm-moments}
holds whenever  $C(x,x)>0$.

\section{Simulation of $M$}\label{sec:inversionM}

Let the situation be as in Section~\ref{sec:Bern}.
For $m=0,1,2,\dots$, let
\begin{equation*}
  p_m = P(M=m) = \lambda_m\prod_{i>m}(1-\lambda_i).
\end{equation*}
Note that $m'=\sup\{k\ge0:\lambda_k=1\}$ is finite, and $p_m=0$
whenever $m<m'$.
For $m\ge m'$,
the $p_m$'s can be computed using the recursion
\begin{equation*}
  p_{m'}=\prod_{k=m'+1}^\infty(1-\lambda_k) ,\quad
p_{m+1} = \frac{\lambda_{m+1}}{\lambda_{m}(1-\lambda_{m+1})} p_m,\quad m=m',m'+1,\ldots
\end{equation*}
The calculation of $p_{m'}$ may involve numerical methods.
Let $F$ denote the distribution function of $M$ and introduce
\begin{equation*}
  q_m = F(m) = \prob(M\le m) = \sum_{k=0}^m p_k.
\end{equation*}
 The inversion method for simulation of $M$ is based on the fact that
$F^{-}(U)=\min\{m:q_m\ge U\}$ is distributed as $M$ if $U$ is
uniformly distributed on $(0,1)$.

\section{Proof of Theorem~\ref{thm:algorithmWorks} and related remarks}\label{sec:algproof}

Let the situation be as in Theorem~\ref{thm:algorithmWorks}.
In the sequel, for ease of
  presentation, we ignore null sets.

We start by proving by induction that for $i=n,\dots,1$,
  \eqref{eq:p_i} is a probability density and
  $\bv(X_n),\dots,\bv(X_i)$ are linearly independent (considering
  complex scalars).

  For $i=n$ and $x\in S$, we have
  $
    p_n(x) = {\|\bv(x)\|^2}/{n} \ge 0 \text{ for all } x \in S
  $,
  and
  \begin{equation*}
    \int_S p_n(x) \dee x = \frac{1}{n} \int_S \|\bv(x)\|^2 \dee x =
    \frac{1}{n} \int_S \sum_{k=1}^n |\phi_k(x)|^2 \dee x = 1.
  \end{equation*}
  Hence $p_n$ is a probability density. Clearly, $p_n(x)=0$ whenever
  $\bv(x)={\bm 0}$, so as $X_n$ is generated from $p_n$,
  $\bv(X_n)\not={\bm 0}$ (almost surely). Thus the induction hypothesis is verified for $i=n$.

  Suppose $1\le i < n$. By the induction hypothesis,  $H_{i}$ as
  defined by \eqref{defHi} has
dimension
  $n-i$. Let $P_i$ be the matrix of the orthogonal projection from
  $\complex^n$ onto $H_i^\perp$. By \eqref{e:star}, for all $x\in S$,
$p_i(x)=\|P_i\bv(x)\|^2/i\ge0$  and
  \begin{equation}\label{eq:zerodensity}
    p_i(x) = 0 \quad \text{whenever } \bv(x)\in H_i.
  \end{equation}
  By the spectral theorem, $P_i = U\Lambda_i U^*$, where $U$ is
  unitary and $\Lambda_i$ is diagonal with the first $i$ diagonal elements
  equal to one and the rest zero. Let $u_{kj}$ denote the $(k,j)$'th entry of $U$.
  Then
  \begin{equation*}
    p_i(x)= \frac{1}{i} \bv(x)^* U \Lambda_i U^* \bv(x) =  \frac{1}{i}
    \|\Lambda_iU^*\bv(x)\|^2
  \end{equation*}
  where the $j$'th entry of $\Lambda_iU^*\bv(x) $ is
  $\sum_{k=1}^nu_{kj}\phi_k(x)$ if $j\le i$, and 0 otherwise, so
  \begin{align*}
    \int_S& p_i(x) \dee x = \frac{1}{i} \int_S \sum_{j=1}^i \sum_{k=1}^n\sum_{l=1}^n u_{kj}
    \phi_k(x) \conj{u}_{lj} \conj{\phi_l(x)} \dee x\\
    &= \frac{1}{i} \sum_{j=1}^i \sum_{k=1}^n\sum_{l=1}^n u_{kj}
    \conj{u}_{lj}\int_S \phi_k(x)\conj{\phi_l(x)} \dee x
    = \frac{1}{i} \sum_{j=1}^i \sum_{k=1}^n
    |u_{kj}|^2 \int_S |\phi_k(x)|^2
    \dee x = 1.
  \end{align*}
  Thus $p_i$ is a probability density. Finally, it follows immediately from \eqref{eq:zerodensity} and the
  induction hypothesis that $\bv(X_n),\dots,\bv(X_{i+1}),\bv(X_{i})$
  are linearly independent with probability one.

Hence, the induction
  hypothesis is verified for all $i=n,\dots,1$.

 Now, for iteration $i<n$, write $P_i=P_i(X_n,\dots,X_{i+1})$ and $H_i^\perp=H_i^\perp(X_n,\dots,X_{i+1})$ to emphasize the
  dependence on the previously generated variables. For $i=n$, set
  $P_i(X_n,\dots,X_{i+1})=I_n$ and $H_i^\perp(X_n,\dots,X_{i+1})=\complex^n$. Let
  \begin{equation*}
    \Omega=\{(x_1,\dots,x_n)\in S^n:\, \bv(x_1),\dots,\bv(x_n)
    \text{ are linearly independent}\}
  \end{equation*}
  be the support of $(X_1,\dots,X_n)$. Since
    $p_i(x) = \| P_i \bv(x) \|^2/i$, $(X_1,\dots,X_n)$ has density
  \begin{equation*}
    p(x_1,\dots,x_n) = \frac{1}{n!} \prod_{i=1}^n \|
    P_i(x_n,\dots,x_{i+1})\bv(x_i)\|^2, \quad
    (x_1,\dots,x_n)\in\Omega.
  \end{equation*}
  This product is exactly the square of the volume of
  the parallelepiped determined by the vectors
  $\bv(x_1),\dots,\bv(x_n)$, which is equal to the
  determinant of the $n\times n$ Gram matrix with $(i,j)$'th entry  $\bv(x_i)^*\bv(x_j)$,
which in turn is equal to the matrix $[K](x_1,\dots,x_n)$.
  Thus, for $(x_1,\dots,x_n)\in\Omega$,
  \begin{equation}\label{e:p-den}
    p(x_1,\dots,x_n) = \frac{1}{n!}
\det [K](x_1,\dots,x_n).
  \end{equation}
 Moreover, if $(x_1,\dots,x_n)\in S^n\setminus\Omega$, $\det[K](x_1,\dots,x_n)=|\det[\bv(x_1) \dots
  \bv(x_n)]|^2=0$. Hence \eqref{e:p-den} is valid for all
  $(x_1,\dots,x_n)\in S^n$.

  Viewing $\{X_1,\dots,X_n\}$ as a point process, the number of points
  is fixed and equal to $n$, and hence by definition of $\rho^{(n)}$
  for $\{X_1,\dots,X_n\}$,
  \begin{equation}\label{e:density-n}
    \rho^{(n)}(x_1,\ldots,x_n)=n!p(x_1,\dots,x_n) = \det
    [K](x_1,\dots,x_n), \quad (x_1,\ldots,x_n) \in S^n.
  \end{equation}
This completes the proof of Theorem~\ref{thm:algorithmWorks}.

\begin{remark}\label{rem:algorithWorks}
Let $n>0$, and define $H_n=\{{\bm 0}\}$ and for $i=n-1,\ldots,1$,
\begin{equation}\label{defHi}
   H_i =\spanspace{\bv(X_n),\ldots, \bv(X_{i+1})}=
   \left\{\sum_{j=i+1}^n \alpha_j\bv(X_j) :\, \alpha_j\in\complex
   \right\}.
\end{equation}
 With probability one, $\bv(X_n),\dots,\bv(X_i)$ are linearly
 independent,  cf.\ the proof above.
Thus, almost surely, $H_i$ is a subspace of  $\complex^n$ of dimension
$n-i$. For $i=n-1,\ldots,1$,
by the Gram-Schmidt procedure employed in Algorithm \ref{alg:sim},
$\be_1,\dots,\be_{n-i}$ is
an orthonormal basis of $H_i$. Further, for $i=n,\ldots,1$,
$ip_i(x)$ is the square
norm of the orthogonal projection  of $\bv(x)$ onto
$H_i^\perp$ (the orthogonal complement to $H_i$).
\end{remark}

\begin{remark}\label{rem:successive}
According to the previous remark,
\begin{equation}\label{e:star}
ip_i(x)=\|P_i\bv(x)\|^2
\end{equation}
where $P_i$
 is the matrix of the orthogonal projection from $\complex^n$ onto
 $H_i^\perp$. Denoting by $I_n$ the $n\times n$
identity matrix, we have for $i<n$,
\begin{equation}\label{successive} P_i= \prod_{k=n}^{i+1}
\left(I_n-\frac{\bv(x_k){\bv(x_k)}^*}{K(x_k,x_k)}\right).
\end{equation}
This provides an alternative way
 to calculate the density $p_i(x)$, where $P_i$ is obtained
 recursively. This idea was used in \cite{Scardicchio:etal:09} but, as
 noticed there, the successive multiplication of matrices leads to
 numerical instabilities. Some corrections must then be applied at
 each step to make $P_i$ a proper projection matrix when $n-i$ is
 large.  In contrast, the calculation of $p_i(x)$ in
Algorithm~\ref{alg:sim} is straightforward and
numerically stable.
 \end{remark}

\begin{remark}\label{rem:rejsampling}
  Note that for
  $x$ such that $\bv(x)\in H_i^\perp$, $p_i(x)=\|\bv(x)\|^2/i$.  Thus
  for small values of $i$, simulation of $X_i$ by rejection sampling
  with respect to a uniform density may be inefficient. However, the
  computation of $p_i(x)$ is fast, so this is not a major drawback in
  practice.  For the examples in this paper, we have just been using
  rejection sampling with a uniform instrumental
  distribution. Appendix~\ref{sec:closeupper} discusses other choices of the instrumental
  distribution.
\end{remark}

\section{Close upper bounds on the conditional distributions of
  Algorithm~1}\label{sec:closeupper}

In Remark~\ref{rem:rejsampling} we discussed rejection sampling from
the densities $p_i$, $i=n,\ldots,1$, using uniform instrumental
distributions.  For intensive simulations purposes, for each $i$, it
is desirable to construct an unnormalized instrumental density which
is larger than and close to $p_i$ as well as easy to simulate from.

To find such an unnormalized density, we first notice the following.
It follows from
Remark~\ref{rem:successive}
that  $ip_i(x)$ is the norm of a vector obtained after $n-i$
successive orthogonal projections of $\bv(x)$. These projections
commute, so that $ip_i(x)$ is lower than the norm of any projection of
$\bv(x)$ of a lower order. By \eqref{e:star} and
\eqref{successive}, if $i+1\leq k\leq n$, then
\begin{equation*} p_i(x)\leq  \frac{1}{i} \left\| \left(I_n-\frac{\bv(x_k)\bv(x_k)^*}{K(x_k,x_k)}\right) \bv(x) \right\|^2\end{equation*}
and so by \eqref{eq:projectionKernelN},
\begin{equation}\label{bound-pi}
p_i(x)\leq    \frac{1}{i}  \min_{i+1\leq k\leq n} \left(K(x,x) -
  \frac{|K(x,x_k)|^2}{K(x_k,x_k)}\right),\quad i<n.
\end{equation}
Here the right hand side is an unnormalized density, since it is a continuous
function of $x\in S$ where $S$ is compact.

The proof of the following lemma uses \eqref{bound-pi} to derive an
explicit upper bound in the
specific setting of Section~\ref{sec:stationary}, i.e.\ in the
stationary case, when $S=[-1/2,1/2]^d$, and when the eigenfunctions
are Fourier basis functions 
\[\phi_k(x)=\mathrm{e}^{2\pi \mathrm{i}  k\cdot x},\quad k\in\Z^d,\
x\in S.\]

Let $x=(x(1),\dots,x(d))\in\real^d$ and
$y=(y(1),\dots,y(d))\in\real^d$, and suppose  that
\[\{\phi_1,\ldots,\phi_n\}=\{\varphi_{j_1,\ldots,j_d}:j_1\in J_1(n_1),\ldots,j_d\in J_d(n_d)\}\]
where $\varphi_{j_1,...,j_d}(x) = \exp\left( 2\pi \mathrm i
  \sum_{k=1}^d j_k x(k) \right)$ and for $q=1,\dots,d$,
 $J_q(n_q)$ denotes some finite
subset of $\Z$ with $n_q$ elements, such that $n=\prod_{q=1}^d n_q$.
Then the projection kernel
\eqref{eq:projectionKernelN} becomes
\begin{equation}\label{proj-stat}
K(x,y)=\prod_{q=1}^d\sum_{j_q\in J_q(n_q)}
  \mathrm{e}^{2\pi\mathrm{i} j_q (x(q)-y(q))}.
\end{equation}
Moreover, for any $r\in\N$,
denote $S_q(r)=\sum_{j_q\in J_q(n_q)}  j_q^r$, and for any number $a$, define $a_+=\max(a,0)$.

\begin{lemma}\label{lem-ineq}
Let $K$ be the projection kernel \eqref{proj-stat}.
For step $i=n-1,\ldots,1$ of Algorithm \ref{alg:sim},  given the $n-i$ previous points $x_k=({x_k}(1),\dots,{x_k}(d))$, $k=i+1,\dots,n$, we have
\begin{equation}\label{e:sss}   p_i(x)\leq \frac{n}{i}\left(1- \max_{i+1\leq k\leq n} \prod_{q=1}^d \left(1-\frac{2\pi}{n_q}\;|x(q)-x_k(q)|\sqrt{n_qS_q(2)-S^2_q(1)}\right)_+\right).
\end{equation}
\end{lemma}
\begin{proof}
For $x,y\in\real$, let $K_q(x,y)=\sum_{j_q\in J_q(n_q)
}\mathrm{e}^{ 2\pi \mathrm{i}  j_q (x-y)}$. An analytic expansion of $|K_q(x,y)|^2$  leads to
\begin{equation*}
|K_q(x,y)|^2=\sum_{p= 0}^\infty(-1)^p (x-y)^{2p} (2\pi)^{2p} \sum_{l=0}^{2p} \frac{(-1)^l }{l! (2p-l)!} S_q(2p-l)S_q(l).
\end{equation*}
Note that
\begin{equation*}\sum_{l=0}^{2p} \frac{(-1)^l }{l! (2p-l)!} S_q(2p-l)S_q(l)=\frac{1}{(2p)!}\sum_{(i,j) \in J^2_q(n_q)}(j-i)^{2p}\geq 0. \end{equation*}
Therefore, the function $x\rightarrow|K_q(x,y)|^2$ can be expanded into an alternate series. In particular, for any $x,y\in\real$, since $S_q(0)=n_q$,
\begin{equation*}
|K_q(x,y)|^2\geq n_q^2-4\pi^2 (x-y)^2(n_qS_q(2)-S^2_q(1)).
\end{equation*}
This lower bound is a concave function of $|x-y|$ when
\begin{equation*}
|x-y|\leq \frac{n_q}{2\pi\sqrt{n_qS_q(2)-S^2_q(1)}}
\end{equation*}
and so
\begin{equation*}
\frac{|K_q(x,y)|^2}{K_q(y,y)} =\frac{|K_q(x,y)|^2}{n_q} \geq \left(n_q- 2\pi |x-y|\sqrt{n_qS_q(2)-S^2_q(1)}\right)_+.
\end{equation*}
Combining this with \eqref{bound-pi}, we obtain \eqref{e:sss}.
\end{proof}

The upper bound in \eqref{e:sss} provides an unnormalized instrumental
 density close to
$p_i$. When $d=1$, this instrumental density is  a  stepwise linear
function, and hence it is very easy to make
simulations under the instrumental density.
When $d=2$, the instrumental density provided by
\eqref{e:sss} is
a stepwise polynomial function. One strategy is then to provide a
further upper-bound making rejection sampling feasible. In our
experience this is
not so hard for the DPP models we have considered, but since it depends much on the points
$x_{i+1},\dots,x_n$ and the particular model, it seems not easy to state a general result.

\section{Proof of Theorem~\ref{thm:density} and related
  remarks}\label{sec:proofdensitythm}

Theorem~\ref{thm:density}  was first verified in \cite{Macchi:75}. Note that
  the right hand side in \eqref{e:unconddensity} is not depending on the
  ordering of the events.  Equation~\eqref{e:unconddensity} follows
  from a longer but in principle straightforward calculation, using
  \eqref{e:imp}, \eqref{e:density-n}, and the fact that if $Y$ follows
  the homogeneous Poisson process on $S$ with unit intensity, then
\begin{equation*}
\rho^{(n)}(x_1,\ldots,x_n)=\mean f(Y\cup\{x_1,\ldots,x_n\}).
\end{equation*}
See \cite{shirai:takahash:03} and \cite{McCullagh:Moeller:06}.

\begin{remark}\label{rem:noeigenvalues}
It is possible to express $\Ct$ and $D$ in terms of $C$ without any
direct reference to the spectral representations \eqref{eq:eigenrep}
and \eqref{e:defCt}:
Let
\begin{equation}\label{e:fnat}
C_S^1(x,y)=C_S(x,y),\quad C^k_S(x,y)=\int_S C_S^{k-1}(x,z)C_S(z,y)\,\mathrm
dz,\quad x,y\in S,\ k=2,3,\ldots.
\end{equation}
Then
\begin{equation}\label{e:DD}
D=\sum_{k=1}^\infty \tr_S(C_S^k)/k
\end{equation}
and
\begin{equation}\label{eq:CtSUM}
\Ct(x,y)=\sum_{k=1}^\infty C_S^k(x,y),\quad x,y\in S.
\end{equation}
Also, as noticed in \cite{Macchi:75}, $\Ct$ is the unique solution
to the integral equation
\begin{equation*}\label{e:inteq}
\Ct(x,y)-\int_S\Ct(x,z)C(z,y)\,\mathrm dz=C(x,y),\quad x,y\in S.
\end{equation*}

\end{remark}

\begin{remark}\label{rem:Panpangelou}
The density \eqref{e:unconddensity} is
hereditary in the sense that $f(\{x_1,\ldots,x_n\})>0$ whenever
 $f(\{x_1,\ldots,x_{n+1}\})>0$.
This allows us to define the Papangelou
conditional intensity for all finite point configurations $\bx=\{x_1,\ldots,x_n\}\subset
S$ and points $u\in S\setminus\bx$ by
\[\lambda(u;\bx)={f(\bx\cup\{u\})}/{f(\bx)}={\det[\Ct](x_1,\ldots,x_n,u)}/{\det[\Ct](x_1,\ldots,x_n)}\]
(taking $0/0=0$). \cite{Georgii:Yoo:05} use this to
study the link to Gibbs point
processes, and establish the following result of statistical
interest:
for any finite point configurations $\bx\subset S$ and $\by\subset S$,
\begin{equation}\label{e:mono}
\lambda(u;\bx)\ge\lambda(u;\by)\quad \mbox{whenever
$\bx\subset\by$}
\end{equation}
and for any point $u\in S\setminus\bx$,
\begin{equation}\label{e:locstab}
\lambda(u;\bx)\le\Ct(u,u)
\end{equation}
(Theorem~3.1 in \cite{Georgii:Yoo:05}).
The
monotonicity property \eqref{e:mono} is once again confirming
the repulsiveness of a DPP, and
\eqref{e:locstab} means that $X_S$ is locally stable.  

Hence
$X_S$ can be coupled with a Poisson process $Y_S$ on $S$ with intensity
function given by $\Ct(u,u)$, $u\in S$, such that $X_S\subseteq Y_S$
(see \cite{kendall00} and \cite{moeller:waagepetersen:00}).
This coupling is such that $X_S$ is
obtained by a dependent thinning of $Y_S$ as detailed in the
abovementioned references. By considering a sequence
$S_1\subset S_2\subset\ldots$ of
compact sets such that $\real^d=\cup_n S_n$ (e.g.\ a sequence of
increasing balls whose diameters converge to infinity), and a
corresponding sequence of processes $X_n\sim\detproc(C;S_n)$ which are
coupled with a Poisson process $Y$ on $\real^d$ with intensity
function given by $\Ct(u,u)$, $u\in\real^d$, such that
$X_1\subseteq X_2\subset\ldots\subseteq Y$, we obtain that $\cup_n
X_n\subseteq Y$ follows $\detproc(C)$. In other words,
 $X$ can be realized as a dependent thinning of
the Poisson process $Y$.

Imposing certain
conditions concerning a finite range assumption on an extended
version of $\Ct$ to $\real^d$ and requiring $C$ to be small enough, it is
possible to extend the Papangelou conditional intensity for
$X_S$ to a global Papangelou conditional intensity for $X$ and hence
to derive the reduced Palm distribution of
$X$ (for details, see Proposition~3.9 in \cite{Georgii:Yoo:05}).
Unfortunately, these conditions are rather restrictive, in
particular when $d\ge2$.
\end{remark}

\section{Proof of Proposition~\ref{prop:c1c2}}\label{proofProp_spectral}
For any compact set $S\subset\real^d$, define
the integral operator $T_{S}:L^2(S)\to L^2(S)$ by
\begin{equation}\label{int-operator}
T_{S}(h)(x)=\int_S C(x,y)h(y)\,\mathrm dy,\quad h\in L^2(S),\quad x\in S.
\end{equation}
The $\{\lambda_k\}$'s and $\{\phi_k\}$'s involved in \eqref{eq:eigenrep}  correspond to the eigenvalues and eigenfunctions of $T_S$, i.e.\ for all $k$,
\begin{equation}
T_{S}(\phi_k)=\lambda_k\phi_k.
\end{equation}

For $h\in L^2(S)$, define
$h_S\in L^2(\real^d)$ by $h_S(x)=h(x)$ if $x\in S$ and $h_S(x)=0$
otherwise.
From \eqref{stat-cov},  the
integral operator  $T_S$ in  \eqref{int-operator} becomes the convolution operator given by
\[T_S(h)(x)=C_0\star h_S(x)=\int_S C_0(x-y)h(y)\,\mathrm dy,\quad x\in
S.\]

Recall that the spectrum of
 $T_S$ consists of all $\lambda\in\complex$ such that the operator
 $T_S-\lambda I_S$ is not invertible or it is invertible and unbounded
(with respect to the usual operator norm), where  $I_S$ denotes the identity operator on $L^2(S)$.

Consider the multiplicative operator $Q_\varphi$ on  $L^2(\real^d)$
associated to $\varphi$, i.e.\ $Q_\varphi(h)(x)=\varphi(x) h(x)$ for
$h\in L^2(\real^d)$. Its
restriction to  $L^2(S)$ is given by
$Q_{\varphi,S}(h)=Q_{\varphi_S}(h_S)$
 for
$h\in L^2(S)$.
Note that $T_S(h)=\mathcal F^{-1}Q_{\varphi}\mathcal F(h_S)$ for
$h\in L^2(S)$.
 Since  the
Fourier operator is a unitary operator (as
$\mathcal F\mathcal F^{-1}=\mathcal F^{-1}\mathcal F=I$ where $I$ denotes the identity operator on  $L^2(\real^d)$),
the spectrum of $T_S$ is equal to the spectrum of $Q_{\varphi_S}$, which in turn
is equal to ${\mbox{ess-im}}(\varphi_S)$
(the essential image of $\varphi_S$), see (12) in Section 8.4.3 in
\cite{birman:solomjak}. In our case, ${\mbox{ess-im}}(\varphi_S)$ is
 the closure of $\varphi(S)$. Consequently,
 (C2) is equivalent to $\varphi\leq 1$.

\section{Proof of Corollary~\ref{cor:first}}\label{proofCoro_spectral}

Assume $(i)$ in Corollary~\ref{cor:first}. Then $0\leq\varphi\leq 1$ implies that
$\int |\varphi(x)|^2\dee x \leq \int |\varphi(x)|\dee x<\infty$, i.e.\
 $\varphi\in L^2(\real^d)$, and
so by  Parseval's identity $C_0\in L^2(\real^d)$. Further,
 $C_0=\mathcal F^{-1}(\varphi)$ with $\varphi\in L^1(\real^d)$,
so $C_0$ is continuous. By Bochner's theorem, the continuity of $C_0$ and the non-negativity of $\varphi$ imply that $C_0$ is positive-definite, and so
(C1) follows from (5.1). Moreover,
 (C2) holds by Proposition 5.1. Hence $(i)$ implies $(ii)$.

Conversely, assume $(ii)$. Combining Bochner's theorem  and the fact that  $C_0$ is continuous and $C_0\in
L^2(\real^d)$,  we deduce that there exists $\varphi\in L^1(\real^d)$  such that $C_0=\mathcal
F^{-1}(\varphi)$ (see also page 104 in \cite{Yaglom:87}). By (C1), we have that $\varphi\ge 0$. The
fact that $\varphi\leq 1$ follows from Proposition 5.1. Hence $(ii)$ implies $(i)$.

\section{Quantifying and comparing repulsiveness}\label{sec:quantify}

We now discuss different criteria to quantify repulsiveness. These criteria are used to compare the DPP models  introduced in Sections~\ref{sec:examples}-\ref{sec:spectral}.

 Recall that $\rho K(r)$ is the conditional expectation of the number
of further points of $X$ in a ball of radius $r$ centred at $x$ given that
$X$ has a point at $x$. As a first criterion for repulsiveness, for two
stationary DPPs with kernels $C_1$ and $C_2$, 
 common
intensity $\rho$, and corresponding $K$-functions $K_1$ and $K_2$, we may
say that $\detproc(C_1)$ exhibits stronger repulsiveness than
$\detproc(C_2)$ if $K_1(r)\le K_2(r)$ for all
$r\ge0$. If the corresponding
 pair
 correlation functions $g_1$ and $g_2$ are isotropic, i.e.\
$g_i(x,y)=g_{i0}(\|x-y\|)$, $i=1,2$, then
\begin{equation}\label{e:K1K2}
K_1\le K_2\quad\mbox{if and only if}\quad g_{10}\le g_{20}.
\end{equation}

In this sense, within each class of the Gaussian, \wm{},
and Cauchy models introduced in Section~\ref{sec:examples}, when
$\nu$ is fixed, the degree of repulsiveness increases as $\alpha$
increases. However, the increased degree of repulsiveness comes at the
cost of a decreased maximal intensity cf.\ \eqref{eq:maxrhoGauss},
\eqref{eq:maxrhoWM}, and \eqref{eq:maxrhoCauchy}. 
Letting
$\alpha=\alpha_{\max}$ given by \eqref{eq:alphamax}, the degree of repulsiveness of both
the \wm{} and the Cauchy models grows as $\nu$
grows, and the limit is the Gaussian case, cf.\ (i)-(ii) in Section~\ref{sec:examples}.

On the other hand, the power exponential spectral model of
Section~\ref{sec:spectral} contains the Gaussian model as a special
case when $\nu=2$, and it provides examples of more and more
repulsive DPPs as $\nu$ increases from zero to infinity, cf.\
Figure~\ref{fig:gengamma}.

However, superposing the two plots in Figure~\ref{fig:pcfs} or
considering Figure~\ref{fig:compareLdiff}, the comparison between a
\wm{} model and a Cauchy model is not always possible with our 
criterion for repulsiveness based on the $K$-functions.

Instead, for any stationary point process defined on
$\real^d$, with distribution $P$, constant intensity $\rho>0$,
and pair correlation
function $g(x,y)=g(x-y)$ (with a slight abuse of notation),
 we suggest
\begin{equation}\label{eq:rough}
\mu = \rho\int [1-g(x)]\dee x
\end{equation}
as a rough
measure for repulsiveness provided the integral exists.
Denote $o$ the origin of $\real^d$ and note that
 the function $x\mapsto \rho g(o,x)=\rho g(x)$ is the intensity function
 for the reduced Palm distribution $P^!_o$ (intuitively, this is the
 conditional distribution of all remaining events when we condition on
 that $o$ is an event, cf.\ Appendix~\ref{sec:Palm}).
 Therefore, $\mu$ is the limit as
$r\rightarrow\infty$ of the difference between the expected number of
events within distance $r$ from $o$ under respectively $P$ and
$P^!_o$.   For a stationary Poisson process, $\mu=0$.  For any
stationary point process, we always have $\mu\leq 1$ (see e.g.\ (2.5)
in \cite{kuna07}). When $g\leq 1$ (as in the case of a DPP),
we clearly have $\mu\ge0$, so that $0\leq\mu\leq 1$.

Especially, for a stationary DPP,
\begin{equation*}
  \mu = \rho\int [1-g(x)]\dee x = \frac{1}{\rho} \int |C_0(x)|^2 \dee x =  \frac{1}{\rho} \int |\varphi(x)|^2 \dee x
\end{equation*}
where the second equality follows from \eqref{intensity} and \eqref{eq:pcf}, and the
last equality follows from Parseval's identity.
Using an obvious notation, we
say that $\detproc(C_1)$ is more repulsive than $\detproc(C_2)$ if
$\rho_1=\rho_2$ and $\mu_1\ge\mu_2$. In the isotropic case,
this is in agreement with our former definition of
repulsiveness: if $\rho_1=\rho_2$, then $K_1\le K_2$ implies that
$\mu_1\ge\mu_2$, cf.\ \eqref{e:K1K2}.

A stationary
DPP with intensity $\rho$ and a maximal value of $\mu$ can be specified as
follows.
Since
$0\leq\varphi(x)^2\leq\varphi(x)\leq 1$, we have $\mu=1$ if and only if $\int \varphi(x)^2 \dee x = \int
\varphi(x) \dee x= \rho$. So  $\mu$ is maximal if $\varphi$ is an
indicator function with support on a Borel subset of $\real^d$ of
volume $\rho$.
A natural choice is
\begin{equation}\label{indicator}
\varphi(x)=\begin{cases}1& \text{if }
    \|x\|\leq \tau\\ 0 &
    \text{otherwise}
\end{cases}
\end{equation}
where $\tau^d=\rho d \Gamma(d/2) / (2\pi^{d/2})$, and we refer to the
corresponding DPP as 'the most repulsive stationary DPP'. For $d=1$, $C_0$ is then
proportional to a sinc function:
\begin{equation*}\label{e:1trekant}
C_0(x)= \sin(\pi
\rho x)/(\pi x)\quad \mbox{if $d=1$}.
\end{equation*}
For $d=2$, $C_0$ is then
proportional to
the 'jinc-like' function \eqref{e:2trekant}.

As already noticed in Section~\ref{sec:spectral}, the indicator
function \eqref{indicator} corresponds to the limit of \eqref{e:527}
when $\nu$ tends to infinity. Thus the power exponential spectral
model contains the most repulsive stationary DPP as a
limiting case.  Figure~\ref{fig:gengamma} illustrates how this
limiting case is approached as $\nu$ increases. Notice in particular
the slightly oscillating nature of $g$ in
Figure~\ref{fig:gengamma}\subref{fig:gengamma:pcf} for $\nu>2$. 
For Gibbs hard-core point processes, 
oscillation in the pair correlation function is also seen, but at the
hard-core distance, the pair correlation function jumps from zero to a
value larger than one (see e.g.\
\cite{illian:penttinen:stoyan:stoyan:08}).

\section{Fourier approximation of the Whittle-Mat{\'e}rn
  covariance function}\label{matern-approx}

This appendix discusses the quality of the kernel approximation \eqref{e:stjerne} for the \wm{} model. To simplify the notation we let $\Cappstat(u)=\Capp(x,y)$ where $\Capp$ is given by \eqref{e:approximation} and $u=x-y\in\unitbox$. We thus consider the approximation $C_0(u)\approx \Cappstat(u)$ where \begin{equation}\label{e:stjernestat}
\Cappstat(u)=\sum_{k\in\Z^d}\varphi(k) \mathrm{e}^{2\pi \mathrm{i} k\cdot u},\quad u\in \unitbox.
\end{equation}

To provide an upper bound on the approximation error we need some preliminary results on $K_\nu$ (the Bessel function of the second
kind) which appears in \eqref{e:matern}.


There are several equivalent ways to define $K_\nu$.
By Equation 8.432 in \cite{GR}, for all $x>0$ and all $\nu>0$,
\begin{equation}\label{e:GR}
K_{\nu}(x)=\frac{\sqrt
    \pi}{\Gamma(\nu+\frac{1}{2})}\left(\frac{x}{2}\right)^\nu\int_1^\infty
  \mathrm{e}^{-xt}(t^2-1)^{\nu-\frac{1}{2}}\,\mathrm dt.
\end{equation}
As $x\rightarrow 0$, then
$x^{\nu}K_{\nu}(x)\rightarrow2^{\nu-1}\Gamma(\nu)$. Hence by
\eqref{e:matern}, $C_0(0)=\rho$.

The following lemma provides an upper bound and  gives an idea of the
decay rate for $K_\nu$. The inequality reduces to
an equality for $\nu=1/2$. Moreover, according to various plots (omitted in this article), the bound seems sharp when $\nu>1/2$.  We denote $\gamma=\Gamma(1+2\nu)^{-1/2\nu}$.

\begin{lemma}\label{upper-bessel}
For all $x>0$,
 \begin{equation}\label{ineq-bessel1}
K_{\nu}(x)\leq
   2^{\nu-1}\Gamma(\nu) x^{-\nu} \left(1-(1-\mathrm{e}^{-\gamma
       x})^{2\nu}\right)\quad\mbox{if $\nu\geq 1/2$}
\end{equation}
and
\begin{equation}\label{ineq-bessel2}
K_{\nu}(x)\leq  K_{1/2}(x)= \sqrt{\pi/(2x)}\,\mathrm{e}^{-x}\quad\mbox{if $\nu\leq 1/2$}.\end{equation}
\end{lemma}
\begin{proof}
When $\nu\geq 1/2$, from \eqref{e:GR},
\begin{equation*}K_{\nu}(x)\leq\frac{\sqrt \pi}{\Gamma(\nu+\frac{1}{2})}\left(\frac{x}{2}\right)^\nu\int_1^\infty \mathrm{e}^{-xt}t^{2\nu-1}\,\mathrm{d}t=\frac{2^{-\nu}\sqrt \pi}{\Gamma(\nu+\frac{1}{2})}x^{-\nu}\Gamma(2\nu,x)\end{equation*}
where $\Gamma(2\nu,\cdot)$ denotes the incomplete Gamma function with parameter $2\nu$:
\begin{equation*} \Gamma(2\nu,x) = \int_x^\infty
  t^{2\nu-1}\mathrm{e}^{-t}\,\mathrm dt.\end{equation*}
From \cite{alzer:97} we deduce that
\begin{equation*} \Gamma(2\nu,x) \leq
  \left(1-\left(1-\mathrm{e}^{-\gamma
      x}\right)^{2\nu}\right){\Gamma(1+2\nu)}/\left({2\nu}\right)\end{equation*}
whenever $x>0$, $\nu\geq1/2$, and $0\leq\gamma\leq \Gamma(1+2\nu)^{-1/2\nu}$.
  Hence \eqref{ineq-bessel1} follows by using the relations $\Gamma(2\nu+1)=2\nu\Gamma(2\nu)$ and $\Gamma(\nu)\Gamma(\nu+1/2)=2^{1-2\nu}\sqrt\pi\Gamma(2\nu)$.

When $\nu<1/2$, using \eqref{e:GR} and the fact that $t^2-1>2t-2$ when
$t>1$, we obtain
\begin{equation*}K_{\nu}(x)\leq\sqrt\frac{\pi}{2}\frac{1}{
    \Gamma(\nu+\frac{1}{2})}x^\nu\int_1^\infty
  \mathrm{e}^{-xt}(t-1)^{\nu-\frac{1}{2}}\,\mathrm dt.\end{equation*}
Finally, making the change of variables $u=x(t-1)$, we obtain
 \eqref{ineq-bessel2}.
\end{proof}

For the Whittle-Mat{\'e}rn model,
the following Proposition~\ref{prop:inequality}  provides an error bound for the
approximation \eqref{e:stjernestat}  of $C_0(u)$ by
$\Cappstat(u)$ when $u\in\unitbox$. We let
\[\beta=\left(\alpha\sqrt d\ (\Gamma(1+2\nu)^{1/2\nu}\vee 1)
\right)^{-1}\]
\begin{equation}\label{cste}
c(\rho,\nu,\alpha,d)=\begin{cases} (4\alpha)^{1-2\nu}\rho^2 {\pi}d/{\Gamma(\nu)^2}\quad\textrm{if}\quad \nu\leq \frac{1}{2} \\ 4\nu^2 \rho^2 d\quad\textrm{if}\quad \nu\geq \frac{1}{2}\end{cases}
\end{equation}
\begin{equation}\label{error1}
\epsilon(\nu,\alpha,1)=\frac{\mathrm{e}^{-\beta}}{\beta}+\frac{2\mathrm{e}^{-\beta}}{1-\mathrm{e}^{-\beta}}\left(\frac{\mathrm{e}^{-\beta}}{\beta}+\frac{1}{1-\mathrm{e}^{-\beta}}-1\right)
\end{equation}
and  for $d\geq 2$,
\begin{equation}\label{error2}
\epsilon(\nu,\alpha,d)=\mathrm{e}^{-\beta}\left(\frac{1}{\beta}+\frac{2}{(1-\mathrm{e}^{-\beta})^2}-\frac{1}{2}\right)\left(\frac{1}{\beta}+\frac{2\mathrm{e}^{-\beta}}{1-\mathrm{e}^{-\beta}}\left(\frac{1}{\beta}+\frac{1}{1-\mathrm{e}^{-\beta}}\right)\right)^{d-1}.
\end{equation}

\begin{proposition}\label{prop:inequality}
Let $C_0$ be the Whittle-Mat{\'e}rn covariance function given by \eqref{e:matern} and let
$\Cappstat$ be the approximation  \eqref{e:stjernestat}
of $C_0$ on $\unitbox$. If  $0\le\rho\le\rho_{\max}$
where $\rho_{\max}$ is given by \eqref{eq:maxrhoWM},  then
\begin{equation}\label{e:mnb}
\int_{[-1/2,1/2]^d}|C_0(x)-C_{{\mathrm{app}},0}(x)|^2\mathrm{d}x\leq c(\rho,\nu,\alpha,d) \epsilon(\nu,\alpha,d).
\end{equation}
\end{proposition}
\begin{proof} We have
\begin{align*}
\int_{[-1/2,1/2]^d}|C_0(x)-C_{{\mathrm{app}},0}(x)|^2 \,\mathrm{d}x
&=\int_{[-1/2,1/2]^d}\left| \sum_{k\in\Z^d} (\alpha_k-\varphi(k))
  \mathrm{e}^{2\pi\mathrm i x\cdot k}\right|^2 \mathrm{d}x \\ &= \sum_{k\in\Z^d} (\alpha_k-\varphi(k))^2
\end{align*}
with
\begin{equation*}
\varphi(k)-\alpha_k=\int_{\real^d\setminus[-1/2,1/2]^d}
C_0(y)\mathrm{e}^{-2\pi\mathrm i k\cdot y}\,\mathrm dy.
\end{equation*}
Defining $h(y)= C_0(y) (1-1\!\!1_{[-1/2,1/2]^d}(y))$, we have $\varphi(k)-\alpha_k=\mathcal F(h)(k)$ and
\begin{equation*}
\sum_{k\in\Z^d} (\alpha_k-\varphi(k))^2=\sum_{k\in\Z^d} \left(\mathcal F(h)(k)\right)^2=\sum_{k\in\Z^d}  \mathcal F(h\star h)(k).
\end{equation*}
The Poisson summation
formula on a lattice (see
\cite{stein-weiss}, Chapter~VII, Corollary~2.6) gives
\begin{equation*}
 \sum_{k\in\Z^d}  \mathcal F(h\star h)(k) =  \sum_{k\in\Z^d} h\star h (k).
 \end{equation*}
When $\nu\geq 1/2$, we have $1-(1-\mathrm{e}^{-\gamma
    x})^{2\nu}\leq 2\nu \mathrm{e}^{-\gamma x}$ for all $x>0$,
so   from \eqref{e:matern} and \eqref{ineq-bessel1},
\begin{equation}\label{convol}
h\star h (x)=\int_{\mathcal D} C_0(y)C_0(x-y)\, \mathrm{d}y \leq
4\rho^2\nu^2 \int_{\mathcal D}
\mathrm{e}^{-\frac{\gamma}{\alpha}(\|y\|+\|x-y\|)}\,\mathrm{d}y,\quad x\in\real^d,
 \end{equation}
where $\mathcal D=\mathcal D(x)=\{y \in\real^d:\|x-y\|_{\infty}>1/2,
\|y\|_{\infty}>1/2\}$ and $\|\cdot\|_\infty$ denotes the uniform
norm.

Suppose that $\nu\geq 1/2$.
When $d=1$, the latest integral in \eqref{convol} can be computed easily to get
\begin{equation*}
\int_{|y|>\frac{1}{2},|x-y|>\frac{1}{2}} \mathrm{e}^{-\frac{\gamma}{\alpha}(|y|+|x-y|)}\mathrm{d}y=\mathrm{e}^{-\frac{\gamma}{\alpha}|x|}\left(\frac{\alpha}{\gamma}\mathrm{e}^{-\frac{\gamma}{\alpha}}+|x|-1\right)
\end{equation*}
if $|x|\geq1$, and the value of the integral  at $x=0$ is
$\frac{\alpha}{\gamma} \mathrm{e}^{-\frac{\gamma}{\alpha}}$. Thereby, when $d=1$,
\begin{equation*}
\sum_{k\in\Z^d} (\alpha_k-\varphi(k))^2=\sum_{k\in\Z^d} h\star h (k)\leq 4\rho^2\nu^2\left[\frac{\alpha}{\gamma} \mathrm{e}^{-\frac{\gamma}{\alpha}}+2\sum_{k=1}^\infty \mathrm{e}^{-\frac{\gamma}{\alpha}k}\left(\frac{\alpha}{\gamma}\mathrm{e}^{-\frac{\gamma}{\alpha}}+k-1\right)\right]
\end{equation*}
and \eqref{e:mnb}, which involves the terms \eqref{cste} (for $\nu\geq
1/2$) and \eqref{error1}, follows from the
expansion  \[\sum_{k=1}^\infty (a+k)q^k=\frac{q}{1-q}
\left(a+\frac{1}{1-q}\right)\quad \mbox{for any $a\in\real$ and $|q|<1$.}\]
When $d\geq 2$, the integral in \eqref{convol} is more difficult to
compute and we therefore establish an upper bound as follows.
Since $\|y\|\geq(|y_1|+\dots +|y_d|)/\sqrt d$,
 \begin{align*}
h\star h (x)&\leq 4\rho^2\nu^2 \int_{\mathcal D} \prod_{j=1}^d \mathrm{e}^{-\frac{\gamma}{\alpha\sqrt d}(|y_j|+|x_j-y_j|)}\,\mathrm{d}y_j\\
&\leq 4\rho^2\nu^2 d\int_{|y_1-x_1|>\frac{1}{2}}\mathrm{e}^{-\frac{\gamma}{\alpha\sqrt d}(|y_1|+|x_1-y_1|)}\,\mathrm{d}y_1\prod_{j=2}^{d} \int_{\real}  \mathrm{e}^{-\frac{\gamma}{\alpha\sqrt d}(|y|+|x_j-y|)}\,\mathrm{d}y.
 \end{align*}
These integrals are computable: for any $\beta>0$,
 \begin{equation*}
 \int_{|y_1-x_1|>\frac{1}{2}}\mathrm{e}^{-\beta(|y|+|x_1-y|)}\,\mathrm{d}y=\begin{cases} \frac{\mathrm{e}^{-\beta}}{\beta}\textrm{cosh}(\beta x_1)\quad\textrm{if}\quad|x_1|\leq \frac{1}{2},\\ \mathrm{e}^{-\beta|x_1|}\left(\frac{1-\mathrm{e}^{-\beta}}{2\beta}+|x_1|-\frac{1}{2}\right)\quad\textrm{if}\quad|x_1|\geq \frac{1}{2}\end{cases}
 \end{equation*}
 and
 \begin{equation*}
 \int_{\real} \mathrm{e}^{-\beta(|y|+|x_j-y|)}\,\mathrm{d}y= \mathrm{e}^{-\beta|x_j|}\left(|x_j|+\frac{1}{\beta}\right).
 \end{equation*}
 Therefore, when $d\geq 2$, setting $\beta=\gamma/(\alpha\sqrt d)$,
 \begin{multline*}
\sum_{k\in\Z^d} (\alpha_k-\varphi(k))^2\\\leq 4\rho^2\nu^2 d \left( \frac{\mathrm{e}^{-\beta}}{\beta}+2\sum_{k=1}^{\infty}  \mathrm{e}^{-\beta k}\left(\frac{1-\mathrm{e}^{-\beta}}{2\beta}+k-\frac{1}{2}\right)\right)\left(\sum_{k\in\Z} \mathrm{e}^{-\beta|k|}\left(|k|+\frac{1}{\beta}\right)\right)^{d-1}
  \end{multline*}
and the bound \eqref{e:mnb}, which involves the term \eqref{error2},
follows after a straightforward calculation.

Suppose that $\nu\leq 1/2$. From \eqref{ineq-bessel2} we deduce
 \begin{equation*}
h\star h (x)\leq \frac{\rho^2}{\Gamma(\nu)^2}2^{2-2\nu}\frac{\pi}{2}\int_{\mathcal D} \| y/\alpha\|^{\nu-\frac{1}{2}} \| (x- y)/\alpha\|^{\nu-\frac{1}{2}} \mathrm{e}^{-\frac{1}{\alpha}(\|y\|+\| x-y\|)}\mathrm{d}y.
 \end{equation*}
If $\| x\|_\infty>1/2$, then $\|x\|>1/2$, and so $\|x\|^{\nu-1/2}<2^{1/2-\nu}$ and
\begin{equation*}
h\star h (x)\leq \frac{\rho^2}{\Gamma(\nu)^2} 2^{2-4\nu}\alpha^{1-2\nu}\pi\int_{\mathcal D}\mathrm{e}^{-\frac{1}{\alpha}(\|y\|+\| x-y\|)}\,\mathrm{d}y.
 \end{equation*}
The latter integral may be bounded similarly as the one in
\eqref{convol}, and thereby \eqref{e:mnb}, which involves the term
\eqref{cste},  follows.
\end{proof}

Note that
the inequality \eqref{e:mnb} reduces to an equality in the
particular case $d=1$ and $\nu=1/2$. Finally, the plots in
Figure~\ref{fig:error}
confirm that for reasonable values of $\rho$, $\nu$, and $\alpha$
satisfying \eqref{eq:maxrhoWM}, the error bound \eqref{e:mnb} is
small.

  \begin{figure}[htbp]
    \setlength{\tabcolsep}{0cm} \centerline{
      \begin{tabular}[]{ccc}
  \includegraphics[angle=0,scale=.4]{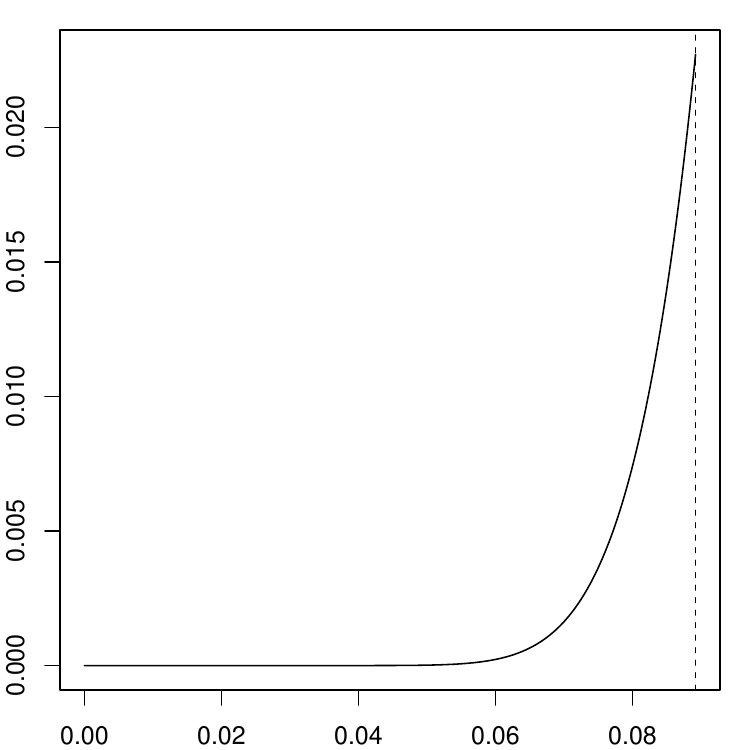} &
  \includegraphics[angle=0,scale=.4]{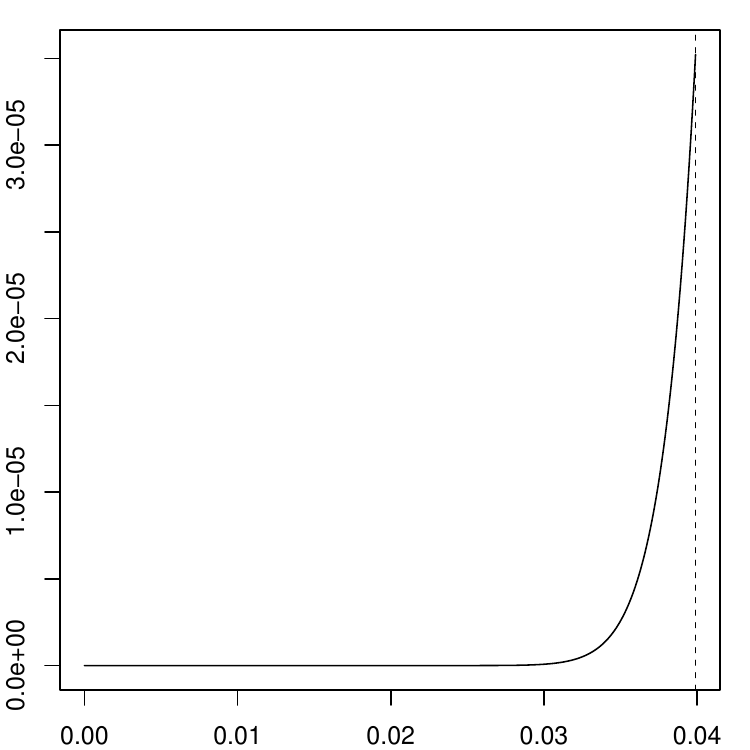} &
   \includegraphics[angle=0,scale=.4]{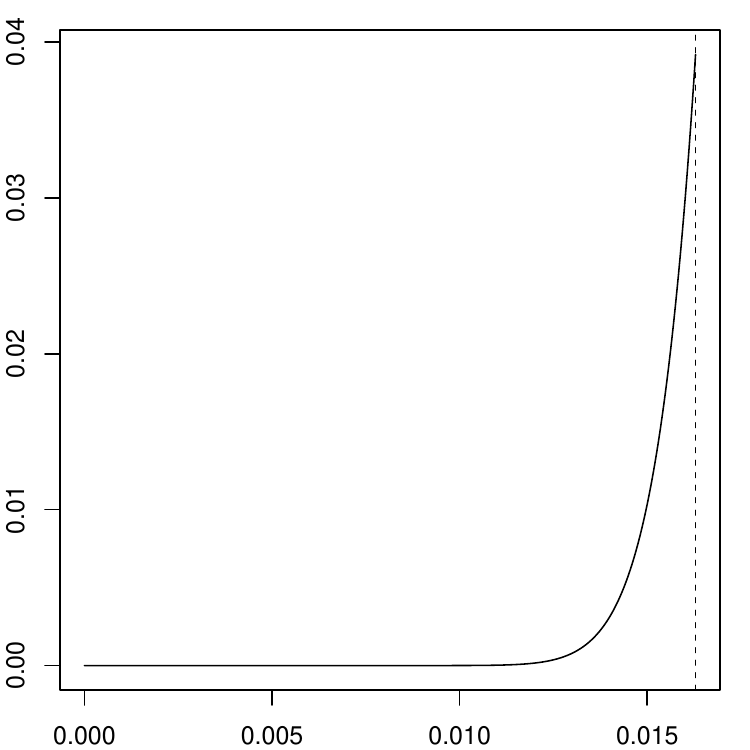} \\
  \footnotesize {$\rho=100$, $\nu=0.1$} &
  \footnotesize {$\rho=100$, $\nu=0.5$} & \footnotesize {$\rho=100$, $\nu=3$} \\
  \includegraphics[angle=0,scale=.4]{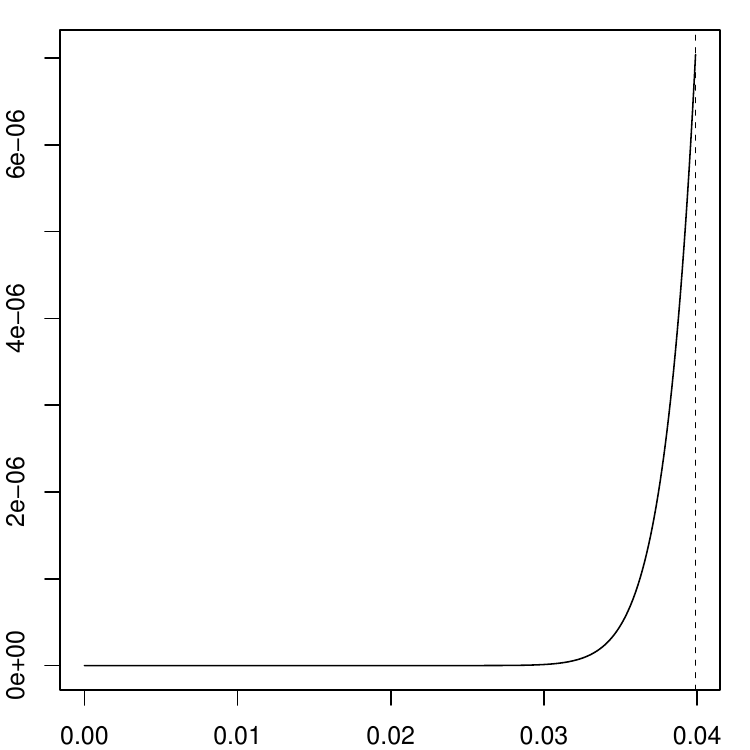} &
  \includegraphics[angle=0,scale=.4]{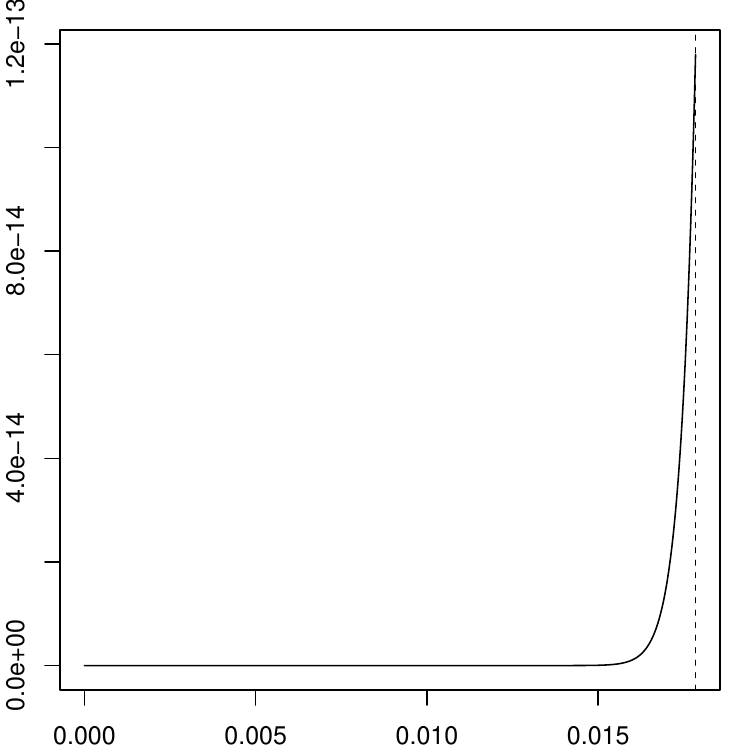} &
  \includegraphics[angle=0,scale=.4]{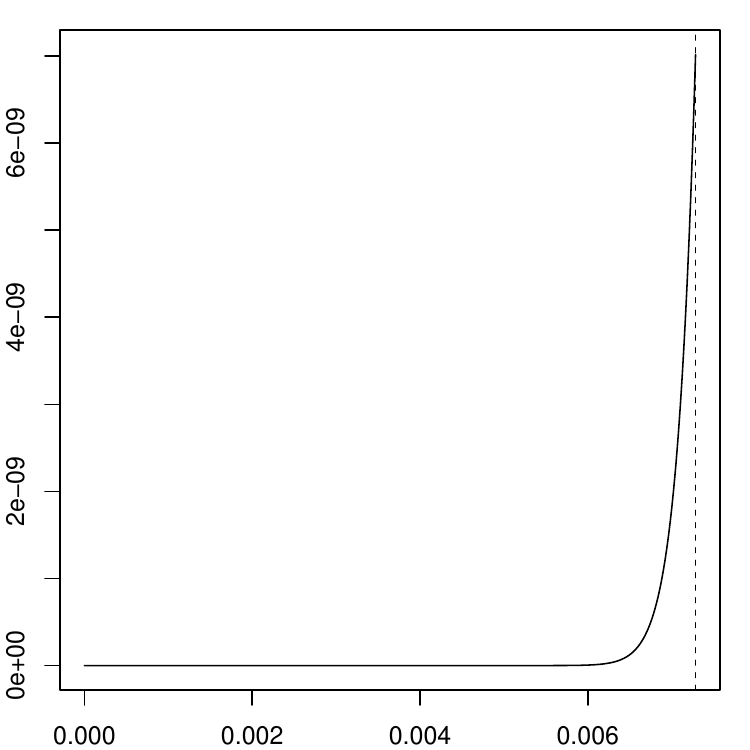}\\
  \footnotesize {$\rho=500$, $\nu=0.1$} &
  \footnotesize {$\rho=500$, $\nu=0.5$} & \footnotesize {$\rho=500$, $\nu=3$}
      \end{tabular}
  }
  \caption{{\small Error-bound \eqref{e:mnb} in terms of $\alpha$ for different values of $\rho$ and $\nu$, when $d=2$. The dotted line represents the maximal possible value of $\alpha$ following from \eqref{eq:maxrhoWM}.}}\label{fig:error}
  \end{figure}

\section{Alternative approximation of the density}\label{sec:approx3}

Let $S\subset\real^d$ be compact. In this appendix, in addition
to Assumption~\ref{e:assumption}, we assume the slightly stronger
condition that the spectral density $\varphi$ is strictly less than
$1$. This ensures that
all eigenvalues $\lambda_k$ are strictly less
than $1$ for all index $k$ so that the density $f$ in
Theorem~\ref{thm:density} is well-defined.
Recall that $f$ is given in terms of
$\Ct$ and $D$, cf.\ \eqref{e:unconddensity}. Below we introduce
computationally convenient approximations of $\Ct$ and $D$ which can
be used with \eqref{e:unconddensity} to obtain an approximation of
$f$.

\subsection{Convolution approximation of $f$}\label{sec:convolutionapprox}

We start by showing that $\Ctappstat$ given by
\begin{equation}\label{eq:Ctilde0conv}
  \Ctappstat(u) = \sum_{k=1}^\infty C_0^{\conv k}(u), \quad
  u\in\real^d,
\end{equation}
is well-defined, where
\begin{equation}\label{e:convdef}
C_0^{\conv 1}(u) = C_0(u),\quad C_0^{\conv k}(u)=\int
C_0^{\conv(k-1)}(x)C_0(u-x)\,\dee x,\quad u\in \real^d,\ k=2,3,\ldots.
\end{equation}
Since $0\le\varphi<1$ and $\varphi\in L^1(\real^d)$, for all
$p\in[1,\infty]$, we have $\varphi\in
L^p(\real^d)$. Define $\tilde\varphi = \varphi/(1-\varphi)$.
For any $u\in\real^d$, $\varphi(u) = \lim_{n\to\infty}
\tilde\varphi_n(u)$, where $\tilde\varphi_n(u)=\sum_{k=1}^n
\varphi(u)^k$. We see that $\tilde\varphi\in L^1(\real^d)$ since
\begin{equation*}
  \|\tilde\varphi\|_1 = \int\tilde\varphi(u)\dee u = \sum_{k=1}^\infty
  \int\varphi(u)^k\dee u \le \sum_{k=1}^\infty
  \|\varphi\|_\infty^{k-1}\int\varphi(u)\dee u = \frac{\|\varphi\|_1}{1-\|\varphi\|_\infty}
   < \infty
\end{equation*}
using the monotone convergence theorem to swap summation and
integration to obtain the second identity. Therefore $\mathcal F^{-1}\tilde\varphi$ is well-defined. Using the dominated convergence theorem and similar arguments
as above, we see that $\left(\mathcal F^{-1}\tilde\varphi\right)(u)$ is
equal to the right hand side of \eqref{eq:Ctilde0conv}.

For $x,y\in S$, we define $\Ctapp(x,y)=\Ctappstat(x-y)$ and use the approximation $\Ct(x,y)\approx\Ctapp(x,y)$. The expansion
\eqref{eq:Ctilde0conv} corresponds to \eqref{eq:CtSUM} with
$C^k_S(x,y)$ substituted by $C_0^{\conv k}(x-y)$.

Using the same
substitution in \eqref{e:DD} leads us to approximate $D$ by
\begin{equation}\label{eq:Dconv}
  \Dapp = |S| \sum_{k=1}^\infty C_0^{\conv k}(0)/k.
\end{equation}
Since $C_0^{\conv k}(0)=\int\varphi(u)^k\dee u$, we obtain an
alternative expression for $\Dapp$ by applying the
monotone convergence theorem,
\begin{equation*}
  \Dapp = |S| \int -\log(1-\varphi(u))\dee u  =
  |S| \int \log(1+\tilde\varphi(u))\dee u .
\end{equation*}
Then the convolution approximation of $f$ is defined by
\begin{equation}\label{e:unconddensityapp}
  \fapp(\{x_1,\ldots,x_n\}) = \exp({|S|-\Dapp}) \det
  [\Ctapp](x_1,\ldots,x_n).
\end{equation}
As mentioned above, the approximations $\Ctapp$ and $\Dapp$ involve
approximating $C_S^k(x,y)$ by $C_0^{\conv k}(u)$, where $u=x-y$. In
fact the approximations provide upper bounds, since $C_S^k(x,y)\le
C_0^{\conv k}(u)$ for all $x,y$ and $k$. Heuristically, when
approximating $C_S^k(x,y)$ by $C_0^{\conv k}(u)$, we expect that the
relative error increases as $k$ grows, since the approximation is
applied iteratively, cf.\ \eqref{e:fnat} and
\eqref{e:convdef}. However, the final approximations
$\Ct\approx\Ctapp$ and $D\approx\Dapp$ involve sums of $C_S^k(x,y)$
and $C_0^{\conv k}(u)$, and the terms with a large relative error may
only have a small effect if $C_0^{\conv k}(u)$ tends to zero
sufficiently fast for $k\to\infty$. Since $C_0^{\conv k}$ is a
covariance function, we have $C_0^{\conv k}(u)\le C_0^{\conv k}(0)$
for all $k=1,2,\dots$. Consequently, we expect that the accuracy of
approximating $f$ by $\fapp$ depends on how fast $C_0^{\conv k}(0)$
tends to zero. This is further discussed in the examples below.

\subsection{Examples}

To use the density approximation $\fapp$ in practice we truncate the
sums in \eqref{eq:Ctilde0conv} and \eqref{eq:Dconv}, i.e.\
\begin{equation*}
  \Ctappstat(u) \approx \sum_{k=1}^N C_0^{\conv k}(u)
  \quad\text{and}\quad \Dapp \approx |S| \sum_{k=1}^N C_0^{\conv k}(0)/k
\end{equation*}
where $N$ is a positive integer. Furthermore, we need closed form
expressions for $C_0^{\conv k}(u)$. For the normal variance mixture
models presented in Section~\ref{sec:examples}, we have $C_0^{\conv
  k}(u)=(\rho/\rhomax)^k h^{\conv k}(u)$, and so it suffices to find
closed form expressions for $h^{\conv k}$. For the Gaussian model,
\begin{equation*}
  h^{\conv k}(u) = (k\pi\alpha^2)^{-d/2}\exp(-\|u/\alpha\|^2/k),\quad
  u\in\real^d,
\end{equation*}
while for the \wm{} model,
\begin{equation*}
  h^{\conv k}(u)=\frac{\|u/\alpha\|^{\nu'}
    K_{\nu'}(\|u/\alpha\|)}{2^{\nu'-1}(\sqrt{\pi}\alpha)^{d}\Gamma(\nu'+d/2)},\quad
  u\in\real^d,
\end{equation*}
where $\nu'=k(\nu+d/2)-d/2$. We have no closed form expression for the
Cauchy model.

For both the Gaussian and the \wm{} covariance function, $h^{\conv
  k}(0)$ decays as $k^{-d/2}$ when $k\to\infty$, and therefore the
rate of convergence of $\Ctappstat(0) = \sum_{k=1}^\infty
(\rho/\rhomax)^k h^{\conv k}(0)$ depends crucially on $d$ and $\rho$.
For $d<3$, the series only converges if $\rho<\rhomax$, and the series
converges slowly when $\rho$ is close to
$\rhomax$.

\begin{figure}[!htbp]%
  \centering
  \includegraphics[scale=.45]{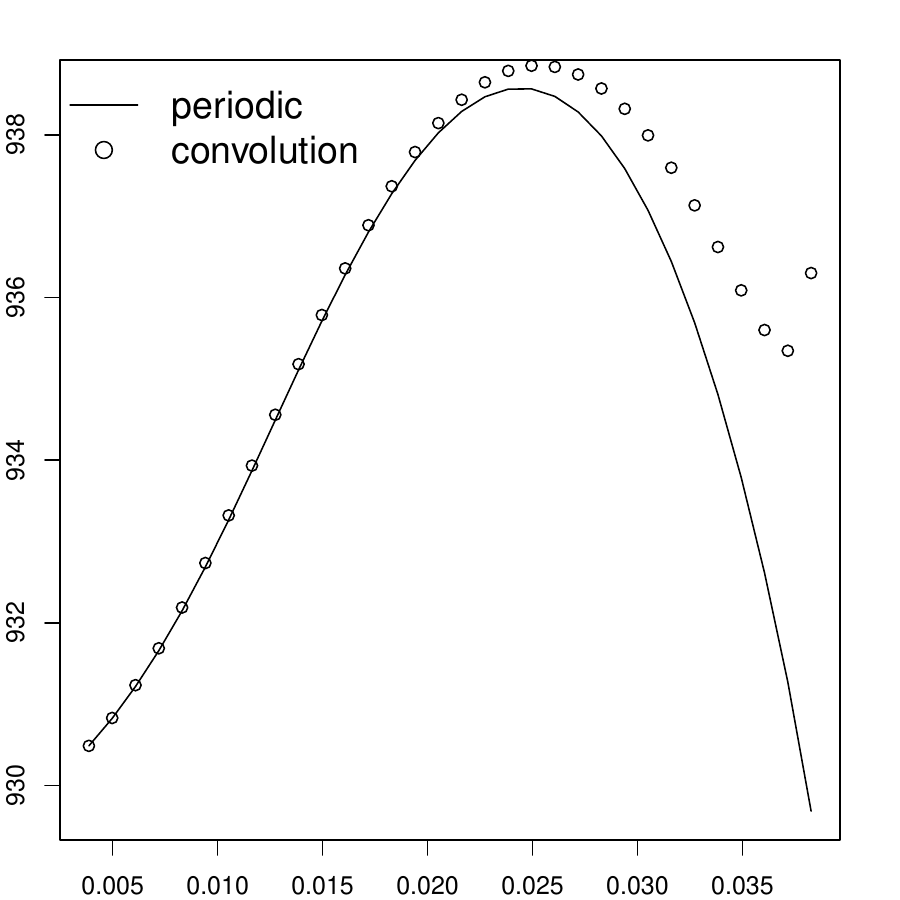}
  \caption{ Comparison of using the convolution and periodic density
    approximations to approximate the log-likelihood of the Gaussian model
    as a function of $\alpha$ based on a simulated dataset in the unit
    square with $\rho=200$ and $\alpha=0.02$. }
  \label{fig:comparedensities}
\end{figure}

Based on a simulated point pattern in the unit square,
Figure~\ref{fig:comparedensities} compares the approximations obtained
using the convolution and
periodic density approximations to approximate the log-likelihood for
the Gaussian model
with $\rho=200$ and $\alpha=0.02$. The simulated
point pattern has $\hat\rho=213$
points. In the likelihood calculations, $\rho=\hat\rho$ is fixed
such that the only varying parameter is $\alpha\in(0,\alpha_{\max})$,
where $\alpha_{\max}=1/\sqrt{\pi\hat\rho}=0.39$. For both
approximations, the truncation $N$ was increased until almost no change
appeared in the approximations. In this example, $N=256$ for the
convolution approximation and $N=512$ for the periodic
approximation. As in the simulation study in
Section~\ref{sec:simstudy},
the periodic approximation is giving effectively unbiased
estimates. However, similar simulation studies (not reported here)
using the convolution approximation yielded estimates of $\alpha$
which were positively biased, which is in agreement with
Figure~\ref{fig:comparedensities}. In particular using the convolution
approximation we get a large proportion of estimates with
$\hat\alpha=\alpha_{\max}$. For smaller values of $\alpha$,
$\alpha<\alpha_{\max}/2$ say, the two approximations are very similar,
and in this case $\rho/\rhomax<1/2$, so the convolution approximation
converges rapidly, and in this range of $\alpha$-values, a truncation of $N=10$
is sufficient to obtain stable results. This is computationally much
faster than using the periodic approximation with $N=512$, and
therefore the convolution approximation is appealing when
$\rho/\rhomax$ is small.

\bibliography{bibliography}

\begin{thebibliography}{}

\bibitem[\protect\citeauthoryear{Affandi, Kulesza, and Fox}{Affandi
  et~al.}{2012}]{Affandi:Kulesza:Fox:12}
Affandi, R., A.~Kulesza, and E.~Fox (2012).
\newblock Markov deteriminantal point processes.
\newblock In K.~M. N.~de Freitas (Ed.), {\em Proceedings of the Twenty-Eight
  Conference on Uncertainty in Artificial Intelligence (UAI-12)}, Corvallis,
  Oregon, pp.\  26--35. AUAI Press.

\bibitem[\protect\citeauthoryear{Alzer}{Alzer}{1997}]{alzer:97}
Alzer, H. (1997).
\newblock On some inequalities for the incomplete {G}amma function.
\newblock {\em Mathematics of Computation\/}~{\em 66}, 771--778.

\bibitem[\protect\citeauthoryear{Baddeley, M{\o}ller, and
  Waagepetersen}{Baddeley et~al.}{2000}]{baddeley:moeller:waagepetersen:97}
Baddeley, A., J.~M{\o}ller, and R.~Waagepetersen (2000).
\newblock Non- and semi-parametric estimation of interaction in inhomogeneous
  point patterns.
\newblock {\em Statistica Neerlandica\/}~{\em 54}, 329--350.

\bibitem[\protect\citeauthoryear{Baddeley and Turner}{Baddeley and
  Turner}{2000}]{baddeley:turner:00}
Baddeley, A. and R.~Turner (2000).
\newblock Practical maximum pseudolikelihood for spatial point patterns.
\newblock {\em Australian and New Zealand Journal of Statistics\/}~{\em 42},
  283--322.

\bibitem[\protect\citeauthoryear{Baddeley and Turner}{Baddeley and
  Turner}{2005}]{baddeley:turner:05}
Baddeley, A. and R.~Turner (2005).
\newblock Spatstat: an {{\sf R}} package for analyzing spatial point patterns.
\newblock {\em Journal of Statistical Software\/}~{\em 12\/}(6), 1--42.
\newblock {URL}: {{\tt www.jstatsoft.org}}, {ISSN}: 1548-7660.

\bibitem[\protect\citeauthoryear{Barndorff-Nielsen}{Barndorff-Nielsen}{1977}]{OBN:77}
Barndorff-Nielsen, O. (1977).
\newblock Exponentially decreasing distributions for the logarithm of particle
  size.
\newblock {\em Proceedings of the Royal Society of London, Series A,
  Mathematical and Physical Sciences\/}~{\em 353}, 401--419.

\bibitem[\protect\citeauthoryear{Barndorff-Nielsen}{Barndorff-Nielsen}{1978}]{OBN:78}
Barndorff-Nielsen, O. (1978).
\newblock Hyperbolic distributions and distributions on hyperbolae.
\newblock {\em Scandinavian Journal of Statistics\/}~{\em 5}, 151--157.

\bibitem[\protect\citeauthoryear{Barot, Gignoux, and Menaut}{Barot
  et~al.}{1999}]{barot:etal:99}
Barot, S., J.~Gignoux, and J.-C. Menaut (1999).
\newblock Demography of a savanna palm tree: predictions from comprehensive
  spatial pattern analyses.
\newblock {\em Ecology\/}~{\em 80}, 1987--2005.

\bibitem[\protect\citeauthoryear{Besag}{Besag}{1977a}]{besag:77b}
Besag, J. (1977a).
\newblock Some methods of statistical analysis for spatial data.
\newblock {\em Bulletin of the International Statistical Institute\/}~{\em 47},
  77--92.

\bibitem[\protect\citeauthoryear{Besag}{Besag}{1977b}]{Besag:77a}
Besag, J.~E. (1977b).
\newblock Contribution to the discussion of {D}r {R}ipley's paper.
\newblock {\em Journal of Royal Statistical Society, Series B, Statistical
  Methodology\/}~{\em 39}, 193--195.

\bibitem[\protect\citeauthoryear{Birman and Solomjak}{Birman and
  Solomjak}{1987}]{birman:solomjak}
Birman, M.~S. and Z.~Solomjak (1987).
\newblock {\em Spectral theory of Self-Adjoint Operator in Hilbert Space}.
\newblock Dortrecht: D. Reidel Publishing Company.

\bibitem[\protect\citeauthoryear{Chakraborty, Gelfand, Wilson, Latimer, and
  Silander}{Chakraborty et~al.}{2011}]{Chakraborty11}
Chakraborty, A., A.~E. Gelfand, A.~M. Wilson, A.~M. Latimer, and J.~A. Silander
  (2011).
\newblock Point pattern modelling for degraded presence-only data over large
  regions.
\newblock {\em Journal of Royal Statistical Society, Series C, Applied
  Statistics\/}~{\em 60}, 757--776.

\bibitem[\protect\citeauthoryear{{De laco}, Palma, and Posa}{{De laco}
  et~al.}{2003}]{laco:palma:posa:03}
{De laco}, S., M.~Palma, and D.~Posa (2003).
\newblock Covariance functions and models for complex-valued random fields.
\newblock {\em Stochastic Environmental Research and Risk Assessment\/}~{\em
  17}, 145--156.

\bibitem[\protect\citeauthoryear{Diggle}{Diggle}{2003}]{Diggle:03}
Diggle, P. (2003).
\newblock {\em Statistical Analysis of Spatial Point Patterns\/} (Second ed.).
\newblock London: Hodder Arnold.

\bibitem[\protect\citeauthoryear{Diggle and Gratton}{Diggle and
  Gratton}{1984}]{Diggle:Gratton:84}
Diggle, P. and R.~Gratton (1984).
\newblock {M}onte {C}arlo methods of inference for implicit statistical models
  (with discussion).
\newblock {\em Journal of Royal Statistical Society, Series B, Statistical
  Methodology\/}~{\em 46}, 193--227.

\bibitem[\protect\citeauthoryear{Diggle, Menezes, and Su}{Diggle
  et~al.}{2010}]{Diggle10}
Diggle, P.~J., R.~Menezes, and T.-L. Su (2010).
\newblock Geostatistical inference under preferential sampling.
\newblock {\em Journal of Royal Statistical Society, Series C, Applied
  Statistics\/}~{\em 59}, 191--232.

\bibitem[\protect\citeauthoryear{Gelfand, Diggle, Guttorp, and Fuentes}{Gelfand
  et~al.}{2010}]{Gelfand:etal:10}
Gelfand, A.~E., P.~J. Diggle, P.~Guttorp, and M.~Fuentes (2010).
\newblock {\em Handbook of Spatial Statistics}.
\newblock CRC Press, Boca Raton.

\bibitem[\protect\citeauthoryear{Georgii and Yoo}{Georgii and
  Yoo}{2005}]{Georgii:Yoo:05}
Georgii, H.-O. and H.~J. Yoo (2005).
\newblock Conditional intensity and {G}ibbsianness of determinantal point
  processes.
\newblock {\em Journal of Statistical Physics\/}~{\em 118}, 617--666.

\bibitem[\protect\citeauthoryear{Glass and Tobler}{Glass and
  Tobler}{1971}]{Glass:Tobler:71}
Glass, L. and W.~R. Tobler (1971).
\newblock Uniform distribution of objects in a ho mogeneous field: Cities on a
  plain.
\newblock {\em Nature\/}~{\em 233}, 67--68.

\bibitem[\protect\citeauthoryear{Gneiting}{Gneiting}{1997}]{gneiting:97}
Gneiting, T. (1997).
\newblock Normal scale mixtures and dual probability densities.
\newblock {\em Journal of Statistical Computation and Simulation\/}~{\em 59},
  375--384.

\bibitem[\protect\citeauthoryear{Gneiting}{Gneiting}{2002}]{Gneiting:02}
Gneiting, T. (2002).
\newblock Compactly supported correlation functions.
\newblock {\em Journal of Multivariate Analysis\/}~{\em 83}, 493 -- 508.

\bibitem[\protect\citeauthoryear{Goovaerts}{Goovaerts}{1997}]{goovaerts:97}
Goovaerts, P. (1997).
\newblock {\em Geostatistics for Natural Resources Evaluation}.
\newblock New York: Oxford University.

\bibitem[\protect\citeauthoryear{Gradshteyn and Ryzhik}{Gradshteyn and
  Ryzhik}{2007}]{GR}
Gradshteyn, I.~S. and I.~M. Ryzhik (2007).
\newblock {\em Table of Integrals, Series, and Products\/} (7th ed.).
\newblock San Diego: Academic Press.

\bibitem[\protect\citeauthoryear{Harkness and Isham}{Harkness and
  Isham}{1983}]{Harkness:Isham:83}
Harkness, R. and V.~Isham (1983).
\newblock A bivariate spatial point pattern of ants’ nests.
\newblock {\em Applied Statistics\/}~{\em 32}, 293--303.

\bibitem[\protect\citeauthoryear{Hough, Krishnapur, Peres, and Vir\`{a}g}{Hough
  et~al.}{2006}]{Hough:etal:06}
Hough, J.~B., M.~Krishnapur, Y.~Peres, and B.~Vir\`{a}g (2006).
\newblock Determinantal processes and independence.
\newblock {\em Probability Surveys\/}~{\em 3}, 206--229.

\bibitem[\protect\citeauthoryear{Hough, Krishnapur, Peres, and Vir\`{a}g}{Hough
  et~al.}{2009}]{Hough:etal:09}
Hough, J.~B., M.~Krishnapur, Y.~Peres, and B.~Vir\`{a}g (2009).
\newblock {\em Zeros of Gaussian Analytic Functions and Determinantal Point
  Processes}.
\newblock Providence: American Mathematical Society.

\bibitem[\protect\citeauthoryear{Huang and Ogata}{Huang and
  Ogata}{1999}]{Huang:Ogata:99}
Huang, F. and Y.~Ogata (1999).
\newblock Improvements of the maximum pseudo-likelihood estimators in various
  spatial statistical models.
\newblock {\em Journal of Computational and Graphical Statistics\/}~{\em
  8\/}(3), 510--530.

\bibitem[\protect\citeauthoryear{Illian, Penttinen, Stoyan, and Stoyan}{Illian
  et~al.}{2008}]{illian:penttinen:stoyan:stoyan:08}
Illian, J., A.~Penttinen, H.~Stoyan, and D.~Stoyan (2008).
\newblock {\em Statistical Analysis and Modelling of Spatial Point Patterns}.
\newblock John Wiley and Sons, Chichester.

\bibitem[\protect\citeauthoryear{Jensen and M{\o}ller}{Jensen and
  M{\o}ller}{1991}]{jensen:moeller:91}
Jensen, J.~L. and J.~M{\o}ller (1991).
\newblock Pseudolikelihood for exponential family models of spatial point
  processes.
\newblock {\em Annals of Applied Probability\/}~{\em 1}, 445--461.

\bibitem[\protect\citeauthoryear{Kelly and Ripley}{Kelly and
  Ripley}{1976}]{kelly:ripley:76}
Kelly, F.~P. and B.~D. Ripley (1976).
\newblock A note on {S}trauss' model for clustering.
\newblock {\em Biometrika\/}~{\em 63}, 357--360.

\bibitem[\protect\citeauthoryear{Kendall and M{\o}ller}{Kendall and
  M{\o}ller}{2000}]{kendall00}
Kendall, W.~S. and J.~M{\o}ller (2000).
\newblock Perfect simulation using dominating processes on ordered spaces, with
  application to locally stable point processes.
\newblock {\em Advances in Applied Probability\/}~{\em 32}, 844--865.

\bibitem[\protect\citeauthoryear{Kulesza and Taskar}{Kulesza and
  Taskar}{2012}]{Kulesza:Taskar:12}
Kulesza, A. and B.~Taskar (2012).
\newblock Determinantal point processes for machine learning.
\newblock {\em Foundations and Trends in Machine Learning\/}~{\em 5}, 123--286.

\bibitem[\protect\citeauthoryear{Kuna, Lebowitz, and Speer}{Kuna
  et~al.}{2007}]{kuna07}
Kuna, T., J.~Lebowitz, and E.~Speer (2007).
\newblock Realizability of point processes.
\newblock {\em Journal of Statistical Physics\/}~{\em 129}, 417--439.

\bibitem[\protect\citeauthoryear{Leonardi and Torrisi}{Leonardi and
  Torrisi}{2013}]{Leonardi13}
Leonardi, E. and G.~L. Torrisi (2013).
\newblock Large deviations of the interference in the ginibre network model.
\newblock {\em preprint (arxiv:1304.2234)\/}, 1--35.

\bibitem[\protect\citeauthoryear{{Lieshout}}{{Lieshout}}{2000}]{lieshout:00}
{Lieshout}, M. N. M.~v. (2000).
\newblock {\em Markov Point Processes and Their Applications}.
\newblock Imperial College Press, London.

\bibitem[\protect\citeauthoryear{Lindgren, Lindstr{\"o}m, and Rue}{Lindgren
  et~al.}{2011}]{Lindgren:etal:11}
Lindgren, F., J.~Lindstr{\"o}m, and H.~Rue (2011).
\newblock An explicit link between {G}aussian fields and {G}aussian {M}arkov
  random fields: the stochastic partial differential equation approach.
\newblock {\em Journal of Royal Statistical Society, Series B, Statistical
  Methodology\/}~{\em 73}, 423–--498.

\bibitem[\protect\citeauthoryear{Macchi}{Macchi}{1975}]{Macchi:75}
Macchi, O. (1975).
\newblock The coincidence approach to stochastic point processes.
\newblock {\em Advances in Applied Probability\/}~{\em 7}, 83--122.

\bibitem[\protect\citeauthoryear{McCullagh and M{\o}ller}{McCullagh and
  M{\o}ller}{2006}]{McCullagh:Moeller:06}
McCullagh, P. and J.~M{\o}ller (2006).
\newblock The permanental process.
\newblock {\em Advances in Applied Probability\/}~{\em 38}, 873--888.

\bibitem[\protect\citeauthoryear{Mecke and Stoyan}{Mecke and
  Stoyan}{2005}]{Mecke:Stoyan:05}
Mecke, K.~R. and D.~Stoyan (2005).
\newblock Morphological characterization of point patterns.
\newblock {\em Biometrical Journal\/}~{\em 47}, 473--488.

\bibitem[\protect\citeauthoryear{Miyoshi and Shirai}{Miyoshi and
  Shirai}{2013}]{Miyoshi:Shirai13}
Miyoshi, N. and T.~Shirai (2013).
\newblock A cellular network model with ginibre configurated base stations.
\newblock Technical report, Department of Mathematical and Computing Sciences
  Tokyo Institute of Technology, series B: Applied Mathematical Science.

\bibitem[\protect\citeauthoryear{M{\o}ller and Waagepetersen}{M{\o}ller and
  Waagepetersen}{2004}]{moeller:waagepetersen:00}
M{\o}ller, J. and R.~P. Waagepetersen (2004).
\newblock {\em Statistical Inference and Simulation for Spatial Point
  Processes}.
\newblock Boca Raton: Chapman and Hall/CRC.

\bibitem[\protect\citeauthoryear{M{\o}ller and Waagepetersen}{M{\o}ller and
  Waagepetersen}{2007}]{moeller:waagepetersen:07}
M{\o}ller, J. and R.~P. Waagepetersen (2007).
\newblock Modern spatial point process modelling and inference (with
  discussion).
\newblock {\em Scandinavian Journal of Statistics\/}~{\em 34}, 643--711.

\bibitem[\protect\citeauthoryear{{Myllym{\"a}ki}, {Mrkvicka}, {Seijo}, and
  {Grabarnik}}{{Myllym{\"a}ki} et~al.}{2013}]{myllymaki:etal:13}
{Myllym{\"a}ki}, M., T.~{Mrkvicka}, H.~{Seijo}, and P.~{Grabarnik} (2013).
\newblock {Global envelope tests for spatial processes}.
\newblock Preprint on arxiv:1307.0239.

\bibitem[\protect\citeauthoryear{Nelder and Mead}{Nelder and
  Mead}{1965}]{Nelder:Mead:65}
Nelder, J.~A. and R.~Mead (1965).
\newblock A simplex method for function minimization.
\newblock {\em Computer Journal\/}~{\em 7}, 308--313.

\bibitem[\protect\citeauthoryear{Numata}{Numata}{1964}]{numata:64}
Numata, M. (1964).
\newblock Forest vegetation, particularly pine stands in the vicinity of
  {Choshi ---} flora and vegetation in {Choshi, Chiba} prefecture, {VI (in
  Japanese)}.
\newblock {\em Bulletin of the Choshi Marine Laboratory\/}~(6), 27--37.
\newblock Chiba University.

\bibitem[\protect\citeauthoryear{Ogata and Tanemura}{Ogata and
  Tanemura}{1986}]{ogata:tanemura:86}
Ogata, Y. and M.~Tanemura (1986).
\newblock Likelihood estimation of interaction potentials and external fields
  of inhomogeneous spatial point patterns.
\newblock In I.~Francis, B.~Manly, and F.~Lam (Eds.), {\em Pacific Statistical
  Congress}, Amsterdam, pp.\  150--154. Elsevier.

\bibitem[\protect\citeauthoryear{Pommerening}{Pommerening}{2002}]{Pommerening:02}
Pommerening, A. (2002).
\newblock Approaches to quantifying forest structures.
\newblock {\em Forestry\/}~{\em 75}, 305--324.

\bibitem[\protect\citeauthoryear{Proke{\u s}ov{\'a} and Jensen}{Proke{\u
  s}ov{\'a} and Jensen}{2013}]{Prok:Jensen:10}
Proke{\u s}ov{\'a}, M. and E.~B.~V. Jensen (2013).
\newblock Asymptotic {P}alm likelihood theory for stationary point processes.
\newblock {\em Annals of the Institute of Statistical Mathematics\/}~{\em 65},
  387--412.

\bibitem[\protect\citeauthoryear{{R Development Core Team}}{{R Development Core
  Team}}{2011}]{R:11}
{R Development Core Team} (2011).
\newblock {\em R: A Language and Environment for Statistical Computing}.
\newblock Vienna, Austria: R Foundation for Statistical Computing.
\newblock {ISBN} 3-900051-07-0.

\bibitem[\protect\citeauthoryear{Riesz and Sz.-Nagy}{Riesz and
  Sz.-Nagy}{1990}]{riesz:nagy}
Riesz, F. and B.~Sz.-Nagy (1990).
\newblock {\em Functional Analysis}.
\newblock New York: Dover Publications.

\bibitem[\protect\citeauthoryear{Ripley}{Ripley}{1976}]{ripley:76}
Ripley, B.~D. (1976).
\newblock The second-order analysis of stationary point processes.
\newblock {\em Journal of Applied Probability\/}~{\em 13}, 255--266.

\bibitem[\protect\citeauthoryear{Ripley}{Ripley}{1977}]{ripley:77}
Ripley, B.~D. (1977).
\newblock Modelling spatial patterns (with discussion).
\newblock {\em Journal of Royal Statistical Society, Series B, Statistical
  Methodology\/}~{\em 39}, 172--212.

\bibitem[\protect\citeauthoryear{Ripley}{Ripley}{1988}]{Ripley:88}
Ripley, B.~D. (1988).
\newblock {\em Statistical Inference for Spatial Processes}.
\newblock Cambridge University Press.

\bibitem[\protect\citeauthoryear{Ripley and Kelly}{Ripley and
  Kelly}{1977}]{ripley:kelly:77}
Ripley, B.~D. and F.~P. Kelly (1977).
\newblock Markov point processes.
\newblock {\em Journal of the London Mathematical Society\/}~{\em 15},
  188--192.

\bibitem[\protect\citeauthoryear{Scardicchio, Zachary, and
  Torquato}{Scardicchio et~al.}{2009}]{Scardicchio:etal:09}
Scardicchio, A., C.~Zachary, and S.~Torquato (2009).
\newblock Statistical properties of determinantal point processes in
  high-dimensional {E}uclidean spaces.
\newblock {\em Physical Review E\/}~{\em 79\/}(4), Article 041108.

\bibitem[\protect\citeauthoryear{Schoenberg}{Schoenberg}{2005}]{Schoenberg05}
Schoenberg, F. (2005).
\newblock Consistent parametric estimation of the intensity of a
  spatial-temporal point process.
\newblock {\em Journal of Statistical Planning and Inference\/}~{\em 128},
  79--93.

\bibitem[\protect\citeauthoryear{Shirai and Takahashi}{Shirai and
  Takahashi}{2003}]{shirai:takahash:03}
Shirai, T. and Y.~Takahashi (2003).
\newblock Random point fields associated with certain {F}redholm determinants.
  {I}. {F}ermion, {P}oisson and boson point processes.
\newblock {\em Journal of Functional Analysis\/}~{\em 2}, 414--463.

\bibitem[\protect\citeauthoryear{Soshnikov}{Soshnikov}{2000}]{Soshnikov:00}
Soshnikov, A. (2000).
\newblock Determinantal random point fields.
\newblock {\em Russian Mathematical Surveys\/}~{\em 55}, 923--975.

\bibitem[\protect\citeauthoryear{Stein and Weiss}{Stein and
  Weiss}{1971}]{stein-weiss}
Stein, E.~M. and G.~Weiss (1971).
\newblock {\em Introduction to Fourier Analysis on Euclidean Spaces}.
\newblock Princeton: Princeton University Press.

\bibitem[\protect\citeauthoryear{Stein}{Stein}{2005}]{Stein05}
Stein, M.~L. (2005).
\newblock Space–-time covariance functions.
\newblock {\em Journal of the American Statistical Association\/}~{\em 100},
  310--321.

\bibitem[\protect\citeauthoryear{Stoyan, Kendall, and Mecke}{Stoyan
  et~al.}{1995}]{stoyan:kendall:mecke:95}
Stoyan, D., W.~S. Kendall, and J.~Mecke (1995).
\newblock {\em Stochastic Geometry and Its Applications\/} (Second ed.).
\newblock Chichester: Wiley.

\bibitem[\protect\citeauthoryear{Strauss}{Strauss}{1975}]{strauss:75}
Strauss, D.~J. (1975).
\newblock A model for clustering.
\newblock {\em Biometrika\/}~{\em 63}, 467--475.

\bibitem[\protect\citeauthoryear{van Lieshout}{van
  Lieshout}{2011}]{Lieshout:11}
van Lieshout, M. N.~M. (2011).
\newblock A j–function for inhomogeneous point processes.
\newblock {\em Statistica Neerlandica\/}~{\em 65}, 183--201.

\bibitem[\protect\citeauthoryear{Wu}{Wu}{1995}]{Wu:95}
Wu, Z. (1995).
\newblock Compactly supported positive definite radial functions.
\newblock {\em Advances in Computational Mathematics\/}~{\em 4}, 283--292.

\bibitem[\protect\citeauthoryear{Yaglom}{Yaglom}{1987}]{Yaglom:87}
Yaglom, A.~M. (1987).
\newblock {\em Correlation Theory of Stationary and Related Random Functions}.
\newblock New York: Springer-Verlag.

\end{thebibliography}
\bibliographystyle{chicago}
\end{document}